\documentclass[11pt,letterpaper,reqno]{amsart}

\usepackage{amsmath,amssymb,amsthm,mathtools,mathrsfs}
\usepackage{enumitem}
\usepackage{microtype}
\usepackage{xcolor}
\usepackage{hyperref}
\usepackage{aliascnt}
\usepackage[nameinlink,capitalise,noabbrev]{cleveref}
\hypersetup{
  colorlinks=true,
  linkcolor=blue!55!black,
  citecolor=blue!55!black,
  urlcolor=blue!55!black
}

\newtheorem{theorem}{Theorem}[section]
\newaliascnt{lemma}{theorem}
\newtheorem{lemma}[lemma]{Lemma}
\aliascntresetthe{lemma}
\newaliascnt{proposition}{theorem}
\newtheorem{proposition}[proposition]{Proposition}
\aliascntresetthe{proposition}
\newaliascnt{corollary}{theorem}
\newtheorem{corollary}[corollary]{Corollary}
\aliascntresetthe{corollary}
\newaliascnt{claim}{theorem}

\aliascntresetthe{claim}
\newaliascnt{problem}{theorem}

\aliascntresetthe{problem}
\theoremstyle{definition}
\newaliascnt{definition}{theorem}

\aliascntresetthe{definition}
\newaliascnt{remark}{theorem}
\newtheorem{remark}[remark]{Remark}
\aliascntresetthe{remark}

\crefname{theorem}{theorem}{theorems}
\Crefname{theorem}{Theorem}{Theorems}
\crefname{lemma}{lemma}{lemmas}
\Crefname{lemma}{Lemma}{Lemmas}
\crefname{proposition}{proposition}{propositions}
\Crefname{proposition}{Proposition}{Propositions}
\crefname{corollary}{corollary}{corollaries}
\Crefname{corollary}{Corollary}{Corollaries}
\crefname{claim}{claim}{claims}
\Crefname{claim}{Claim}{Claims}
\crefname{problem}{problem}{problems}
\Crefname{problem}{Problem}{Problems}
\crefname{definition}{definition}{definitions}
\Crefname{definition}{Definition}{Definitions}
\crefname{remark}{remark}{remarks}
\Crefname{remark}{Remark}{Remarks}

\newcommand{\norm}[1]{\left\lVert#1\right\rVert}

\newcommand{\R}{\mathbb R}
\newcommand{\NN}{\mathbb N}
\newcommand{\one}{\mathbf 1}
\newcommand{\supp}{\operatorname{supp}}
\newcommand{\car}{\operatorname{car}}
\newcommand{\Ran}{\operatorname{Ran}}

\newcommand{\id}{\operatorname{id}}

\newcommand{\Iso}{\operatorname{Iso}}
\newcommand{\Mp}{\mathbf M_p}
\newcommand{\cA}{\mathcal A}

\newcommand{\cD}{\mathcal D}

\newcommand{\cK}{\mathcal K}

\newcommand{\M}{\mathcal M}
\newcommand{\N}{\mathcal N}
\newcommand{\A}{\mathcal A}
\newcommand{\B}{\mathcal B}
\newcommand{\Z}{Z}
\newcommand{\sa}{\mathrm{sa}}
\newcommand{\LS}{LS}
\newcommand{\Pf}{\mathcal P_f}
\newcommand{\Ftau}{\mathcal F(\tau)}

\newcommand{\DeltaM}[1]{\Delta^{\M}_{#1}}
\newcommand{\DeltaN}[1]{\Delta^{\N}_{#1}}

\title[Characterization of surjective isometries: the real case]
{CHARACTERIZATION OF SURJECTIVE ISOMETRIES ON\\
SYMMETRIC FUNCTION SPACES AND ITS\\
NONCOMMUTATIVE COUNTERPART: THE REAL CASE}

\author[T. Guo]{Tianbao Guo}
\address{Institute for Advanced Study in Mathematics of HIT,
Harbin 150001, China}
\email{tianbaoguo.edu@gmail.com}

\author[J. Huang]{Jinghao Huang}
\thanks{Corresponding author: Jinghao Huang.}
\address{Institute for Advanced Study in Mathematics of HIT,
Harbin 150001, China}
\email{jinghao.huang@hit.edu.cn}

\date{}

\begin{document}
\begin{abstract}
Let $(\Omega,\mu)$ and $(\Lambda,\lambda)$ be complete atomless
localizable semifinite measure spaces. Suppose that
$E(\Omega,\mu)$ and $F(\Lambda,\lambda)$ are real
rearrangement-invariant Banach function spaces with order-continuous
norms, in the Banach-lattice sense, and that neither norm is
proportional to the $L_2$-norm. Every surjective real-linear isometry
\[
   U:E(\Omega,\mu)\longrightarrow F(\Lambda,\lambda)
\]
has the form $Uf=w\,\Phi(f)$, where $w$ has full support and $\Phi$ is
induced by a complete measure-class Boolean isomorphism. Both factors
are uniquely determined by $U$.

Let $(\M,\tau)$ and $(\N,\nu)$ be atomless semifinite von Neumann
algebras, and let $E(\M,\tau)$ and $F(\N,\nu)$ be symmetric operator
spaces satisfying the same assumptions on their norms. Every
surjective real-linear isometry
\[
   V:E(\M,\tau)_{\sa}\longrightarrow F(\N,\nu)_{\sa}
\]
has the form $V(x)=hJ(x)$, where $J:\M\to\N$ is a normal surjective
Jordan $*$-isomorphism and $h\in\LS(\Z(\N))_{\sa}$ is central with full
support; again, the two factors are unique. We also identify the
bounded skew-Hermitian operators on the real self-adjoint part and
derive commutative and noncommutative isometric forms of Mityagin's
question.
\end{abstract}

\maketitle

\section{Introduction}
\label{sec:introduction}

We consider surjective real-linear isometries in two settings:
rearrangement-invariant function spaces over semifinite measure spaces,
and the self-adjoint parts of noncommutative symmetric spaces. Banach
and Lamperti determined the isometries of $\ell_p$ and $L_p$
\cite{Banach,Lamperti}; for complex symmetric function spaces, see
Lumer and Zaidenberg \cite{LumerSIP,Zaidenberg97}. Kalton and
Randrianantoanina settled the real problem on a finite atomless interval
\cite{KRnote,KR}. Their result cannot simply be recovered from the
complex theory, since a real-linear isometry need not admit an
isometric complexification.

The finite-interval proof of Kalton and Randrianantoanina represents
an operator by a measurable family of signed measures and shows that
almost every such measure is supported at one point. Its recurrence step uses the finite
total measure in an essential way. On an infinite measure space, a
kernel term may leave the finite set in which it started, so that the
same iteration is unavailable. The first theorem removes this
restriction. Order continuity on a localizable space is understood
below in the Banach-lattice sense, for decreasing nets.

\begin{theorem}
\label{thm:between-main}
Let $(\Omega,\mu)$ and $(\Lambda,\lambda)$ be complete atomless
localizable semifinite measure spaces. Let $E(\Omega,\mu)$ and
$F(\Lambda,\lambda)$ be real rearrangement-invariant Banach function
spaces with order-continuous norms, neither norm being proportional to
the $L_2$-norm. If
\[
   U:E(\Omega,\mu)\longrightarrow F(\Lambda,\lambda)
\]
is a surjective real-linear isometry, then
\[
   Uf=w\,\Phi(f),\qquad f\in E(\Omega,\mu),
\]
where $w\in L_0(\Lambda,\lambda)$ has full support and $\Phi$ is the
normal lattice and algebra isomorphism induced by a complete
measure-class Boolean isomorphism. The pair $(w,\Phi)$ is unique. On
standard measure spaces the formula may be written as
\[
   (Uf)(s)=w(s)f(\sigma(s)),
\]
where $\sigma$ is invertible and nonsingular modulo null sets.
\end{theorem}

The proof starts from the random-measure representation of Kalton and
Weis \cite{KaltonRep,Weis}. The $p$-variation estimate from
\cite[Section~6]{KR} is localized to sets of finite measure; this
forces the local kernels to be atomic and gives a bound independent of
the particular isometry. If two nonzero atoms persist on a set of
positive measure, finitely many terms from the kernels of an isometry
and its inverse produce nonnegative matrices with an expanding product.
Randomizing the intermediate variable controls the kernel of a
composition. Perron--Frobenius theory then turns the expansion into an
iteration that violates the uniform local $p$-variation bound. The
kernel is therefore supported at one point almost everywhere.
Conjugated sign changes and band projections handle isometries between
different spaces, and a separable reduction passes from the standard
$\sigma$-finite case to localizable measure spaces.

The second setting is noncommutative. Surjective isometries of
noncommutative $L_p$-spaces were studied by Yeadon and Sherman
\cite{Yeadon,Sherman05,Sherman06}; related results for symmetric
operator ideals and semifinite symmetric spaces appear in
\cite{Sourour,deJagerConradie,HS}. The full-space theorems in
\cite{HS,FGHS} are complex-linear and do not classify real-linear
isometries between self-adjoint parts.

The required real geometry is encoded by skew-Hermitian operators. A
bounded real-linear operator $H$ on a real Banach space $X$ is called
skew-Hermitian when $\exp(tH)$ is a surjective isometry for every
$t\in\mathbb R$.

\begin{theorem}
\label{thm:SH}
Let $(\M,\tau)$ be an atomless semifinite von Neumann algebra and let
$X=E(\M,\tau)_{\sa}$ have order-continuous, non-Hilbertian norm. If
$H:X\to X$ is bounded and skew-Hermitian, then there exists
$a=a^*\in\M$ such that
\[
   H(x)=i(xa-ax),\qquad x\in X.
\]
The implementing element is unique modulo $\Z(\M)_{\sa}$.
\end{theorem}

For complex symmetric operator spaces, the corresponding Hermitian
operators have the form $x\mapsto ax+xb$
\cite[Theorem~3.10]{HS}. Here the argument is carried out on the real
self-adjoint part itself. A trace identity for orthogonal finite
projections yields the two diagonal corner identities for $H(p)$ and
an operator-norm bound on the finite-support algebra. After
complexification one obtains a bounded Jordan $*$-derivation on
$C_0(\M,\tau)$, which is inner.

This description leads to the noncommutative representation theorem.

\begin{theorem}
\label{thm:nc-main}
Let $(\M,\tau)$ and $(\N,\nu)$ be atomless semifinite von Neumann
algebras. Let $E(\M,\tau)$ and $F(\N,\nu)$ be symmetric operator spaces
with order-continuous norms, neither norm being proportional to the
$L_2$-norm. If
\[
   V:E(\M,\tau)_{\sa}\longrightarrow F(\N,\nu)_{\sa}
\]
is a surjective real-linear isometry, then there exist a normal
surjective Jordan $*$-isomorphism $J:\M\to\N$ and an operator
$h\in\LS(\Z(\N))_{\sa}$ such that
\[
   s(h)=1,\qquad
   V(x)=hJ(x),\quad x\in E(\M,\tau)_{\sa}.
\]
The Jordan $*$-isomorphism $J$ and the central multiplier $h$ are
uniquely determined by $V$.
\end{theorem}

No positivity or disjointness-preserving hypothesis is placed on $V$.
Conjugating the commutator operators from
Theorem~\ref{thm:SH} identifies, through their commutants, the relevant
maximal abelian von Neumann subalgebras. On each of them,
Theorem~\ref{thm:between-main} shows that
\[
   p\longmapsto s(V(p)),\qquad p\in\Pf(\M),
\]
is orthogonally additive. The extension theorem of de Jager and
Conradie \cite{DJC} supplies the normal Jordan $*$-isomorphism, while
compatibility of the coefficients on finite projections produces the
central multiplier.

Using the restriction argument from \cite{FGHS}, the representation
theorems also settle the isometric part of Mityagin's question
\cite{Mityagin} in the present real settings. Apart
from the $L_p$ case, no real symmetric space on a finite interval is
surjectively isometric to one on the half-line or the real line; the
same obstruction holds for the self-adjoint noncommutative spaces.

Part~\ref{part:commutative} proves
Theorem~\ref{thm:between-main}. Part~\ref{part:noncommutative}
contains the proofs of Theorems~\ref{thm:SH} and~\ref{thm:nc-main}.
The Mityagin applications follow in the last section. The two appendices contain the local kernel estimates and the
recurrence argument that excludes more than one atom.

\part{Commutative rearrangement-invariant spaces}
\label{part:commutative}

\section{Preliminaries}
\label{sec:comm-prelim}

Throughout Part~\ref{part:commutative}, $I=(0,\alpha)$ with
$0<\alpha\le\infty$, equipped with Lebesgue measure $m$. All functions
are real valued and are identified modulo equality almost everywhere.
The notation and basic facts concerning rearrangement-invariant spaces
are taken from \cite{BS,LT,KR,FlemingJamison}.

For $f\in L^0(I)$, let
\[
   d_f(\lambda)=m\{t\in I:|f(t)|>\lambda\},
   \qquad \lambda\ge0,
\]
and
\[
   f^*(s)=\inf\{\lambda\ge0:d_f(\lambda)\le s\},
   \qquad s>0.
\]
Thus $f$ and $g$ are equimeasurable precisely when
$d_f=d_g$, or equivalently $f^*=g^*$.  We write
$f\prec\!\prec g$ when
\[
   \int_0^t f^*(s)\,ds
   \le
   \int_0^t g^*(s)\,ds,
   \qquad t>0.
\]

A Banach function space $X$ on $I$ is called
\emph{rearrangement-invariant} if equimeasurable functions have the
same norm. Banach function spaces are understood in the sense of
\cite[Chapter~1]{BS}; in particular, their norms have the Fatou
property.  The fundamental function of $X$ is
\[
   \varphi_X(t)=\|\chi_{(0,t)}\|_X,
   \qquad 0<t<\alpha,
\]
so that $\|\chi_A\|_X=\varphi_X(m(A))$ whenever $m(A)<\infty$.
For a measurable set $A\subset I$, set
\[
   X(A)=\{f\in X:\supp f\subseteq A\},
   \qquad
   P_Af=\chi_Af.
\]
In particular,
\[
   X_R=X(0,R),
   \qquad 0<R<\alpha.
\]
The norm is order continuous if $0\le f_n\le |f|$ and
$f_n\to0$ almost everywhere imply $\|f_n\|_X\to0$.  Under this
assumption, bounded simple functions with support of finite measure
are norm dense in $X$.

The K\"othe dual is
\[
   X^\times
   =
   \left\{
      g\in L^0(I):
      \int_I |fg|\,dm<\infty
      \text{ for every } f\in X
   \right\},
\]
with norm
\[
   \|g\|_{X^\times}
   =
   \sup_{\|f\|_X\le1}\int_I|fg|\,dm.
\]
If $X$ has order-continuous norm, then $X^*=X^\times$ isometrically.
For every $0<R<\alpha$, restriction to $(0,R)$ gives
\begin{equation}
   (X_R)^\times=X^\times(0,R).
   \label{eq:local-dual}
\end{equation}

\begin{lemma}
\label{lem:ri-hilbert}
Let $(\Omega,\mu)$ be an atomless semifinite measure space and let
$Z(\Omega,\mu)$ be a real rearrangement-invariant Banach function
space with order-continuous norm.  Then $Z$ is a Hilbert space if and
only if
\[
   Z=L_2(\Omega,\mu)
\]
as sets and there is a constant $c>0$ such that
\[
   \|f\|_Z=c\|f\|_2,
   \qquad f\in Z.
\]
\end{lemma}

\begin{proof}
Only the necessity requires proof.  Suppose that the norm of $Z$ is
induced by an inner product $(\cdot,\cdot)_Z$.  If $A$ and $B$ are
disjoint sets of finite measure, multiplication by
$\one-2\chi_A$ is a surjective isometry of $Z$.  It sends $\chi_A$ to
$-\chi_A$ and fixes $\chi_B$; hence
\[
   (\chi_A,\chi_B)_Z
   =(-\chi_A,\chi_B)_Z=0.
\]
Consequently, with $\varphi_Z(t)=\|\chi_A\|_Z$ for $\mu(A)=t$,
\[
   \varphi_Z(s+t)^2=\varphi_Z(s)^2+\varphi_Z(t)^2
\]
whenever $s,t<\infty$ and $s+t\le\mu(\Omega)$.  The function
$t\mapsto\varphi_Z(t)^2$ is increasing. The additive identity above
therefore forces
\[
   \varphi_Z(t)^2=c^2t
\]
for some $c>0$.  The characteristic functions of pairwise disjoint
sets are mutually orthogonal in $(\cdot,\cdot)_Z$.  Thus every simple
function $f=\sum_{k=1}^n\alpha_k\chi_{A_k}$ with support of finite
measure satisfies
\[
   \|f\|_Z^2
   =\sum_{k=1}^n|\alpha_k|^2\varphi_Z(\mu(A_k))^2
   =c^2\|f\|_2^2.
\]
Finite-support simple functions are dense both in $Z$ and in
$L_2(\Omega,\mu)$.  The two spaces are therefore the completions of
the same vector lattice under proportional norms, which proves the
assertion.
\end{proof}

\begin{lemma}
\label{lem:nonhilbert-scale}
Let $X$ have order-continuous norm. If the norm of $X$ is not
proportional to the $L_2$-norm, then $X_R$ is non-Hilbertian for some
$R>0$.
\end{lemma}

\begin{proof}
Assume that every $X_R$ is Hilbertian.  By
\Cref{lem:ri-hilbert}, for each $R>0$ there is a constant $c_R>0$ such
that
\[
   \|f\|_X=c_R\|f\|_2,
   \qquad f\in X_R.
\]
If $0<R<S$, testing this identity on a nonzero function supported in
$(0,R)$ gives $c_R=c_S$.  Hence $c_R=c$ is independent of $R$.
Bounded simple functions with support of finite measure are dense in
both spaces. Therefore,
\[
   X=L_2(0,\alpha),
   \qquad \|f\|_X=c\|f\|_2,
\]
contrary to the hypothesis.
\end{proof}

Two functions are disjoint, written $f\perp g$, when
$|f|\wedge|g|=0$ almost everywhere.  An operator is disjointness
preserving if it sends disjoint pairs to disjoint pairs.  Following
\cite{KR}, an operator $T:X\to Y$ is called \emph{elementary} if
\[
   Tf(s)=a(s)f(\sigma(s))
   \quad\text{a.e.}
\]
for a measurable weight $a$ and a nonsingular measurable map
$\sigma$.  If $T$ is invertible, $a$ may be taken nonzero almost
everywhere and $\sigma$ invertible modulo null sets.  Such operators
are disjointness preserving.

For the real Banach-space terminology we follow
\cite[Section~3]{KR}. Given a real Banach space $X$, put
\[
   \Pi(X)
   =
   \{(x,x^*)\in X\times X^*:
      \|x\|=\|x^*\|=1,\ x^*(x)=1\}.
\]
A bounded operator $A:X\to X$ is numerically positive if
\[
   x^*(Ax)\ge0
   \qquad ((x,x^*)\in\Pi(X)).
\]
For a projection $P$, this is equivalent to $\|I-P\|=1$.
A nonzero $u\in X$ is a \emph{Flinn element} if there is
$f\in X^*$ with $f(u)=1$ such that the rank-one projection
$f\otimes u$ is numerically positive.  Equivalently,
\[
   f(u)=1,
   \qquad
   f(x)x^*(u)\ge0
   \quad ((x,x^*)\in\Pi(X)).
\]
If $(u,f)$ is a Flinn pair and $U:X\to Y$ is a surjective real-linear
isometry, then
\[
   \bigl(Uu,(U^*)^{-1}f\bigr)
\]
is a Flinn pair in $Y$ \cite[Proposition~3.2]{KR}.  This invariance is
the real-geometric input used in the kernel argument below.

\section{Representing kernels}
\label{sec:representing-kernels}

The random-measure argument uses two finite-window properties. If a
finite interval $I$ is decomposed as
$I=I_1\sqcup I_2$ with $m(I_1)=m(I_2)=m(I)/2$, a
rearrangement-invariant space $Z(I)$ has \emph{property $(P)$} when
\[
   \|\chi_{I_1}\|_Z
   <\|\chi_{I_1}+t\chi_{I_2}\|_Z
   \qquad(t>0).
\]
This does not depend on the chosen equal-measure decomposition.  The
space $Z(I)$ has \emph{property $(P')$} when $Z(I)^\times$ has property
$(P)$.  By \cite[Lemma~5.2]{KR}, every rearrangement-invariant space on
a finite interval has at least one of these properties.

Let $Y=X$ in the norm-continuous case and let $Y=X^\times$, equipped
with $\sigma(X^\times,X)$, in the dual case.  Assume that for some
finite interval $I_0$ the localized space $Y(I_0)$ is non-Hilbertian
and has property $(P)$.  In the dual case all isometries below are
assumed weak-star continuous.

\begin{theorem}
\label{thm:at-most-one-atom}
For every surjective isometry $T:Y\to Y$, the locally finite signed
kernel constructed in \Cref{app:local-kernels} admits a measurable
atomic representation
\[
   \nu_s^T=\sum_{j\ge1}a_j^T(s)\delta_{\sigma_j^T(s)}
\]
for which
\[
   m\left\{s:
     \left|\{j\ge1:a_j^T(s)\ne0\}\right|\ge2
   \right\}=0.
\]
Thus almost every representing measure has at most one nonzero atom.
\end{theorem}

The proof is given in \Cref{app:local-kernels,sec:kernel-composition}.
The first appendix establishes the local $p$-variation estimate and the
atomic representation of the kernels. The second proves the estimate
for composed kernels and the recurrence argument used to exclude two
nonzero atoms.

\section{Isometries on rearrangement-invariant spaces}
\label{sec:commutative-isometries}

The kernel theorem covers the direct property-$(P)$ argument and,
after passage to the K\"othe dual, the property-$(P')$ case. Surjectivity rules out a zero kernel on a set of positive
measure.
\begin{theorem}
\label{thm:exactly-one-atom}
In both the norm-continuous property-$(P)$ case and the weak-star
continuous dual property-$(P)$ case, the representing kernel of every
surjective isometry $T$ has exactly one nonzero term almost everywhere.
Equivalently,
\[
 \left|\{j\ge1:a_j^T(s)\ne0\}\right|=1
 \qquad\text{for almost every }s.
\]
\end{theorem}

\begin{proof}
By \Cref{thm:at-most-one-atom},
\[
 \left|\{j\ge1:a_j^T(s)\ne0\}\right|\le1
 \qquad\text{for almost every }s.
\]
Surjectivity rules out vanishing on a set of positive measure.

Let
\[
 Z=\{s:\nu_s^T=0\}.
\]
Suppose that $m(Z)>0$. By semifiniteness, there exists a measurable
set $Z_0\subset Z$ such that
\[
 0<m(Z_0)<\infty.
\]
For every bounded function $f$ with support of finite measure, the kernel
representation gives
\[
 P_{Z_0}Tf=0.
\]
In the norm-continuous case, such functions are norm dense. In the
weak-star case, they are weak-star dense, and both $T$ and $P_{Z_0}$
are weak-star continuous. Hence in either case
\[
 P_{Z_0}T=0
\]
on the whole space.

Surjectivity then forces
\[
 P_{Z_0}g=0
 \qquad\text{for every }g\in Y.
\]
Taking $g=\chi_{Z_0}$ gives a contradiction, since
$m(Z_0)>0$. Therefore
\[
 \nu_s^T\ne0
 \qquad\text{for almost every }s.
\]
Together with the first inequality, this yields
\[
 \left|\{j\ge1:a_j^T(s)\ne0\}\right|=1
 \qquad\text{for almost every }s.
\]
\end{proof}

Choose the first nonzero term in the measurable atomic enumeration.
Then there are measurable $a$ and $\sigma$ such that
\begin{equation}
 \nu_s^T=a(s)\delta_{\sigma(s)},
 \qquad a(s)\ne0\quad\text{almost everywhere}.
 \label{eq:one-atom-kernel}
\end{equation}
The null-set property implies that $\sigma$ is nonsingular: if
$m(N)=0$, then
\[
 0=|\nu_s^T|(N)=|a(s)|\one_N(\sigma(s))
\]
for almost every $s$, so $m(\sigma^{-1}(N))=0$.

The kernel formula initially holds on bounded functions of finite
support.  In the norm-continuous case it extends by order
continuity.
\begin{lemma}
\label{lem:core-extension}
In the order-continuous case,
\begin{equation}
 Tf=a(f\circ\sigma)
 \label{eq:elementary-core}
\end{equation}
for every $f\in X$.
\end{lemma}

\begin{proof}
The kernel formula gives \eqref{eq:elementary-core} for bounded
functions with support of finite measure. For $f\in X$, choose such
functions $f_n$ with
\[
 \norm{f_n-f}_X\to0,
 \qquad
 f_n\to f \quad\text{a.e.}
\]
Since inverse images of null sets under $\sigma$ are null,
\[
 f_n\circ\sigma\to f\circ\sigma
 \quad\text{a.e.}
\]
Also $Tf_n\to Tf$ in $X$, so, after passing to a subsequence,
$Tf_n\to Tf$ almost everywhere. Passing to the limit in
\eqref{eq:elementary-core} proves the result.
\end{proof}

Applying the one-atom formula to both $T$ and $T^{-1}$ shows
that the two measurable maps are inverse to one another modulo null sets.
\begin{theorem}
\label{thm:direct-elementary}
Suppose that a non-Hilbert finite interval of $X$ has property $(P)$.
Then every surjective real-linear isometry $T:X\to X$ has the form
\[
 Tf=a(f\circ\sigma),
 \qquad f\in X,
\]
where $a\ne0$ almost everywhere and $\sigma$ is invertible modulo
null sets. Moreover, both $\sigma$ and its inverse send inverse images
of null sets to null sets.
\end{theorem}

\begin{proof}
By \Cref{thm:exactly-one-atom,lem:core-extension},
\[
 Tf=a(f\circ\sigma),
\]
where $a\ne0$ almost everywhere and inverse images of null sets under
$\sigma$ are null. The one-atom formula for $S=T^{-1}$ gives
\[
 Sg=b(g\circ\rho),
\]
with the analogous properties for $b$ and $\rho$.

For every finite interval $I$ with rational endpoints, the identities
$ST\chi_I=\chi_I$ and $TS\chi_I=\chi_I$ give
\[
 b(t)a(\rho(t))
 \chi_I(\sigma(\rho(t)))
 =\chi_I(t)
\]
and
\[
 a(s)b(\sigma(s))
 \chi_I(\rho(\sigma(s)))
 =\chi_I(s)
\]
almost everywhere. Since there are only countably many such intervals,
these identities hold simultaneously outside one null set. As these
intervals separate points, we obtain
\[
 \sigma(\rho(t))=t,
 \qquad
 \rho(\sigma(s))=s
\]
almost everywhere. Consequently,
\[
 b(t)a(\rho(t))=1,
 \qquad
 a(s)b(\sigma(s))=1
\]
almost everywhere. Thus $\rho$ and $\sigma$ are inverse transformations
modulo null sets.
\end{proof}

In the property-$(P')$ case, apply the same result to the K\"othe dual.
The required adjoint relation is

\begin{lemma}
\label{lem:preadjoint}
Let $A:X\to X$ be bounded and suppose that, for every bounded
finite-support $g\in X^\times$,
\[
 A^*g(t)=c(t)g(\tau(t)),
\]
where $c\ne0$ almost everywhere and $\tau$ is invertible modulo null
sets. Then
\[
 Af(y)=d(y)f(\tau^{-1}(y)),
 \qquad f\in X,
\]
where
\[
 d(y)=J_\tau(y)c(\tau^{-1}(y)),
 \qquad
 J_\tau=\frac{d(\tau_*m)}{dm}.
\]
Moreover, $d\ne0$ almost everywhere.
\end{lemma}

\begin{proof}
For $f\in X$ and every bounded $g\in X^\times$ with support of finite measure,
\[
 \begin{aligned}
 \int (Af)(y)g(y)\,dm(y)
 &=\int f(t)(A^*g)(t)\,dm(t)\\
 &=\int f(t)c(t)g(\tau(t))\,dm(t)\\
 &=\int J_\tau(y)c(\tau^{-1}(y))
       f(\tau^{-1}(y))g(y)\,dm(y).
 \end{aligned}
\]
Bounded finite-support functions in $X^\times$ separate points of
$X$, and hence
\[
 Af(y)=J_\tau(y)c(\tau^{-1}(y))f(\tau^{-1}(y))
\]
almost everywhere.

Since $\tau$ and $\tau^{-1}$ preserve null sets under inverse images,
$\tau_*m$ and $m$ have the same null sets. Hence $J_\tau>0$ almost
everywhere. Since $c\ne0$ almost everywhere, we obtain
$d\ne0$ almost everywhere.
\end{proof}

By \cite[Lemma~5.2]{KR}, a non-Hilbertian finite interval satisfies
$(P)$ or $(P')$. The two cases give the same weighted-composition
formula.
\begin{theorem}
\label{thm:main}
Let $X$ be a real rearrangement-invariant Banach function space on
$(0,\infty)$ with order-continuous norm, and suppose that $X$ is
non-Hilbertian. Then every surjective real-linear isometry
\[
 T:X\longrightarrow X
\]
has the form
\[
 Tf(s)=a(s)f(\sigma(s)),
 \qquad f\in X,
\]
where $a$ is measurable and nonzero almost everywhere, and
$\sigma:(0,\infty)\to(0,\infty)$ is an invertible measurable
transformation modulo null sets. Moreover, inverse images of null sets
under both $\sigma$ and $\sigma^{-1}$ are null.
\end{theorem}

\begin{proof}
By \Cref{lem:nonhilbert-scale}, there exists $R>0$ such that $X_R$ is
non-Hilbertian. By \cite[Lemma~5.2]{KR}, $X_R$ has either property
$(P)$ or property $(P')$.

If $X_R$ has property $(P)$, the conclusion follows from
\Cref{thm:direct-elementary}.

Suppose that $X_R$ has property $(P')$. By \eqref{eq:local-dual},
\[
 X^\times(0,R)=(X_R)^\times
\]
has property $(P)$. Moreover, $(X_R)^\times$ is non-Hilbertian, since
otherwise its K\"othe dual $X_R$ would be Hilbertian.

Set
\[
 U=(T^{-1})^*:X^\times\longrightarrow X^\times.
\]
Then $U$ and $U^{-1}$ are weak-star continuous surjective isometries.
Hence the weak-star versions of
\Cref{prop:atomic-kernel-estimates,thm:kernel-composition,thm:finite-matrix-construction,thm:at-most-one-atom}
apply to both operators. By \Cref{thm:exactly-one-atom},
\[
 Ug(t)=c(t)g(\tau(t))
\]
for every bounded $g\in X^\times$ with support of finite measure, where
$c\ne0$ almost everywhere.

Apply the one-atom conclusion to $U^{-1}$. The countable family of
intervals with rational endpoints then separates the two coordinate
maps, showing that $\tau$ is invertible modulo null sets and that both
$\tau$ and $\tau^{-1}$ preserve null sets under inverse images.

Since $U=(T^{-1})^*$, \Cref{lem:preadjoint} gives
\[
 T^{-1}f(y)=d(y)f(\tau^{-1}(y)),
 \qquad f\in X,
\]
with $d\ne0$ almost everywhere. Since $\tau$ is invertible, the inverse
operator $T$ is again of weighted-composition form. Thus
\[
 Tf(s)=a(s)f(\sigma(s))
\]
for a measurable $a\ne0$ almost everywhere and an invertible
measurable transformation $\sigma$, with the stated null-set
properties.
\end{proof}

\begin{remark}
The transformation $\sigma$ in \Cref{thm:main} need not preserve
Lebesgue measure.  It is enough that inverse images of null sets under
both $\sigma$ and $\sigma^{-1}$ are null.
\end{remark}

\subsection{Isometries between two spaces}
\label{sec:between}

For an isometry between two different spaces, measurable sign changes
determine the corresponding band projections.

Let $(\Omega,\mu)$ be an atomless standard $\sigma$-finite measure
space, let $\mathfrak A_\Omega$ be its measure algebra, and write
$\operatorname{Aut}(\mathfrak A_\Omega)$ for the group of complete
measure-class Boolean automorphisms.  If $Z$ is a real Banach function
space over $\Omega$, put
\[
 \mathfrak D_Z
 =\{M_\varepsilon:\varepsilon\in L_\infty(\Omega),\
                 \ \varepsilon^2=1\text{ a.e.}\}.
\]
Thus $\mathfrak D_Z$ is the abelian group of all measurable sign
changes.  Every element of $\mathfrak D_Z$ is an involutive
surjective isometry.

If $\alpha\in\operatorname{Aut}(\mathfrak A_\Omega)$, denote by
$\Phi_\alpha$ the induced normal lattice automorphism of
$L_0(\Omega)$.  An invertible elementary operator has the form
\[
 W=M_b\Phi_\alpha,
\]
where $b$ is measurable and nonzero almost everywhere.

The Boolean automorphism in an elementary representation is unique.
\begin{lemma}
\label{lem:induced-automorphism}
Let $Z$ be a real Banach function space on an atomless semifinite
measure space, and assume that every surjective real-linear
self-isometry of $Z$ is invertible elementary. For each $W\in\Iso(Z)$
there is a unique $\alpha_W\in\operatorname{Aut}(\mathfrak A_\Omega)$
such that
\[
 W=M_b\Phi_{\alpha_W}
\]
for some measurable $b$ with full support. The map
\[
 \pi_Z:\Iso(Z)\longrightarrow\operatorname{Aut}(\mathfrak A_\Omega),
 \qquad \pi_Z(W)=\alpha_W,
\]
is a group homomorphism, and $\mathfrak D_Z$ is a normal subgroup of
$\Iso(Z)$.
\end{lemma}

\begin{proof}
Write $W=M_b\Phi_\alpha$, where $b$ has full support.  If also
$W=M_c\Phi_\beta$, then for every set of finite measure $A$,
\[
 \supp W\chi_A=\alpha(A)=\beta(A).
\]
Order density of the finite-measure elements and Dedekind completeness
of the measure algebra
\cite[313K, 322B, 322E]{Fremlin3} force $\alpha=\beta$.
The composition formula
\[
 (M_b\Phi_\alpha)(M_c\Phi_\beta)
 =M_{b\Phi_\alpha(c)}\Phi_{\alpha\beta}
\]
shows that $\pi_Z$ is a homomorphism.

If $M_\varepsilon\in\mathfrak D_Z$, then
\[
 W M_\varepsilon W^{-1}
 =M_{\Phi_\alpha(\varepsilon)}.
\]
Since $\Phi_\alpha(\varepsilon)^2=1$, the right-hand side is again a
sign change.  Hence $\mathfrak D_Z\triangleleft\Iso(Z)$.
\end{proof}

The argument uses a standard fact about automorphisms of an atomless
measure algebra.

\begin{lemma}
\label{lem:no-normal-involutions}
Let $(\Omega,\mu)$ be atomless, standard, and $\sigma$-finite.  Let
$\mathcal G$ be an abelian subgroup of
$\operatorname{Aut}(\mathfrak A_\Omega)$ such that
\[
 \alpha^2=\id\qquad(\alpha\in\mathcal G)
\]
and such that $\mathcal G$ is normalized by every measure-preserving
Boolean automorphism.  Then
\[
 \mathcal G=\{\id\}.
\]
\end{lemma}

\begin{proof}
Suppose that $\alpha\in\mathcal G$ is nontrivial.  Choose
$b\in\mathfrak A_\Omega$ with $\alpha(b)\ne b$.  Replacing $b$ by
$\alpha(b)$ if necessary, the element
\[
 a=b\wedge\alpha(b)^c
\]
is nonzero.  Since $\alpha^2=\id$,
\[
 a\wedge\alpha(a)=0.
\]

By atomlessness and semifiniteness, choose pairwise disjoint elements
$c_0,c_1,c_2\le a$ such that
\[
0<m(c_0)=m(c_1)=m(c_2)<\infty.
\]
Choose measurable representatives $C_0,C_1,C_2$ of these elements.
Since the $C_i$ have the same finite measure and the underlying
measure space is standard and atomless, the isomorphism theorem for
atomless standard probability spaces
\cite[Theorem~9.2.2]{Bogachev} gives measure-preserving isomorphisms
\[
\theta_0:C_0\to C_1,\qquad
\theta_1:C_1\to C_2,\qquad
\theta_2:C_2\to C_0
\]
modulo null sets.  Patching the maps $\theta_i$ on $C_i$ and taking the identity outside
$C_0\cup C_1\cup C_2$ defines a measure-preserving automorphism
satisfying
\[
\tau(C_0)=C_1,\qquad
\tau(C_1)=C_2,\qquad
\tau(C_2)=C_0
\]
modulo null sets.

Put $c=c_0\vee c_1\vee c_2$.  Since $c\le a$ and
$a\wedge\alpha(a)=0$, one has $c\wedge\alpha(c)=0$; hence $\tau$ is
the identity on $\alpha(c)$.  By normality,
\[
 \beta:=\tau\alpha\tau^{-1}\in\mathcal G.
\]
A direct computation in the measure algebra gives
\[
 \alpha\beta(c_0)=c_2,
 \qquad
 \beta\alpha(c_0)=c_1.
\]
Indeed, $\tau^{-1}(c_0)=c_2$ and $\tau$ fixes $\alpha(c_2)$, whereas
$\tau^{-1}$ fixes $\alpha(c_0)$.  Thus $\alpha\beta\ne\beta\alpha$,
contrary to the abelianness of $\mathcal G$.  Therefore
$\mathcal G$ is trivial.
\end{proof}

Conjugation by $U$ preserves the sign-change groups.
\begin{proposition}
\label{prop:sign-characteristic}
Let $E(\Omega,\mu)$ and $F(\Lambda,\lambda)$ be real rearrangement-invariant Banach function
spaces on atomless standard $\sigma$-finite measure spaces.  Assume
that every surjective real-linear self-isometry of $E$ and every such
self-isometry of $F$ is invertible elementary.  If
\[
 U:E\longrightarrow F
\]
is a surjective real-linear isometry, then
\[
 U\mathfrak D_EU^{-1}=\mathfrak D_F.
\]
\end{proposition}

\begin{proof}
Set
\[
 \mathcal H=U\mathfrak D_EU^{-1}.
\]
Conjugation by $U$ identifies $\Iso(E)$ with $\Iso(F)$.  Since
$\mathfrak D_E$ is normal in $\Iso(E)$ by
\Cref{lem:induced-automorphism}, $\mathcal H$ is a normal subgroup of $\Iso(F)$.
It is abelian and every one of its elements is an involution.
Consequently,
\[
 \pi_F(\mathcal H)
 \subseteq\operatorname{Aut}(\mathfrak A_\Lambda)
\]
is an abelian group of involutions.

Every measure-preserving automorphism $\gamma$ of
$\mathfrak A_\Lambda$ induces the rearrangement isometry
$\Phi_\gamma\in\Iso(F)$.  Normality of $\mathcal H$ therefore shows
that $\pi_F(\mathcal H)$ is normalized by all measure-preserving
automorphisms.  By Lemma~\ref{lem:no-normal-involutions},
\[
 \pi_F(\mathcal H)=\{\id\}.
\]
Hence every $H\in\mathcal H$ is a multiplication operator.  Since
$H^2=I$ and the scalar field is real, its multiplier is equal to
$\pm1$ almost everywhere.  Thus
\[
 \mathcal H\subseteq\mathfrak D_F.
\]
Interchanging $E$ and $F$ and replacing $U$ by $U^{-1}$ gives
\[
 U^{-1}\mathfrak D_FU\subseteq\mathfrak D_E.
\]
Conjugating the latter inclusion by $U$ gives the reverse inclusion
$\mathfrak D_F\subseteq\mathcal H$.
\end{proof}

The sign change $I-2P_A$ determines the band projection $P_A$.
\begin{proposition}
\label{prop:band-conjugacy-between}
Under the assumptions of Proposition~\ref{prop:sign-characteristic}, there is a
unique complete Boolean isomorphism
\[
 \theta:\mathfrak A_\Omega\longrightarrow\mathfrak A_\Lambda
\]
such that
\begin{equation}
 U P_A U^{-1}=P_{\theta(A)}
 \qquad(A\in\mathfrak A_\Omega).
 \label{eq:between-band-conjugacy}
\end{equation}
Consequently,
\begin{equation}
 \supp(Uf)=\theta(\supp f),
 \qquad f\in E.
 \label{eq:support-transport}
\end{equation}
\end{proposition}

\begin{proof}
For $A\in\mathfrak A_\Omega$, put
\[
 R_A=I-2P_A\in\mathfrak D_E.
\]
By Proposition~\ref{prop:sign-characteristic}, there is a unique
$\theta(A)\in\mathfrak A_\Lambda$ such that
\[
 U R_AU^{-1}=I-2P_{\theta(A)}.
\]
This is equivalent to \eqref{eq:between-band-conjugacy}.  Applying
conjugation to
\[
 P_AP_B=P_{A\wedge B},
 \qquad I-P_A=P_{A^c},
\]
shows that $\theta$ preserves finite meets and complements.  The
reverse construction with $U^{-1}$ shows that it is bijective.

We verify completeness.  Let $A=\bigvee_{i\in I}A_i$ and put
$B=\bigvee_{i\in I}\theta(A_i)$.  Monotonicity gives
$B\le\theta(A)$.  If $\theta(A)\wedge B^c\ne0$, semifiniteness \cite[211F]{Fremlin2} yields
a nonzero element of finite measure
\[
 0<C\le\theta(A)\wedge B^c.
\]
Set $g=\chi_C\in F$ and $f=U^{-1}g$.  Since
$C\le\theta(A)$, equation \eqref{eq:between-band-conjugacy} gives
$P_Af=f$.  Since $C\wedge\theta(A_i)=0$ for every $i$, it also gives
$P_{A_i}f=0$ for every $i$.  But $A=\bigvee_iA_i$, so $f=0$, a
contradiction.  Hence $\theta(A)=B$, and $\theta$ is complete.

Since $P_{\supp f}f=f$,
\eqref{eq:between-band-conjugacy} gives
\[
 \supp(Uf)\le\theta(\supp f).
\]
Let $C=\supp(Uf)$ and $D=\theta^{-1}(C)$.  Since $P_CUf=Uf$,
\eqref{eq:between-band-conjugacy} and injectivity of $U$ imply
$P_Df=f$.  Thus $\supp f\le D$, and applying $\theta$ yields the
reverse inequality in \eqref{eq:support-transport}.
\end{proof}

Let
\[
 \Phi:L_0(\Omega,\mu)\longrightarrow L_0(\Lambda,\lambda)
\]
be the normal lattice and algebra isomorphism induced by $\theta$;
thus
\[
 \Phi(\chi_A)=\chi_{\theta(A)}.
\]

The band-projection identity extends to bounded measurable
multipliers.
\begin{lemma}
\label{lem:between-module}
For every $\varphi\in L_\infty(\Omega)$ and every $f\in E$,
\begin{equation}
 U(\varphi f)=\Phi(\varphi)Uf.
 \label{eq:between-module}
\end{equation}
\end{lemma}

\begin{proof}
For a simple function $\varphi=\sum_{k=1}^nc_k\chi_{A_k}$,
\eqref{eq:between-band-conjugacy} gives
\[
 U M_\varphi U^{-1}
 =\sum_{k=1}^nc_kP_{\theta(A_k)}
 =M_{\Phi(\varphi)}.
\]
Every bounded measurable function is a uniform limit of simple
functions \cite[Theorem~2.10, p.~47]{Folland}.  Since multiplication by $\varphi\in L_\infty$ satisfies
\[
 \norm{\varphi f}_X
 \le \norm{\varphi}_{\infty}\norm{f}_X
\]
by the lattice property of Banach function spaces
\cite[Chapter~1, \S1, p.~2]{BS}, passage to the uniform
limit proves \eqref{eq:between-module}.
\end{proof}

\begin{theorem}
\label{thm:between-transfer}
Let $(\Omega,\mu)$ and $(\Lambda,\lambda)$ be atomless standard
$\sigma$-finite measure spaces.  Let $E(\Omega,\mu)$ and
$F(\Lambda,\lambda)$ be real rearrangement-invariant Banach function spaces with
order-continuous norms.  Assume that every surjective real-linear
self-isometry of each of $E$ and $F$ is invertible elementary.  Then
every surjective real-linear isometry
\[
 U:E(\Omega,\mu)\longrightarrow F(\Lambda,\lambda)
\]
has a representation
\begin{equation}
 Uf=w\,\Phi(f),
 \qquad f\in E,
 \label{eq:between-elementary}
\end{equation}
where $\Phi:L_0(\Omega,\mu)\to L_0(\Lambda,\lambda)$ is the normal
lattice isomorphism induced by a complete measure-class Boolean
isomorphism and $w$ is measurable with full support.
\end{theorem}

\begin{proof}
Let $\theta$ and $\Phi$ be supplied by
Proposition~\ref{prop:band-conjugacy-between}.  Since $\Omega$ is $\sigma$-finite,
there exists $e\in E$ with
\[
 e>0\qquad\text{almost everywhere}.
\]
Choose a measurable partition
$\Omega=\bigsqcup_{n\ge1}A_n$ with $0<\mu(A_n)<\infty$, choose
$c_n>0$ so that
\[
 \sum_{n\ge1}c_n\norm{\chi_{A_n}}_E<\infty,
\]
and take $e=\sum_nc_n\chi_{A_n}$.

By \eqref{eq:support-transport}, both $Ue$ and $\Phi(e)$ have full
support.  Define
\[
 w=\frac{Ue}{\Phi(e)}.
\]
Then $w$ is measurable and nonzero almost everywhere.

Fix $f\in E$ and put $r=f/e$.  Let
\[
 r_n=(-n)\vee(r\wedge n).
\]
Then $r_n\in L_\infty(\Omega)$,
\[
 r_n e\longrightarrow f\quad\text{almost everywhere},
 \qquad |r_n e|\le|f|,
\]
and order continuity of the norm gives
\[
 \norm{r_n e-f}_E\longrightarrow0.
\]
By Lemma~\ref{lem:between-module},
\[
 U(r_n e)=\Phi(r_n)Ue
       =w\,\Phi(r_n)\Phi(e)
       =w\,\Phi(r_n e).
\]
The left-hand side converges to $Uf$ in $F$.  After passing to a
subsequence it converges almost everywhere.  Normality of $\Phi$
gives
\[
 \Phi(r_n e)\longrightarrow\Phi(f)
 \qquad\text{almost everywhere}.
\]
Hence \eqref{eq:between-elementary} follows.

\end{proof}

For standard $\sigma$-finite spaces, the self-isometry hypotheses
follow from \cite{KR} in finite measure and from \Cref{thm:main} in
infinite measure.
\begin{theorem}
\label{thm:between-standard}
Let $(\Omega,\mu)$ and $(\Lambda,\lambda)$ be atomless standard
$\sigma$-finite semifinite measure spaces.  Let $E(\Omega,\mu)$ and
$F(\Lambda,\lambda)$ be real rearrangement-invariant Banach function
spaces with order-continuous, non-Hilbertian norms.  If
\[
 U:E(\Omega,\mu)\longrightarrow F(\Lambda,\lambda)
\]
is a surjective real-linear isometry, then
\[
 Uf=w\,\Phi(f),
\]
where $w$ has full support and $\Phi$ is induced by a complete
measure-class Boolean isomorphism.  If the two spaces are represented
on standard intervals or half-lines, this can be written as
\[
 (Uf)(s)=w(s)f(\sigma(s))
\]
for an invertible nonsingular measurable transformation $\sigma$
modulo null sets.
\end{theorem}

\begin{proof}
We verify the two self-isometry hypotheses in
Theorem~\ref{thm:between-transfer}.  On an atomless finite standard measure
space, the required elementary representation of self-isometries is
the theorem of Kalton and Randrianantoanina \cite{KR}, after the
usual rescaling to a finite interval.  On an atomless infinite
standard $\sigma$-finite measure space, a measure-space isomorphism
identifies the space with the half-line, and the required conclusion
is \cref{thm:main}. These two verifications apply separately to $E$
and $F$, so Theorem~\ref{thm:between-transfer} applies.
\end{proof}

\subsection{Localizable semifinite measure spaces}
\label{sec:separable-reduction}

The standard $\sigma$-finite result extends to complete localizable
semifinite measure spaces by separable reduction.  Order
continuity is understood for decreasing nets in the Banach-lattice
sense.  The formulation is needed later for maximal abelian
subalgebras of semifinite von Neumann algebras.

Let $\mathscr S$ be a complete sub-$\sigma$-algebra on a measurable
set $A\subseteq\Omega$.  We write
\[
 E[\mathscr S]
 =
 \{f\in E(\Omega,\mu):
   \supp f\le A
   \text{ and } f \text{ is }\mathscr S\text{-measurable}\}.
\]
The notation $F[\mathscr T]$ has the analogous meaning on the target
space.

Each function has $\sigma$-finite support.
\begin{lemma}
\label{lem:sigma-finite-support}
Let $E(\Omega,\mu)$ be a rearrangement-invariant Banach function space
with order-continuous norm over an atomless semifinite measure space.
Then, for every $f\in E$,
\[
 \mu\{|f|>\varepsilon\}<\infty
 \qquad(\varepsilon>0).
\]
In particular, $\supp f$ is $\sigma$-finite, and bounded
simple functions with support of finite measure are norm dense in $E$.
\end{lemma}

\begin{proof}
Fix $\varepsilon>0$ and set
\[
 A=\{|f|>\varepsilon\}.
\]
Suppose that $\mu(A)=\infty$.  Let $\mathcal C$ be the family of
measurable subsets of finite measure of $A$, directed by inclusion, and
define
\[
 x_C=|f|\chi_{A\setminus C},
 \qquad C\in\mathcal C.
\]
By semifiniteness,
\[
 x_C\downarrow0
\]
in $L_0(\Omega)$.  Indeed, if the infimum had nonzero support, then
semifiniteness would provide a nonzero subset of finite measure of that
support, which would be contained in some $C\in\mathcal C$, a
contradiction.  Hence order continuity gives
\[
 \norm{x_C}_E\longrightarrow0.
\]

Choose a measurable set $B_0\subset A$ with
\[
 0<t:=\mu(B_0)<\infty.
\]
For every $C\in\mathcal C$, we have $\mu(A\setminus C)=\infty$.
By semifiniteness and atomlessness
\cite[215D]{Fremlin2}, there exists
$B_C\subset A\setminus C$ with $\mu(B_C)=t$.  Since
$|f|\ge\varepsilon$ on $A$,
\[
 \norm{x_C}_E
 \ge
 \varepsilon\norm{\chi_{B_C}}_E
 =
 \varepsilon\varphi_E(t)>0,
\]
contradicting $\norm{x_C}_E\to0$ where $\varphi_E$ is the fundamental function of $E$;
see \cite[Chapter~2, \S5]{BS}.  Thus
\[
 \mu\{|f|>\varepsilon\}<\infty
 \qquad(\varepsilon>0).
\]

Consequently,
\[
 \supp f
 =
 \bigvee_{n\ge1}\{|f|>1/n\}
\]
is $\sigma$-finite.

Define
\[
 f_n
 =
 \bigl((-n)\vee(f\wedge n)\bigr)
 \chi_{\{|f|>1/n\}}.
\]
Then each $f_n$ is bounded and has support of finite measure,
$f_n\to f$ almost everywhere, and $|f_n|\le|f|$.  By order continuity,
\[
 \norm{f_n-f}_E\longrightarrow0.
\]
Approximating each $f_n$ by simple functions on its support of finite
measure \cite[Chapter~1, \S3]{BS} proves that bounded simple functions
with support of finite measure are norm dense in $E$.
\end{proof}

A countable family is contained in an atomless, countably generated
$\sigma$-finite sub-$\sigma$-algebra.
\begin{lemma}
\label{lem:atomless-hull}
Let $\mathcal C\subset L_0(\Omega,\mu)$ be countable and suppose that
\[
 A=\bigvee_{f\in\mathcal C}\supp f
\]
is $\sigma$-finite.  There exists a complete, countably generated,
atomless sub-$\sigma$-algebra $\mathscr S$ on $A$ such that every
$f\in\mathcal C$ is $\mathscr S$-measurable.  The restricted measure
space $(A,\mathscr S,\mu)$ is standard and $\sigma$-finite modulo null
sets.
\end{lemma}

\begin{proof}
Choose a countable measurable partition
\[
 A=\bigsqcup_{n\ge1}A_n,
 \qquad \mu(A_n)<\infty.
\]
Let $\mathcal R_0$ be a countable Boolean algebra containing the sets
$A_n$ and all rational level sets of the members of $\mathcal C$.
Inductively, after $\mathcal R_k$ has been constructed, for every
nonnull $B\in\mathcal R_k$ of finite measure choose, using
atomlessness, a set $B^\flat\subset B$ such that
\[
 \mu(B^\flat)=\frac12\mu(B),
\]
and let $\mathcal R_{k+1}$ be the Boolean algebra generated by
$\mathcal R_k$ and all such $B^\flat$.  Every $\mathcal R_k$ is
countable.  Let $\mathscr S$ be the completion of
\[
 \sigma\left(\bigcup_{k\ge0}\mathcal R_k\right)
\]
inside $A$.

To verify atomlessness, let $H\in\mathscr S$ satisfy
$0<\mu(H)<\infty$.  Choose $B\in\bigcup_k\mathcal R_k$ such that
\[
 \mu(H\triangle B)<\frac18\mu(H).
\]
Then $0<\mu(B)<\infty$.  At a later stage, the set $B$ is split into
$B^\flat$ and $B\setminus B^\flat$, both of measure $\mu(B)/2$.  The
preceding approximation gives
\[
 \mu(H\cap B^\flat)>0,
 \qquad
 \mu(H\setminus B^\flat)>0.
\]
Thus $H$ is not an atom.  If $H$ has infinite measure, then
$H\cap A_n$ is nonnull for some $n$ and supplies a finite nonzero
part which can be split.  Hence $\mathscr S$ and its associated
measure algebra are atomless; the latter is also separable. By
the isomorphism theorem for separable atomless measure algebras
\cite[Theorem~9.3.4]{Bogachev}, $(A,\mathscr S,\mu)$ is standard
modulo null sets.
\end{proof}

The restricted space over this subalgebra remains
rearrangement-invariant, order continuous, and non-Hilbertian; its
measure algebra is standard.
\begin{lemma}
\label{lem:separable-function-subspace}
Let $\mathscr S$ be as in \Cref{lem:atomless-hull}.  Then
$E[\mathscr S]$ is a closed separable rearrangement-invariant Banach
function space over $(A,\mathscr S,\mu)$, and its norm is order
continuous.

If $\mathscr S_0\subseteq\mathscr S_1\subseteq\cdots$ are such
sub-$\sigma$-algebras and
\[
 \mathscr S_\infty
 =\sigma\left(\bigcup_{n\ge0}\mathscr S_n\right)
\]
up to completion, then
\[
 \overline{\bigcup_{n\ge0}E[\mathscr S_n]}^{\,\norm{\cdot}_E}
 =E[\mathscr S_\infty].
\]
Moreover, if every $\mathscr S_n$ is atomless, then so is
$\mathscr S_\infty$.
\end{lemma}

\begin{proof}
Suppose that $f_n\in E[\mathscr S]$ and $f_n\to f$ in $E$. Then
$f_n\to f$ locally in measure
\cite[Chapter~1, \S1]{BS}. Since $A$ is
$\sigma$-finite, a diagonal application of
\cite[Theorem~2.30, pp.~61--62]{Folland} gives a subsequence
converging almost everywhere on $A$. Hence $f$ is
$\mathscr S$-measurable, and therefore $E[\mathscr S]$ is closed.

To prove separability, take a countable algebra generating
$\mathscr S$ modulo null sets. Simple functions with rational
coefficients and sets from this algebra form a countable family.
By truncating a function, restricting it to sets of finite measure,
and then approximating it by such simple functions, order continuity
\cite[Chapter~1, \S3]{BS} shows that this family is norm
dense in $E[\mathscr S]$. Thus $E[\mathscr S]$ is separable.
Rearrangement invariance and order continuity are inherited from $E$.

For the density assertion, simple functions with support of finite
measure are dense in $E[\mathscr S_\infty]$.  If $B\in\mathscr S_\infty$ has finite
measure, the algebra $\bigcup_n\mathscr S_n$ contains sets $B_k$ with
\[
 \mu(B_k\triangle B)\longrightarrow0.
\]
Since order continuity implies $\varphi_E(t)\to0$ as $t\downarrow0$,
\[
 \norm{\chi_{B_k}-\chi_B}_E
 =\varphi_E(\mu(B_k\triangle B))\longrightarrow0.
\]
This approximation gives density first for simple functions and then
for all of $E[\mathscr S_\infty]$.

To prove atomlessness, let $H\in\mathscr S_\infty$ satisfy
$0<\mu(H)<\infty$. By the monotone class lemma
\cite[Theorem~2.35, p.~66]{Folland}, choose
$B\in\mathscr S_n$ for some $n$ such that
\[
 \mu(H\triangle B)<\frac18\mu(H).
\]
Since $\mathscr S_n$ is atomless, write
\[
 B=B_1\sqcup B_2,
 \qquad
 \mu(B_1)=\mu(B_2)=\frac12\mu(B).
\]
Then
\[
 \mu(H\cap B_i)>0,
 \qquad i=1,2,
\]
so $H$ is not an atom. If $\mu(H)=\infty$, choose by
$\sigma$-finiteness a measurable subset of $H$ of finite positive
measure and apply the finite-measure case to that subset.
\end{proof}

Alternating the construction in the domain and range gives standard
$\sigma$-finite subspaces which are mapped onto one another.
\begin{proposition}
\label{prop:invariant-separable-subspaces}
Let $(\Omega,\mu)$ and $(\Lambda,\lambda)$ be complete atomless
localizable semifinite measure spaces. Let $E(\Omega,\mu)$ and
$F(\Lambda,\lambda)$ be rearrangement-invariant Banach function spaces
with order-continuous, non-Hilbertian norms, and let
\[
 U:E(\Omega,\mu)\longrightarrow F(\Lambda,\lambda)
\]
be a surjective real-linear isometry.

Given countable sets $\mathcal C\subset E$ and $\mathcal D\subset F$,
there exist complete atomless $\sigma$-finite sub-$\sigma$-algebras
$\mathscr S_\infty$ and $\mathscr T_\infty$, countably generated
modulo null sets, such that
\[
 \mathcal C\subset E[\mathscr S_\infty],
 \qquad
 \mathcal D\subset F[\mathscr T_\infty],
\]
and
\[
 U\bigl(E[\mathscr S_\infty]\bigr)
 =F[\mathscr T_\infty].
\]
The corresponding measure spaces are standard modulo null sets, and
both restricted spaces are non-Hilbertian.
\end{proposition}

\begin{proof}
Since simple functions with support of finite measure are dense and
the norms of $E$ and $F$ are non-Hilbertian, we may choose
\[
 u_1,u_2\in E,
 \qquad
 v_1,v_2\in F
\]
with supports of finite measure such that the parallelogram identity
fails for each pair.

By \Cref{lem:sigma-finite-support}, the union of the supports of the
countable family
\[
 \mathcal C\cup\{u_1,u_2\}
\]
is $\sigma$-finite. Apply \Cref{lem:atomless-hull} to obtain
$\mathscr S_0$. By \Cref{lem:separable-function-subspace},
$E[\mathscr S_0]$ is separable; choose a countable dense set
\[
 D_0\subset E[\mathscr S_0].
\]
Apply \Cref{lem:atomless-hull} on the target side to
\[
 U(D_0)\cup\mathcal D\cup\{v_1,v_2\}
\]
and obtain $\mathscr T_0$. Since $F[\mathscr T_0]$ is closed and
contains $U(D_0)$,
\[
 U(E[\mathscr S_0])\subseteq F[\mathscr T_0].
\]

Continue alternately on the two sides. Suppose that
$\mathscr S_n$ and $\mathscr T_n$ have been constructed. Choose a
countable dense set
\[
 C_n\subset F[\mathscr T_n].
\]
Since $\mathscr S_n$ is $\sigma$-finite and countably generated
modulo null sets, choose a countable family $\mathcal A_n\subset
\mathscr S_n$ of sets of finite measure which generates $\mathscr S_n$
modulo null sets. Apply \Cref{lem:atomless-hull} to
\[
 U^{-1}(C_n)\cup\{\chi_A:A\in\mathcal A_n\}.
\]
The characteristic functions belong to $E$ because their supports have
finite measure. By construction, $\mathscr S_{n+1}$ contains every member
of $\mathcal A_n$ modulo null sets, and hence
$\mathscr S_n\subseteq\mathscr S_{n+1}$ modulo null sets. Since
$E[\mathscr S_{n+1}]$ is closed,
\[
 U^{-1}(F[\mathscr T_n])
 \subseteq E[\mathscr S_{n+1}].
\]

Next choose a countable dense set
\[
 D_{n+1}\subset E[\mathscr S_{n+1}]
\]
and choose a countable family $\mathcal B_n\subset\mathscr T_n$ of
sets of finite measure which generates $\mathscr T_n$ modulo null
sets. Applying \Cref{lem:atomless-hull} to
\[
 U(D_{n+1})\cup\{\chi_B:B\in\mathcal B_n\}
\]
gives $\mathscr T_{n+1}$ containing $\mathcal B_n$ modulo null sets.
Thus $\mathscr T_n\subseteq\mathscr T_{n+1}$ modulo null sets and
\[
 U(E[\mathscr S_{n+1}])
 \subseteq F[\mathscr T_{n+1}].
\]

Let
\[
 \mathscr S_\infty
 =
 \overline{\sigma\left(\bigcup_{n\ge0}\mathscr S_n\right)},
 \qquad
 \mathscr T_\infty
 =
 \overline{\sigma\left(\bigcup_{n\ge0}\mathscr T_n\right)},
\]
where the bars denote completion. These $\sigma$-algebras are
$\sigma$-finite and countably generated modulo null sets. By
\Cref{lem:separable-function-subspace}, they are atomless and
\[
 E[\mathscr S_\infty]
 =
 \overline{\bigcup_{n\ge0}E[\mathscr S_n]},
 \qquad
 F[\mathscr T_\infty]
 =
 \overline{\bigcup_{n\ge0}F[\mathscr T_n]}.
\]
The inclusions obtained at each stage therefore imply
\[
 U(E[\mathscr S_\infty])
 \subseteq F[\mathscr T_\infty]
\]
and
\[
 U^{-1}(F[\mathscr T_\infty])
 \subseteq E[\mathscr S_\infty].
\]
Hence
\[
 U(E[\mathscr S_\infty])
 =F[\mathscr T_\infty].
\]

The associated measure spaces are standard modulo null sets by
\Cref{lem:atomless-hull}. Moreover,
$u_1,u_2\in E[\mathscr S_\infty]$ and
$v_1,v_2\in F[\mathscr T_\infty]$, and the parallelogram identity
fails for these pairs. Thus neither restricted space is Hilbertian.
\end{proof}

\begin{theorem}
\label{thm:general-DP}
Under the hypotheses of \Cref{prop:invariant-separable-subspaces},
\[
 f\perp g
 \quad\Longleftrightarrow\quad
 Uf\perp Ug,
 \qquad f,g\in E.
\]
\end{theorem}

\begin{proof}
Let $f,g\in E$ with $f\perp g$.  Apply
\cref{prop:invariant-separable-subspaces} with $\mathcal C=\{f,g\}$ and
$\mathcal D=\varnothing$.  The restriction of $U$ is a surjective
isometry between non-Hilbertian rearrangement-invariant spaces over
atomless standard $\sigma$-finite measure spaces.  Therefore
\cref{thm:between-standard} applies and gives $Uf\perp Ug$. Replacing
$U$ by $U^{-1}$ gives the converse implication.
\end{proof}

\begin{proposition}
\label{prop:band-conjugacy-general}
Under the hypotheses of \Cref{thm:general-DP}, there exists a unique
complete Boolean isomorphism
\[
 \theta:\mathfrak A_\Omega\longrightarrow\mathfrak A_\Lambda
\]
such that
\[
 UP_AU^{-1}=P_{\theta(A)},
 \qquad A\in\mathfrak A_\Omega.
\]
Moreover,
\[
 \supp(Uf)=\theta(\supp f),
 \qquad f\in E.
\]
\end{proposition}

\begin{proof}
Fix $A\in\mathfrak A_\Omega$ and set
\[
 Y_A=U(E(A)),
 \qquad
 Y_{A^c}=U(E(A^c)).
\]
By \Cref{thm:general-DP}, every element of $Y_A$ is disjoint from
every element of $Y_{A^c}$. Since the underlying measure spaces are
localizable, their measure algebras are complete
\cite[322B]{Fremlin3}. Hence the elements
\[
 C=\bigvee_{y\in Y_A}\supp y,
 \qquad
 D=\bigvee_{z\in Y_{A^c}}\supp z
\]
are well defined, and
\[
 C\wedge D=0.
\]

In fact,
\[
 C\vee D=1.
\]
Otherwise, by semifiniteness
\cite[211F]{Fremlin2}, there exists
\[
 0<H\le(C\vee D)^c,
 \qquad
 0<\lambda(H)<\infty.
\]
Then $\chi_H\in F$ by the standard Banach-function-space properties
\cite[Chapter~1, \S1]{BS}. Since
\[
 F=Y_A+Y_{A^c},
\]
we may write
\[
 \chi_H=y+z,
 \qquad
 y\in Y_A,\quad z\in Y_{A^c}.
\]
But $\supp y\le C$ and $\supp z\le D$, whereas
$H\le(C\vee D)^c$. Thus both $y$ and $z$ vanish on $H$, contradicting
$\chi_H=1$ on $H$. Therefore
\[
 D=C^c.
\]

Moreover,
\[
 Y_A=F(C).
\]
The inclusion $Y_A\subseteq F(C)$ follows from the definition of $C$.
Conversely, let $h\in F(C)$. Since
$F=Y_A+Y_{A^c}$, write
\[
 h=y+z,
 \qquad
 y\in Y_A,\quad z\in Y_{A^c}.
\]
Now $\supp y\le C$ and $\supp z\le D=C^c$. Since both $h$ and $y$
are supported in $C$, so is $z=h-y$. Hence
\[
 \supp z\le C\wedge C^c=0,
\]
and therefore $z=0$. Thus $h=y\in Y_A$, proving
\[
 Y_A=F(C).
\]

It follows that
\[
 U(E(A))=F(C),
\]
and hence
\[
 UP_AU^{-1}=P_C.
\]
Define
\[
 \theta(A)=C.
\]

Performing the construction for $U^{-1}$ defines
\[
 \psi:\mathfrak A_\Lambda\longrightarrow\mathfrak A_\Omega.
\]
The identities for the corresponding band projections imply
\[
 \psi(\theta(A))=A,
 \qquad
 \theta(\psi(B))=B.
\]
Thus $\theta$ is bijective.

Moreover,
\[
 P_{A\wedge B}=P_AP_B,
 \qquad
 P_{A^c}=I-P_A.
\]
Conjugating these identities by $U$ gives
\[
 \theta(A\wedge B)
 =\theta(A)\wedge\theta(B),
 \qquad
 \theta(A^c)=\theta(A)^c.
\]
Thus $\theta$ is a Boolean isomorphism.

To see that it is complete, let $(A_i)_{i\in I}$ be any family and put
\[
 A=\bigvee_{i\in I}A_i.
\]
Then $\theta(A)$ is an upper bound for every $\theta(A_i)$. Conversely,
if $B$ is any upper bound for all $\theta(A_i)$, then
$\theta^{-1}(B)$ is an upper bound for all $A_i$, and hence
\[
 A\le\theta^{-1}(B).
\]
Therefore
\[
 \theta(A)\le B.
\]
It follows that
\[
 \theta\left(\bigvee_{i\in I}A_i\right)
 =
 \bigvee_{i\in I}\theta(A_i),
\]
so $\theta$ is complete.

The isomorphism is unique, since the identity
\[
 UP_AU^{-1}=P_{\theta(A)}
\]
determines the band projection $P_{\theta(A)}$, and hence
$\theta(A)$, uniquely.

Let $A=\supp f$. Since $P_Af=f$,
\[
 P_{\theta(A)}Uf
 =
 UP_Af
 =
 Uf.
\]
Hence
\[
 \supp(Uf)\le\theta(\supp f).
\]
The corresponding support inequality for $U^{-1}$ is
\[
 \supp f
 \le
 \theta^{-1}(\supp(Uf)),
\]
and therefore
\[
 \theta(\supp f)\le\supp(Uf).
\]
Thus
\[
 \supp(Uf)=\theta(\supp f).
\]
\end{proof}

Let
\[
 \Phi:L_0(\Omega,\mu)\longrightarrow L_0(\Lambda,\lambda)
\]
be the normal lattice and algebra isomorphism induced by the complete
Boolean isomorphism $\theta$.

The band-projection identities extend to bounded measurable
multipliers.
\begin{lemma}
\label{lem:global-module}
For every $\varphi\in L_\infty(\Omega)$ and every $f\in E$,
\[
 U(\varphi f)=\Phi(\varphi)Uf.
\]
\end{lemma}

\begin{proof}
The assertion holds for characteristic functions by
\Cref{prop:band-conjugacy-general}, hence for simple multipliers by
linearity.  Uniform approximation of a bounded measurable function by
simple functions and the estimate
\[
 \norm{M_\varphi}_{E\to E}\le\norm{\varphi}_\infty
\]
give the general case.
\end{proof}

\begin{proof}[Proof of \Cref{thm:between-main}]
By \Cref{lem:ri-hilbert}, the assumptions imply that both $E$ and $F$
are non-Hilbertian. Hence the hypotheses of
\Cref{prop:invariant-separable-subspaces,thm:general-DP} are satisfied.
By
\Cref{thm:general-DP,prop:band-conjugacy-general,lem:global-module},
only the multiplier $w$ remains to be constructed.

Choose a maximal family $(A_\gamma)_{\gamma\in\Gamma}$ of pairwise
disjoint measurable sets such that
\[
 0<\mu(A_\gamma)<\infty.
\]
By semifiniteness \cite[211F]{Fremlin2} and maximality,
\[
 \bigvee_{\gamma\in\Gamma}A_\gamma=1.
\]
Since $\theta$ is complete, the sets $\theta(A_\gamma)$ are pairwise
disjoint and
\[
 \bigvee_{\gamma\in\Gamma}\theta(A_\gamma)=1.
\]

For each $\gamma$, put
\[
 h_\gamma=U\chi_{A_\gamma}.
\]
By \Cref{prop:band-conjugacy-general},
\[
 \supp h_\gamma=\theta(A_\gamma).
\]
Since $(\Lambda,\lambda)$ is localizable,
\cite[241G]{Fremlin2} gives a function
$w\in L_0(\Lambda,\lambda)$ satisfying
\[
 w\chi_{\theta(A_\gamma)}
 =U\chi_{A_\gamma},
 \qquad \gamma\in\Gamma.
\]
Since $\bigvee_\gamma\theta(A_\gamma)=1$, the support identity gives
\[
 \supp w=1.
\]

Suppose first that $f\in E$ is bounded and
$\supp f\le A_\gamma$. Then
\[
 f=f\chi_{A_\gamma},
\]
and \Cref{lem:global-module} gives
\[
 \begin{aligned}
 Uf
 &=U(f\chi_{A_\gamma})\\
 &=\Phi(f)\,U\chi_{A_\gamma}\\
 &=w\,\Phi(f).
 \end{aligned}
\]

Now let $f\in E$ be arbitrary with $\supp f\le A_\gamma$, and set
\[
 f_n=(-n)\vee(f\wedge n).
\]
Then
\[
 |f_n-f|\downarrow0,
\]
so order continuity gives
\[
 \norm{f_n-f}_E\longrightarrow0.
\]
Hence
\[
 Uf_n\longrightarrow Uf
 \qquad\text{in }F.
\]
Moreover, since $\Phi$ is a lattice isomorphism,
\[
 \Phi(f_n)\longrightarrow\Phi(f)
 \qquad\text{almost everywhere}.
\]
By \Cref{lem:sigma-finite-support}, the functions involved are
supported on a common $\sigma$-finite measurable set. After passing
to a subsequence, a diagonal application of
\cite[Theorem~2.30, pp.~61--62]{Folland} gives
\[
 Uf_n\longrightarrow Uf
 \qquad\text{almost everywhere}.
\]
Since
\[
 Uf_n=w\Phi(f_n),
\]
passing to the limit yields
\[
 Uf=w\Phi(f).
\]

Let $f\in E$ be arbitrary. For every $\gamma$,
\[
 \begin{aligned}
 P_{\theta(A_\gamma)}Uf
 &=UP_{A_\gamma}f\\
 &=w\,\Phi(P_{A_\gamma}f)\\
 &=P_{\theta(A_\gamma)}\,w\Phi(f).
 \end{aligned}
\]
Since
\[
 \bigvee_{\gamma\in\Gamma}\theta(A_\gamma)=1,
\]
these identities imply
\[
 Uf=w\Phi(f)
\]
almost everywhere on $\Lambda$.

For uniqueness, suppose
\[
   U=M_w\Phi=M_{\widetilde w}\widetilde\Phi
\]
with both multipliers of full support, and let $\theta$ and
$\widetilde\theta$ be the complete Boolean isomorphisms inducing
$\Phi$ and $\widetilde\Phi$.  For every set $A$ of finite measure,
\[
   \theta(A)=\supp(U\chi_A)=\widetilde\theta(A).
\]
Finite-measure elements are order dense in a semifinite measure
algebra, and both Boolean maps are complete; hence
$\theta=\widetilde\theta$ and therefore $\Phi=\widetilde\Phi$.  For
the maximal family $(A_\gamma)$ used above,
\[
   (w-\widetilde w)\chi_{\theta(A_\gamma)}=0
   \qquad(\gamma\in\Gamma).
\]
Since $\bigvee_\gamma\theta(A_\gamma)=1$, it follows that
$w=\widetilde w$ almost everywhere.

\end{proof}

\part{Semifinite noncommutative symmetric operator spaces}
\label{part:noncommutative}

\section{Preliminaries}
\label{sec:prelim-semifinite}

Throughout Part~\ref{part:noncommutative}, $(\M,\tau)$ denotes a
semifinite von Neumann
algebra with a faithful normal semifinite trace. We write
$\mathcal P(\M)$ for its projections and
\[
   \Pf(\M)=\{p\in\mathcal P(\M):\tau(p)<\infty\}.
\]
Write $S(\M,\tau)$ for the algebra of $\tau$-measurable operators
and $\LS(\M)$ for the algebra of locally measurable operators
affiliated with $\M$; see
\cite{Segal,FackKosaki,KadisonRingrose2,BerChilinSukochev}. In
particular, $\LS(\Z(\M))$ denotes the locally measurable operators
affiliated with the center.

\subsection{Measurable operators and symmetric spaces}

For $x\in S(\M,\tau)$, write $x=u|x|$. Its right and left supports
are
\[
   r(x)=u^*u,\qquad l(x)=uu^*,
\]
and we put
\[
   s(x)=l(x)\vee r(x).
\]
For $x=x^*$ this agrees with
\[
   s(x)=1-e^{|x|}(\{0\}).
\]
The central support of a projection $p$ is denoted by $z(p)$.
Following Fack and Kosaki \cite{FackKosaki}, define
\[
   d_x(\lambda)=\tau\bigl(e^{|x|}(\lambda,\infty)\bigr),
   \qquad
   \mu_t(x)=\inf\{\lambda\ge0:d_x(\lambda)\le t\}.
\]
We write $\mu(x)$ for $t\mapsto\mu_t(x)$ and use
\[
   \mu(x)=\mu(|x|)=\mu(x^*),
   \qquad
   \mu_t(axb)\le
   \|a\|_\infty\|b\|_\infty\mu_t(x)
\]
for $a,b\in\M$.

Set
\[
   I_\tau=
   \begin{cases}
      (0,\tau(1)),&\tau(1)<\infty,\\
      (0,\infty),&\tau(1)=\infty.
   \end{cases}
\]
If $E$ is a symmetric Banach function space on $I_\tau$, then
\[
   E(\M,\tau)
   =\{x\in S(\M,\tau):\mu(x)\in E\},
   \qquad
   \|x\|_{E(\M,\tau)}=\|\mu(x)\|_E.
\]
This construction and its basic properties are recalled in
\cite{Ovcinnikov70,Ovcinnikov71,DoddsDoddsdePagter,
KaltonSukochev,HS}. In particular,
\[
   \|axb\|_E\le
   \|a\|_\infty\|x\|_E\|b\|_\infty,
   \qquad
   \mu(uxv)=\mu(x)
\]
for $a,b\in\M$ and unitaries $u,v\in\M$. We regard
$E(\M,\tau)_{\sa}$ as a real Banach space.

\subsection{Order continuity and duality}

The norm of $E$ is order continuous if
\[
   0\le f_n\le |f|,\qquad f_n\to0\ \text{a.e.}
\]
imply $\|f_n\|_E\to0$. Put
\[
   \Ftau=\{x\in\M:\tau(s(x))<\infty\}.
\]
We also write
\[
   C_0(\M,\tau)=\overline{\Ftau}^{\,\|\cdot\|_\infty},
\]
for the $C^*$-algebra of $\tau$-compact operators.
If $E$ has order-continuous norm, then
\[
   \overline{\Ftau_{\sa}}^{\,\|\cdot\|_E}
   =E(\M,\tau)_{\sa};
\]
see \cite{DoddsDePagter,HS}. If $e\in\Pf(\M)$, the symmetric norm on
the finite corner $e\M e$ is the norm inherited from the ambient
space; see \cite{FackKosaki}.

The K\"othe dual of $E$ is
\[
   E^\times
   =\left\{g:\int_{I_\tau}|fg|<\infty
     \text{ for every }f\in E\right\},
\]
with its usual norm. The associated operator space is
\[
   E^\times(\M,\tau)
   =\{y\in S(\M,\tau):\mu(y)\in E^\times\}.
\]
For $x\in E(\M,\tau)$ and $y\in E^\times(\M,\tau)$,
\[
   |\tau(xy)|\le\|x\|_E\|y\|_{E^\times}.
\]
When $E$ has order-continuous norm, the trace pairing gives the real
isometric identification
\[
   \bigl(E(\M,\tau)_{\sa}\bigr)^*
   \cong E^\times(\M,\tau)_{\sa},
   \qquad
   \langle x,y\rangle_\tau=\tau(xy);
\]
see \cite{DoddsDePagter,HS}.

\subsection{Skew-Hermitian operators}

For a real Banach space $X$, set
\[
   \Pi(X)=
   \{(x,\phi)\in X\times X^*:
      \|x\|=\|\phi\|=1,\ \phi(x)=1\}.
\]
A bounded real-linear operator $H:X\to X$ is called
\emph{skew-Hermitian} if
\[
   \phi(Hx)=0
   \qquad ((x,\phi)\in\Pi(X)).
\]
Equivalently, $\exp(tH)$ is a surjective real-linear isometry for every
$t\in\mathbb R$; see \cite{LumerSIP,AC}.

For $a=a^*\in\M$, the operator
\[
   \delta_a(x)=i(xa-ax),
   \qquad x\in E(\M,\tau)_{\sa},
\]
is skew-Hermitian because
\[
   \exp(t\delta_a)(x)=e^{-ita}xe^{ita}.
\]
For later use we write
\[
   \Delta_a^{\M}=\delta_a,
   \qquad
   \Delta_b^{\N}(y)=i(yb-by)
\]
for $a=a^*\in\M$, $b=b^*\in\N$, and self-adjoint elements in the
corresponding symmetric spaces.  Thus the macros $\DeltaM{a}$ and
$\DeltaN{b}$ always denote these commutator operators.  Also, if
$V:X\to Y$ is a surjective real-linear isometry, then
$VHV^{-1}$ is skew-Hermitian whenever $H$ is.

\subsection{Atomlessness and Jordan maps}

A von Neumann algebra is atomless if it has no nonzero minimal
projection. Thus, if $p\in\Pf(\M)$ and
$0\le\alpha\le\tau(p)$, there is $q\le p$ with $\tau(q)=\alpha$.

A bijective linear map $J:\M\to\N$ is a Jordan $*$-isomorphism if
\[
   J(x^*)=J(x)^*,
   \qquad
   J(x\circ y)=J(x)\circ J(y),
   \qquad
   x\circ y=\frac12(xy+yx).
\]
Equivalently, $J(x^2)=J(x)^2$ for self-adjoint $x$. The terminology is
that of \cite{Kadison,KadisonRingrose2,Sherman05}.
A normal Jordan $*$-isomorphism has the usual extension to the
corresponding algebras of measurable or locally measurable operators.

\begin{lemma}
\label{lem:nc-nonhilbert}
Let $(\M,\tau)$ be atomless and semifinite, and let $E$ be a symmetric
Banach function space on $I_\tau$ with order-continuous norm. If
$E(\M,\tau)_{\sa}$ is a Hilbert space, then
\[
   E=L_2(I_\tau)
\]
as sets and there is a constant $c>0$ such that
\[
   \|f\|_E=c\|f\|_2,
   \qquad f\in E.
\]
Consequently, if the norm of $E$ is not proportional to the
$L_2$-norm, then $E(\M,\tau)_{\sa}$ is non-Hilbertian.
\end{lemma}

\begin{proof}
Suppose that the norm of $E(\M,\tau)_{\sa}$ is induced by an inner
product. Choose a nonzero projection $e\in\Pf(\M)$, put
\[
   d=\tau(e),
   \qquad c=\frac{\|e\|_E}{\sqrt d},
\]
and, for $m\ge0$, decompose $e$ into $2^m$ mutually orthogonal
projections of trace $d_m=2^{-m}d$. The norm on the span of these
projections is a Hilbert norm invariant under all coordinate sign
changes and permutations. Hence it is a scalar multiple of the
Euclidean norm. Evaluating it at the sum of the $2^m$ projections
gives
\begin{equation}
   \varphi_E(d_m)=c\sqrt{d_m},
   \qquad m\ge0.
   \label{eq:nc-hilbert-grid}
\end{equation}
The same formula determines the norm on the span of every finite
family of mutually orthogonal projections of trace $d_m$, because a
symmetric operator norm depends only on the singular-value function.

Let
\[
   f=\sum_{j=1}^r\alpha_j\chi_{A_j}
\]
be a real simple function with pairwise disjoint sets $A_j$ of finite
measure. Atomlessness and semifiniteness provide mutually orthogonal
finite projections $p_j$ with
\[
   \tau(p_j)=m(A_j).
\]
Put
\[
   x=\sum_{j=1}^r\alpha_jp_j.
\]
Then $\mu(x)=f^*$ and consequently $\|x\|_E=\|f\|_E$ and
$\|x\|_2=\|f\|_2$. For every sufficiently large $m$, choose
$q_{j,m}\le p_j$ such that
\[
   \tau(q_{j,m})=k_{j,m}d_m,
   \qquad 0\le\tau(p_j-q_{j,m})<d_m,
\]
and split $q_{j,m}$ into $k_{j,m}$ projections of trace $d_m$. With
\[
   x_m=\sum_{j=1}^r\alpha_jq_{j,m},
\]
formula \eqref{eq:nc-hilbert-grid} and orthogonality in the Hilbert
space give
\[
   \|x_m\|_E
   =c\left(\sum_{j=1}^r|\alpha_j|^2\tau(q_{j,m})\right)^{1/2}
   =c\|x_m\|_2.
\]
Furthermore,
\[
   s(x-x_m)\le\sum_{j=1}^r(p_j-q_{j,m}),
   \qquad
   \tau(s(x-x_m))<rd_m,
\]
and hence, with $A=\max_j|\alpha_j|$,
\[
   \|x-x_m\|_E
   \le A\,\varphi_E(rd_m)\longrightarrow0
\]
by order continuity. Moreover, $\|x-x_m\|_2\to0$. Passing to the
limit yields
\[
   \|f\|_E=\|x\|_E=c\|x\|_2=c\|f\|_2.
\]
Finite-support simple functions are dense in $E$. This identity shows
that their $E$-completion is the same as their $L_2$-completion; hence $E=L_2(I_\tau)$ as sets and the asserted norm
identity holds throughout $E$.
\end{proof}

\section{Skew-Hermitian operators on the self-adjoint part}

\subsection{Representation of skew-Hermitian operators}

On a complex symmetric operator space, bounded Hermitian operators have
the form $x\mapsto ax+xb$ \cite[Theorem~3.10]{HS}. We prove
Theorem~\ref{thm:SH} directly on the real self-adjoint part.

For a finite projection $p$, the
fundamental-function identity
\[
 \norm{p}_E\,\norm{p}_{E^\times}=\tau(p)
\]
shows that
\[
 \phi_p(x)=\frac{\norm{p}_E}{\tau(p)}\tau(xp)
\]
is a norm-one functional satisfying
$\phi_p(p/\norm{p}_E)=1$.  Hence
\begin{equation}
 \tau(H(p)p)=0,
 \qquad p\in\Pf(\M).
 \label{eq:self-pairing-zero}
\end{equation}

\subsubsection{Trace vanishing on orthogonal projections}

The non-Hilbertian hypothesis yields a finite coordinate subspace on
which the norm is not Euclidean.
\begin{lemma}
\label{lem:block}
There are an integer $n\ge2$, a number $d>0$, and mutually orthogonal
projections $e_1,\dots,e_n\in\Pf(\M)$ with
$\tau(e_k)=d$ such that the symmetric coordinate norm
\[
 (\xi_1,\dots,\xi_n)\longmapsto
 \norm{\sum_{k=1}^n\xi_ke_k}_E
\]
is not Euclidean.
\end{lemma}

\begin{proof}
Suppose, to the contrary, that every finite equal-trace coordinate
block is Euclidean. Choose $0\ne e\in\Pf(\M)$, put
\[
   d_0=\tau(e),
   \qquad
   c=\frac{\|e\|_E}{\sqrt{d_0}},
\]
and, for $m\ge1$, decompose $e$ into $2^m$ mutually orthogonal
projections of equal trace
\[
   d_m=2^{-m}d_0.
\]
Euclideanity of the corresponding coordinate norm gives
\[
   \varphi_E(d_m)=c\sqrt{d_m}.
\]

Let
\[
   x=\sum_{j=1}^r\alpha_jp_j
\]
be finite-spectrum and self-adjoint, with finite support. For large
$m$, choose $q_{j,m}\le p_j$ such that
\[
   \tau(q_{j,m})=k_{j,m}d_m,
   \qquad
   \tau(p_j-q_{j,m})<d_m,
\]
and decompose each $q_{j,m}$ into $k_{j,m}$ projections of trace
$d_m$. Setting
\[
   x_m=\sum_{j=1}^r\alpha_jq_{j,m},
\]
we obtain
\[
\begin{aligned}
   \|x_m\|_E
   &=\varphi_E(d_m)
     \left(\sum_{j=1}^r k_{j,m}|\alpha_j|^2\right)^{1/2}\\
   &=c\left(
     \sum_{j=1}^r|\alpha_j|^2\tau(q_{j,m})
     \right)^{1/2}
   =c\|x_m\|_2.
\end{aligned}
\]
Since
\[
   \tau(s(x-x_m))\le rd_m\longrightarrow0,
\]
order continuity gives $\|x-x_m\|_E\to0$ and
$\|x-x_m\|_2\to0$. Hence
\[
   \|x\|_E=c\|x\|_2.
\]

By spectral approximation, the same identity holds for every
$x\in\Ftau_{\sa}$. Since $\Ftau_{\sa}$ is dense both in
$E(\M,\tau)_{\sa}$ and in $L_2(\M,\tau)_{\sa}$, the two completions
coincide and
\[
   E(\M,\tau)_{\sa}=L_2(\M,\tau)_{\sa},
   \qquad
   \|x\|_E=c\|x\|_2,
\]
contrary to the non-Hilbertian hypothesis.
\end{proof}

\begin{lemma}
\label{lem:offdiag}
If $p,q\in\Pf(\M)$ and $pq=0$, then
\[
 \tau(H(p)q)=0.
\]
\end{lemma}

\begin{proof}
Choose mutually orthogonal projections $e_1,\ldots,e_n$ and $d>0$ as
in \Cref{lem:block}; thus $\tau(e_k)=d$ and the symmetric coordinate
space
\[
 Z=\left\{\sum_{k=1}^n\xi_ke_k:\xi_k\in\mathbb R\right\}
\]
is not Euclidean. Put
\[
 t_{kl}=\tau(H(e_k)e_l),\qquad 1\le k,l\le n.
\]
Applying \eqref{eq:self-pairing-zero} to $e_k$ and to $e_k+e_l$
gives
\begin{equation}
 t_{kk}=0,
 \qquad
 t_{kl}+t_{lk}=0\quad(k\ne l).
 \label{eq:t-antisymmetric}
\end{equation}

The finite-dimensional argument of Arazy
\cite[Lemma~4]{ArazySeq}, used also in
\cite[proof of Lemma~3.3]{HS}, supplies vectors
$\xi=(\xi_1,\ldots,\xi_n)$ and
$\eta=(\eta_1,\ldots,\eta_n)$ such that
\[
 x=\sum_{k=1}^n\xi_ke_k,
 \qquad
 y=\sum_{k=1}^n\eta_ke_k
\]
form a norming pair in
$E(\M,\tau)_{\sa}\times E^\times(\M,\tau)_{\sa}$ and $\xi,\eta$
are linearly independent. For
$\varepsilon=(\varepsilon_1,\ldots,\varepsilon_n)\in\{-1,1\}^n$,
put
\[
 x_\varepsilon=\sum_{k=1}^n\varepsilon_k\xi_ke_k,
 \qquad
 y_\varepsilon=\sum_{k=1}^n\varepsilon_k\eta_ke_k.
\]
Coordinate sign changes preserve both norms and the trace pairing, so
$(x_\varepsilon,y_\varepsilon)$ is again a norming pair. The
skew-Hermitian property and \eqref{eq:t-antisymmetric} therefore give
\begin{equation}
 0=\tau(H(x_\varepsilon)y_\varepsilon)
  =\sum_{k<l}\varepsilon_k\varepsilon_l
    (\xi_k\eta_l-\xi_l\eta_k)t_{kl}.
 \label{eq:sign-average}
\end{equation}
Multiplying by $\varepsilon_r\varepsilon_s$ and averaging over all
signs yields
\[
 (\xi_r\eta_s-\xi_s\eta_r)t_{rs}=0.
\]
Because $\xi$ and $\eta$ are linearly independent, there are
$r<s$ for which the determinant is nonzero. Let $k\ne l$ be
arbitrary and choose a coordinate permutation carrying $(r,s)$ to
$(k,l)$. The permuted vectors again form a norming pair, and applying
\eqref{eq:sign-average} to them gives $t_{kl}=0$. Hence
\begin{equation}
 \tau(H(e_k)e_l)=0
 \qquad(k\ne l).
 \label{eq:block-offdiag}
\end{equation}

For $m\ge0$, set
\[
 d_m=2^{-m}d,
 \qquad
 N_m=2^mn.
\]
Split each $e_k$ into $2^m$ mutually orthogonal projections of trace
$d_m$. The coordinate norm on the resulting $N_m$ projections is not
Euclidean: otherwise its restriction to vectors constant on each of
the $n$ groups would make the original norm on $Z$ Euclidean. Since
any two families of mutually orthogonal projections having the same
finite trace give the same symmetric coordinate norm, the
sign-and-permutation argument above applies to every family of $N_m$
mutually orthogonal projections of trace $d_m$.

Let $u,v\in\Pf(\M)$ satisfy $uv=0$ and
$\tau(u)=\tau(v)=d_m$. Since the original block has total trace $nd$,
\[
 \tau(1-u-v)\ge nd-2d_m=(N_m-2)d_m.
\]
By atomlessness, $u$ and $v$ can therefore be completed to a family of
$N_m$ mutually orthogonal projections of trace $d_m$. Applying
\eqref{eq:block-offdiag} to this family gives
\begin{equation}
 \tau(H(u)v)=0
 \quad\text{whenever }uv=0,
 \quad \tau(u)=\tau(v)=d_m.
 \label{eq:equal-trace-offdiag}
\end{equation}

Let $p,q\in\Pf(\M)$ with $pq=0$. For every $m$, choose
projections $p_m\le p$ and $q_m\le q$ such that
\[
 \tau(p-p_m)<d_m,
 \qquad
 \tau(q-q_m)<d_m,
\]
and such that $p_m$ and $q_m$ are finite sums of projections of trace
$d_m$. By \eqref{eq:equal-trace-offdiag} and linearity,
\begin{equation}
 \tau(H(p_m)q_m)=0.
 \label{eq:pm-qm-zero}
\end{equation}
We have
\begin{equation}
 \tau(H(p)q)
 =\tau\bigl(H(p)(q-q_m)\bigr)
  +\tau\bigl(H(p-p_m)q_m\bigr),
 \label{eq:offdiag-errors}
\end{equation}
where \eqref{eq:pm-qm-zero} has been used. For a measurable operator
$x$ and a finite projection $e$, the Fack--Kosaki inequality gives
\[
 \|xe\|_1
 \le\int_0^{\tau(e)}\mu_t(x)\,dt;
\]
see \cite{FackKosaki}.  K\"othe duality gives, for every $a<\infty$,
\[
   \int_0^a\mu_t(H(p))\,dt
   \le \|H(p)\|_E\,\|\chi_{(0,a)}\|_{E^\times}<\infty.
\]
Thus $\mu(H(p))$ is locally integrable.  Since
$\tau(q-q_m)\to0$, absolute continuity of the integral yields
\[
 \left|\tau\bigl(H(p)(q-q_m)\bigr)\right|
 \le\|H(p)(q-q_m)\|_1\longrightarrow0.
\]
For the second term, K\"othe duality gives
\[
\begin{aligned}
 \left|\tau\bigl(H(p-p_m)q_m\bigr)\right|
 &\le\|H(p-p_m)\|_E\,\|q_m\|_{E^\times}\\
 &\le\|H\|\,\|p-p_m\|_E\,\|q\|_{E^\times}
 \longrightarrow0,
\end{aligned}
\]
because $\|p-p_m\|_E=\varphi_E(\tau(p-p_m))\to0$ by order
continuity. Letting $m\to\infty$ in
\eqref{eq:offdiag-errors} proves the assertion.
\end{proof}

The trace identity forces both diagonal corners of $H(p)$ to
vanish.
\begin{lemma}
\label{lem:corners}
For every $p\in\Pf(\M)$,
\[
 pH(p)p=0,
 \qquad
 (1-p)H(p)(1-p)=0.
\]
\end{lemma}

\begin{proof}
If $e\le p$, then \Cref{lem:offdiag} and
\eqref{eq:self-pairing-zero} give
\[
 \tau(H(p)e)=\tau(H(e)e)+\tau(H(p-e)e)=0.
\]
If $pH(p)p\ne0$, a nonzero positive or negative spectral projection
of $pH(p)p$ gives a finite $e\le p$ for which
$\tau(H(p)e)\ne0$, a contradiction.  Thus $pH(p)p=0$.

If $e\le1-p$ is finite, then \Cref{lem:offdiag} gives
$\tau(H(p)e)=0$.  If $(1-p)H(p)(1-p)\ne0$, semifiniteness supplies a
nonzero finite subprojection of a positive or negative spectral
projection, again contradicting the trace identity.
\end{proof}

The corner decomposition yields the operator-norm estimate
\[
\norm{H(p)}_\infty\le\norm{H}.
\]
\begin{lemma}
\label{lem:uniform}
For every $p\in\Pf(\M)$,
\[
 H(p)\in\Ftau_{\sa},
 \qquad
 \norm{H(p)}_\infty\le\norm{H}.
\]
Consequently,
\[
 \norm{H(x)}_\infty\le2\norm{H}\,\norm{x}_\infty,
 \qquad x\in\Ftau_{\sa}.
\]
\end{lemma}

\begin{proof}
By \Cref{lem:corners}, relative to $1=p+(1-p)$,
\[
 H(p)=\begin{pmatrix}0&A^*\\ A&0\end{pmatrix},
 \qquad A=(1-p)H(p)p.
\]
Suppose $\norm{A}_\infty>\norm{H}$ and choose
$\norm{H}<\lambda<\norm{A}_\infty$.  Let
$q=e^{|A|}(\lambda,\infty)\le p$.  Applying
\Cref{lem:corners} to $p-q$ gives
\[
 Aq=(1-p)H(q)q.
\]
Hence
\[
 \lambda\norm{q}_E
 \le\norm{Aq}_E
 \le\norm{H(q)}_E
 \le\norm{H}\,\norm{q}_E,
\]
a contradiction.  Thus $\norm{A}_\infty\le\norm{H}$, and the displayed
matrix form gives $\norm{H(p)}_\infty=\norm{A}_\infty$.  The right support of $A$ is dominated by $p$, while $l(A)\sim r(A)$ by
the polar decomposition; see, for example,
\cite[Section~6.1]{KadisonRingrose2}.  Hence $H(p)$ has
finite support.

For $0\le x\in\Ftau$, let $r=\supp x$ and $M=\norm{x}_\infty$.  Put
$\Delta_n=2^{-n}M$ and
\[
 x_n=\Delta_n\sum_{k=1}^{2^n}e^x([k\Delta_n,\infty)).
\]
Then $\norm{x_n-x}_E\to0$ and
\[
 \norm{H(x_n)}_\infty
 \le\Delta_n\sum_{k=1}^{2^n}\norm{H(e^x([k\Delta_n,\infty)))}_\infty
 \le M\norm{H}.
\]
The closed unit ball of $\M$ is closed in the measure topology
\cite[Theorem~32]{DoddsDePagter}; hence
$H(x)\in\M$ and $\norm{H(x)}_\infty\le M\norm{H}$.  Every
spectral projection occurring in $x_n$ is dominated by $r$, so
\Cref{lem:corners} gives
\[
   (1-r)H(x_n)(1-r)=0.
\]
Since $H(x_n)\to H(x)$ in measure and multiplication by bounded
operators is continuous in that topology, passage to the limit yields
\[
   (1-r)H(x)(1-r)=0.
\]
The remaining parts of $H(x)$ have one support projection dominated by
the finite projection $r$, while the other is equivalent to it.
Hence $H(x)$ has finite support and therefore belongs to $\Ftau$.
Applying this estimate to $x_+$ and $x_-$ gives the last
assertion.
\end{proof}

\begin{proof}[Proof of \Cref{thm:SH}]
By \Cref{lem:uniform}, $H|_{\Ftau_{\sa}}$ extends uniquely to a
bounded real-linear operator
\[
 \widehat H:C_0(\M,\tau)_{\sa}\longrightarrow
 C_0(\M,\tau)_{\sa}.
\]
Let $p_j,p_k$ be orthogonal finite projections.  The two corner
identities in \Cref{lem:corners}, first for $p_j$ and $p_k$ and then
for $p_j+p_k$, show that
\[
 H(p_j)p_k+p_jH(p_k)+H(p_k)p_j+p_kH(p_j)=0.
 \label{eq:projection-cross-identity}
\]
They also give
\[
   H(p_j)=H(p_j)p_j+p_jH(p_j).
\]
Consequently, if $x=\sum_j\lambda_jp_j$ is finite-spectrum and
self-adjoint, then the diagonal terms in $H(x)x+xH(x)$ equal
$\sum_j\lambda_j^2H(p_j)$, while the mixed terms vanish in pairs by
\eqref{eq:projection-cross-identity}.  Thus
\[
 H(x^2)=H(x)x+xH(x).
\]
Uniform approximation by finite-spectrum self-adjoint elements and the
spectral theorem \cite[Theorem~5.2.2]{KadisonRingrose1} extend this
identity to all $x=x^*\in C_0(\M,\tau)$.

Complexify $\widehat H$ by
\[
 \delta(x+iy)=\widehat H(x)+i\widehat H(y).
\]
Polarization of the square identity shows that $\delta$ is a bounded
Jordan $*$-derivation on the $C^*$-algebra $C_0(\M,\tau)$.  By
Johnson's theorem \cite[Theorem~6.3]{Johnson}, every bounded Jordan
derivation from a $C^*$-algebra into a Banach bimodule is a derivation.
By \cite[Lemma~3.9]{HS}, there exists $c\in\M$ such that
\[
 \delta(z)=cz-zc,
 \qquad z\in C_0(\M,\tau).
\]
Since $\delta$ is a $*$-derivation, the element $c+c^*$ commutes
with $C_0(\M,\tau)$.  In particular, it commutes with every finite
projection.  The finite projections generate $\M$ strongly, and hence
$c+c^*\in\Z(\M)$.  Replacing $c$ by $c-(c+c^*)/2$, we may assume
$c^*=-c$.  With $a=ic=a^*$,
\[
 \delta(z)=i(za-az).
\]
Hence the asserted formula holds on $\Ftau_{\sa}$. By order continuity,
$\Ftau_{\sa}$ is norm dense in $X=E(\M,\tau)_{\sa}$; see
\cite[p.~3306]{HS}. Thus the formula extends by continuity to all of
$X$. If both $a$ and $b$ implement $H$, then $a-b$ commutes with every finite projection and
hence with all of $\M$; thus $a-b\in\Z(\M)_{\sa}$.
\end{proof}

\subsection{The center and commutators}

Let
\[
 V:X=E(\M,\tau)_{\sa}\longrightarrow
 Y=F(\N,\nu)_{\sa}
\]
be a surjective real-linear isometry satisfying the hypotheses of
\Cref{thm:nc-main}. By \Cref{lem:nc-nonhilbert}, both $X$ and $Y$ are
non-Hilbertian. Thus \Cref{thm:SH} applies to skew-Hermitian operators
on either space.

Conjugation by $V$ preserves the center and its locally measurable
affiliated operators.
\begin{theorem}
\label{thm:center}
One has
\[
 V\bigl(X\cap\LS(\Z(\M))_{\sa}\bigr)
 =Y\cap\LS(\Z(\N))_{\sa}.
\]
\end{theorem}

\begin{proof}
Let $z\in X\cap\LS(\Z(\M))_{\sa}$ and $b=b^*\in\N$.
The operator $\DeltaN{b}$ is skew-Hermitian on $Y$, and hence
$V^{-1}\DeltaN{b}V$ is skew-Hermitian on $X$.  By
\Cref{thm:SH}, there is $a=a^*\in\M$ such that
\[
 V^{-1}\DeltaN{b}V=\DeltaM{a}.
\]
Since $z$ is affiliated with the center,
\[
 \DeltaN{b}(Vz)=V\DeltaM{a}(z)=0.
\]
Thus $Vz$ commutes with every bounded self-adjoint element of $\N$ and
is affiliated with $\Z(\N)$. Replacing $V$ by $V^{-1}$ gives the
reverse inclusion.
\end{proof}

For $a=a^*\in\M$, conjugation by $V$ sends $\DeltaM{a}$ to an inner
skew-Hermitian operator on $Y$.  The uniqueness assertion in
\Cref{thm:SH} therefore defines a bijective real-linear map
\begin{equation}
 \Theta_V:
 \M_{\sa}/\Z(\M)_{\sa}
 \longrightarrow
 \N_{\sa}/\Z(\N)_{\sa}
 \label{eq:Theta}
\end{equation}
by the rule
\begin{equation}
 V\DeltaM{a}V^{-1}=\DeltaN{b}
 \quad\Longleftrightarrow\quad
 \Theta_V(\dot a)=\dot b.
 \label{eq:Theta-def}
\end{equation}
Choose, for each $a\in\M_{\sa}$, a representative
$\widehat\Theta(a)\in\N_{\sa}$ of $\Theta_V(\dot a)$. Identities
involving these representatives are understood modulo the center.
The map $\Theta_V$ is a Lie isomorphism because conjugation preserves
commutators of operators. In particular,
\begin{equation}
 \Theta_V\bigl(\dot{-i[a,c]}\bigr)
 =
 \dot{-i[\widehat\Theta(a),\widehat\Theta(c)]},
 \qquad a,c\in\M_{\sa}.
 \label{eq:Theta-Lie}
\end{equation}

\begin{lemma}
\label{lem:central-commutator}
If $a,b\in\M$ and $[a,b]\in\Z(\M)$, then $[a,b]=0$.
\end{lemma}

\begin{proof}
If a central commutator were nonzero, then after restricting to a
nonzero central projection and then to a central spectral slice, it
would be boundedly invertible in the center.  Multiplying one factor by
that central inverse would give bounded elements $A,B$ in a unital
Banach algebra satisfying $[A,B]=1$.  Wintner's argument excludes this:
from $[A,B^n]=nB^{n-1}$ one obtains
\[
 \norm{B^n}\ge \frac{n!}{(2\norm{A})^n},
\]
contrary to $\norm{B^n}\le\norm{B}^n$.
\end{proof}

\subsubsection{Maximal abelian subalgebras}

Let $\A\subseteq\M$ be a maximal abelian von Neumann subalgebra and set
\begin{equation}
 \B_{\A}
 =\{\widehat\Theta(a):a\in\A_{\sa}\}'\cap\N.
 \label{eq:BA-def}
\end{equation}
The algebra $\B_{\A}$ does not depend on the chosen representatives,
since two representatives differ by a central element.

Equation~\eqref{eq:Theta-Lie} also determines the corresponding
commutants.
\begin{proposition}
\label{prop:masa-image}
The algebra $\B_{\A}$ is a maximal abelian von Neumann subalgebra of
$\N$, and
\begin{equation}
 V\bigl(E(\A,\tau|_{\A})_{\sa}\bigr)
 =Y\cap S(\B_{\A})_{\sa}.
 \label{eq:MASA-range}
\end{equation}
\end{proposition}

\begin{proof}
If $a,c\in\A_{\sa}$, then $[a,c]=0$. By \eqref{eq:Theta-Lie},
\[
 [\widehat\Theta(a),\widehat\Theta(c)]\in\Z(\N).
\]
It is therefore zero by \Cref{lem:central-commutator}.  Hence
\[
 \mathcal C_{\A}
 :=W^*(\Z(\N),\widehat\Theta(a):a\in\A_{\sa})
\]
is abelian and is contained in $\B_{\A}$.

Conversely, take $b=b^*\in\B_{\A}$.  Surjectivity of \eqref{eq:Theta}
gives $x=x^*\in\M$ such that $\Theta_V(\dot x)=\dot b$.  For every
$a\in\A_{\sa}$,
\[
 \Theta_V(\dot{-i[x,a]})
 =\dot{-i[b,\widehat\Theta(a)]}=0.
\]
Thus $[x,a]$ is central and hence zero by
\Cref{lem:central-commutator}.  Since $\A$ is maximal abelian,
$x\in\A$.  Consequently,
$b-\widehat\Theta(x)\in\Z(\N)$ and $b\in\mathcal C_{\A}$.
Thus $\B_{\A}=\mathcal C_{\A}$ is abelian. By the definition of
$\B_{\A}$,
\[
 \B_{\A}'\cap\N=\B_{\A}.
\]
Hence $\B_{\A}$ is maximal abelian; see
\cite[Section~9.4]{KadisonRingrose2}.

If $x\in E(\A,\tau|_{\A})_{\sa}$, then
$\DeltaM{a}(x)=0$ for every $a\in\A_{\sa}$.  From
\eqref{eq:Theta-def},
\[
 \DeltaN{\widehat\Theta(a)}(Vx)=V\DeltaM{a}(x)=0.
\]
Therefore $Vx$ is affiliated with $\B_{\A}$.  Conversely, if
$y\in Y\cap S(\B_{\A})_{\sa}$, then the same intertwining identity
gives $\DeltaM{a}(V^{-1}y)=0$ for every $a\in\A_{\sa}$, and hence
$V^{-1}y$ is affiliated with $\A$, which is
\eqref{eq:MASA-range}.
\end{proof}

\section{Proof of the noncommutative representation theorem}
\label{sec:nc-proof}

We apply the projection-extension theorem of de Jager and Conradie
\cite{DJC} to the support map on finite projections. The map is first
obtained on maximal abelian von Neumann subalgebras with semifinite
restricted trace.

\subsection{Construction of the Jordan map}

The restriction of the trace to an arbitrary maximal abelian
subalgebra need not be semifinite.  Finite families of projections,
however, can always be embedded in a maximal abelian subalgebra with
semifinite restricted trace.

\begin{lemma}
\label{lem:semifinite-MASA}
\begin{enumerate}[label=\textup{(\roman*)}]
\item Every finite family of commuting projections in $\Pf(\M)$ is
contained in a maximal abelian von Neumann subalgebra
$\A\subseteq\M$ such that $\tau|_{\A}$ is semifinite.
\item More generally, if $p\in\Pf(\M)$ and
$\mathcal C\subseteq p\M p$ is an abelian von Neumann subalgebra, then
$\mathcal C$ is contained in such an $\A$.
\item Every such $\A$ is atomless.
\end{enumerate}
\end{lemma}

\begin{proof}
For \textup{(i)}, let $r$ be the supremum of the prescribed
projections. Since the family is finite and each projection has finite
trace, $r\in\Pf(\M)$. For \textup{(ii)}, set $r=p$.

In either case, let $\mathcal D\subseteq r\M r$ be the abelian von
Neumann algebra generated by the prescribed elements. By Zorn's
lemma, $\mathcal D$ is contained in a maximal abelian von Neumann
subalgebra
\[
 \A_0\subseteq r\M r.
\]

Since $\tau$ is semifinite, every nonzero projection contains a
nonzero subprojection of finite trace; see, for example,
\cite[Section~8.1]{KadisonRingrose2}. Hence we may choose a maximal
mutually orthogonal family $(r_\alpha)$ of finite projections
satisfying
\[
 r_\alpha\le 1-r.
\]
Maximality implies
\[
 \sum_\alpha r_\alpha=1-r.
\]
For each $\alpha$, choose a maximal abelian von Neumann subalgebra
\[
 \A_\alpha\subseteq r_\alpha\M r_\alpha,
\]
and define
\[
 \A=\A_0\oplus\bigoplus_\alpha\A_\alpha.
\]

Suppose that $x\in\M$ commutes with $\A$. Then $x$ commutes with
$r$ and with every $r_\alpha$.
Consequently,
\[
 x
 =
 rxr+\sum_\alpha r_\alpha x r_\alpha.
\]
Since $rxr$ commutes with $\A_0$ and each
$r_\alpha x r_\alpha$ commutes with $\A_\alpha$, maximality of these
abelian algebras gives
\[
 rxr\in\A_0,
 \qquad
 r_\alpha x r_\alpha\in\A_\alpha.
\]
Thus $x\in\A$, and
\[
 \A'\cap\M=\A.
\]
This proves that $\A$ is maximal abelian; see
\cite[Section~9.4]{KadisonRingrose2}.

For every finite set $F$ of indices, put
\[
 e_F=r+\sum_{\alpha\in F}r_\alpha.
\]
Then $e_F\in\A$, $\tau(e_F)<\infty$, and $e_F\uparrow1$. If
$0\le a\in\A$, then
\[
 0\le ae_F\le a,
 \qquad
 \tau(ae_F)\le \norm{a}_\infty\,\tau(e_F)<\infty,
\]
and $ae_F\uparrow a$, proving that $\tau|_{\A}$ is semifinite.

To exclude atoms in $\A$, suppose that $e$ is a nonzero atom. Then
\[
 e\A e=\mathbb C e.
\]
Maximal abelianness of $\A$ makes $e\A e$ maximal abelian in $e\M e$;
hence
\[
 e\M e=\mathbb C e,
\]
so $e$ is a minimal projection of $\M$. This contradicts the
atomlessness of $\M$. Hence $\A$ is atomless.
\end{proof}

For a maximal abelian $\A\subseteq\M$, let $\B_{\A}$ be defined by
\eqref{eq:BA-def} and put
\begin{equation}
 e_{\A}
 =\bigvee\{q\in\mathcal P(\B_{\A}):\nu(q)<\infty\}.
 \label{eq:eA}
\end{equation}
Then $\nu|_{e_{\A}\B_{\A}}$ is faithful, normal and semifinite.
Moreover,
\begin{equation}
 Y\cap S(\B_{\A})_{\sa}
 =F(e_{\A}\B_{\A},\nu|_{e_{\A}\B_{\A}})_{\sa}.
 \label{eq:target-abelian-space}
\end{equation}
Indeed, order continuity implies
$F(\N,\nu)\subseteq S_0(\N,\nu)$.  Thus, for
$y\in Y\cap S(\B_{\A})$, every spectral projection
$\mathbf1_{(1/n,\infty)}(|y|)$ has finite $\nu$-trace and
\[
 s(y)=\bigvee_{n\ge1}\mathbf1_{(1/n,\infty)}(|y|)\le e_{\A}.
\]
The reverse inclusion in \eqref{eq:target-abelian-space} follows
directly from the definition of the symmetric space on the corner.
The projection $e_{\A}$ is nonzero: choose a nonzero
$p\in\Pf(\M)\cap\A$; then $V(p)\ne0$ and a nonzero finite spectral
projection of $|V(p)|$ belongs to $\B_{\A}$ and is dominated by
$e_{\A}$.  The maximal abelian algebra $\B_{\A}$ is atomless. Indeed,
if $q$ were an atom of $\B_{\A}$, then $q\B_{\A}q=\mathbb Cq$ would be
maximal abelian in $q\N q$, forcing $q\N q=\mathbb Cq$. This would make
$q$ a minimal projection of $\N$, contrary to atomlessness. Hence the
corner $e_{\A}\B_{\A}$ is atomless as well.

The abelian restrictions also have order-continuous norm in the
Banach-lattice sense.

\begin{lemma}
\label{lem:abelian-order-continuity}
Let $\A\subseteq\M$ be an abelian von Neumann subalgebra such that
$\tau|_{\A}$ is semifinite. If the norm of the symmetric function
space $E$ is order continuous, then the norm of
$E(\A,\tau|_{\A})$ is order continuous in the Banach-lattice sense:
whenever $0\le x_\alpha\downarrow0$ in $E(\A,\tau|_{\A})$, one has
$\|x_\alpha\|_E\to0$.
\end{lemma}

\begin{proof}
Fix an index $\alpha_0$ and put $x=x_{\alpha_0}$.  By truncation and
order continuity, for every $\varepsilon>0$ there is a bounded
$0\le g\le x$, supported by a projection of finite trace, such that
\[
   \|x-g\|_E<\varepsilon.
\]
For $\alpha\ge\alpha_0$, set $y_\alpha=x_\alpha\wedge g$. Then
$0\le y_\alpha\downarrow0$ and
\[
   0\le x_\alpha-y_\alpha\le x-g.
\]
We may therefore work in the finite-measure abelian algebra supporting
$g$. A decreasing net $(A_\alpha)$ of measurable subsets of a
finite-measure set, with $\bigwedge_\alpha A_\alpha=0$ in the measure
algebra, satisfies $\mu(A_\alpha)\to0$. Indeed, if
$c=\inf_\alpha\mu(A_\alpha)>0$, choose an increasing sequence of
indices $\alpha_n$ with $\mu(A_{\alpha_n})\downarrow c$. Sequential
continuity from above gives
$\mu(\bigcap_nA_{\alpha_n})=c$.  Fix $\alpha$. For every $n$, choose
$\beta_n\ge\alpha,\alpha_n$. Then
\[
 \mu\bigl((\bigcap_jA_{\alpha_j})\setminus A_{\alpha}\bigr)
 \le \mu(A_{\alpha_n}\setminus A_{\beta_n})
 \le \mu(A_{\alpha_n})-c\longrightarrow0.
\]
Thus the intersection is contained modulo null sets in every
$A_\alpha$, a contradiction.

Applying this observation to the level sets of $y_\alpha$ shows that
$y_\alpha\to0$ in measure. If $\|y_\alpha\|_E$ did not tend to zero,
we could choose an increasing sequence of indices $(\alpha_n)$ such
that $\|y_{\alpha_n}\|_E$ stays bounded away from zero while
$y_{\alpha_n}\to0$ in measure. Since the sequence is decreasing, it
converges almost everywhere to zero. The original order continuity of
$E$ and the domination $y_{\alpha_n}\le g$ would then give
$\|y_{\alpha_n}\|_E\to0$, a contradiction. Hence
$\|y_\alpha\|_E\to0$, and therefore
\[
   \limsup_\alpha\|x_\alpha\|_E\le\varepsilon.
\]
Since $\varepsilon>0$ was arbitrary, $\|x_\alpha\|_E\to0$.
\end{proof}

\begin{lemma}
\label{lem:masa-nonhilbert}
Let $\A\subseteq\M$ be an atomless maximal abelian von Neumann
subalgebra such that $\tau|_{\A}$ is semifinite.  If the norm of
$E(\M,\tau)$ is not proportional to the $L_2$-norm, then
$E(\A,\tau|_{\A})_{\sa}$ is not a Hilbert space.  Consequently, the
space
\[
   F(e_{\A}\B_{\A},\nu|_{e_{\A}\B_{\A}})_{\sa}
\]
is also non-Hilbertian.
\end{lemma}

\begin{proof}
Suppose that $E(\A,\tau|_{\A})_{\sa}$ is Hilbertian.  By
\Cref{lem:ri-hilbert}, its norm is equal to $c\|\cdot\|_2$ for some
$c>0$.  Let
\[
   f=\sum_{k=1}^n\alpha_k\chi_{I_k}
\]
be a simple function on $I_\tau$ whose support has finite measure,
where the intervals $I_k$ are pairwise disjoint.  Atomlessness and
semifiniteness of $\tau|_{\A}$ provide mutually orthogonal projections
$p_k\in\A$ with $\tau(p_k)=m(I_k)$.  For
$x=\sum_k\alpha_kp_k$ we have $\mu(x)=f^*$, and hence
\[
   \|f\|_E=\|x\|_{E(\A)}=c\|x\|_2=c\|f\|_2.
\]
Simple functions with support of finite measure are dense by order
continuity.  It follows that the original function norm defining
$E(\M,\tau)$ is proportional to the $L_2$-norm, a contradiction.
Thus the source restriction is non-Hilbertian.  By
\Cref{prop:masa-image} and \eqref{eq:target-abelian-space}, its image under $V$ is
isometric to the displayed target space, which is therefore
non-Hilbertian as well.
\end{proof}

\begin{proposition}
\label{prop:abelian-restriction}
Let $\A$ be a maximal abelian von Neumann subalgebra such that
$\tau|_{\A}$ is semifinite.  By \Cref{thm:between-main}, the restriction
\[
 V_{\A}:
 E(\A,\tau|_{\A})_{\sa}
 \longrightarrow
 F(e_{\A}\B_{\A},\nu|_{e_{\A}\B_{\A}})_{\sa}
\]
is an elementary surjective real-linear isometry.  In particular,
there exist a measurable multiplier $w_{\A}$ with support $e_{\A}$
and a normal lattice isomorphism $\Phi_{\A}$ such that
\begin{equation}
 V(x)=w_{\A}\Phi_{\A}(x),
 \qquad x\in E(\A,\tau|_{\A})_{\sa}.
 \label{eq:abelian-elementary}
\end{equation}
If $\theta_{\A}$ is the induced complete Boolean isomorphism, then
\begin{equation}
 s(V(p))=\theta_{\A}(p),
 \qquad p\in\Pf(\M)\cap\A.
 \label{eq:theta-support}
\end{equation}
\end{proposition}

\begin{proof}
Surjectivity of the restriction follows from
\Cref{prop:masa-image} and \eqref{eq:target-abelian-space}. By
\Cref{lem:semifinite-MASA} and the discussion following
\eqref{eq:eA}, both abelian measure spaces are atomless and semifinite.
Their measure algebras are complete and localizable, as follows from
the standard representation of commutative von Neumann algebras with
faithful normal semifinite traces. Their norms are order continuous in
the Banach-lattice sense by
\Cref{lem:abelian-order-continuity}. They are non-Hilbertian by
\Cref{lem:masa-nonhilbert}, and hence are not proportional to the
corresponding $L_2$-norms by \Cref{lem:ri-hilbert}. Therefore
\Cref{thm:between-main} gives
\eqref{eq:abelian-elementary}.  Since the multiplier has full support,
it does not change supports, and \eqref{eq:theta-support} follows.
\end{proof}

Define
\begin{equation}
 \phi:\Pf(\M)\longrightarrow\mathcal P(\N),
 \qquad
 \phi(p)=s(V(p)).
 \label{eq:phi}
\end{equation}

Because $\phi(p)=s(V(p))$ is defined independently of the chosen
abelian subalgebra, the local Boolean maps agree on finite
projections.  In particular, $\phi$ is orthogonally additive.
\begin{proposition}
\label{prop:phi-additive}
If $p,q\in\Pf(\M)$ and $pq=0$, then
\begin{equation}
 \phi(p+q)=\phi(p)+\phi(q).
 \label{eq:phi-additive}
\end{equation}
In particular, $\phi(p)\phi(q)=0$.  Also,
$\phi(p)\ne0$ whenever $p\ne0$.
\end{proposition}

\begin{proof}
Choose, by \Cref{lem:semifinite-MASA}, a semifinite maximal abelian
algebra $\A$ containing $p$ and $q$.  By
\Cref{prop:abelian-restriction},
\[
 \phi(p)=\theta_{\A}(p),\qquad
 \phi(q)=\theta_{\A}(q),\qquad
 \phi(p+q)=\theta_{\A}(p+q).
\]
The Boolean isomorphism $\theta_{\A}$ is additive on orthogonal
projections, giving \eqref{eq:phi-additive}.  If $p\ne0$, then
$V(p)\ne0$ because $V$ is injective, so $\phi(p)=s(V(p))\ne0$.
\end{proof}

We verify the remaining hypotheses of the projection-extension theorem
of de Jager and Conradie \cite{DJC}.

For $x\in\Ftau$, write uniquely $x=x_1+ix_2$ with
$x_1,x_2\in\Ftau_{\sa}$ and define
\begin{equation}
 U(x)=V(x_1)+iV(x_2).
 \label{eq:complexification}
\end{equation}
Then $U:\Ftau\to S(\N,\nu)$ is complex linear and
\[
 \phi(p)=s(U(p)),\qquad p\in\Pf(\M).
\]

The remaining hypothesis is measure continuity on each finite
corner.
\begin{lemma}
\label{lem:U-measure-continuous}
Suppose $(x_n)\subseteq\Ftau$, $x\in\Ftau$,
$x_n\to x$ in operator norm and $s(x_n)\le s(x)$ for all $n$.  Then
\[
 U(x_n)\longrightarrow U(x)
\]
in the $\nu$-measure topology.
\end{lemma}

\begin{proof}
Put $p=s(x)$.  Since $s(x_n)\le p$, the right support of $x_n-x$ is
dominated by $p$.  Hence the ideal property of a symmetric operator
space \cite{DoddsDoddsdePagter} gives
\[
 \norm{x_n-x}_E
 \le \norm{x_n-x}_\infty\,\norm{p}_E.
\]
Moreover,
\[
 \norm{\operatorname{Re}(x_n-x)}_E,
 \ \norm{\operatorname{Im}(x_n-x)}_E
 \le \norm{x_n-x}_E.
\]
Since $V$ is an isometry on the self-adjoint part, each of
$V(\operatorname{Re}(x_n-x))$ and
$V(\operatorname{Im}(x_n-x))$ converges to zero in $F$-norm, hence in
measure.  Their complex sum is $U(x_n)-U(x)$, which also converges to
zero in measure.
\end{proof}

By \Cref{prop:phi-additive}, \Cref{lem:U-measure-continuous}, and
\cite[Theorem~3.7]{DJC}, the map $\phi$ extends to a positive complex
linear map
\begin{equation}
 \Phi:\Ftau\longrightarrow\N
 \label{eq:Phi-finite-core}
\end{equation}
which is contractive in operator norm and satisfies
\begin{equation}
 \Phi(x^2)=\Phi(x)^2,
 \qquad x=x^*\in\Ftau.
 \label{eq:Phi-square}
\end{equation}
Moreover, the factorization lemma in
\cite[Lemma~3.4]{DJC} gives
\begin{equation}
 U(x)=U(p)\Phi(x),
 \qquad x=x^*\in\Ftau,
 \quad p\in\Pf(\M),\ p\ge s(x).
 \label{eq:local-factorization-Phi}
\end{equation}

Normality on finite corners is inherited from the local Boolean
maps.
\begin{lemma}
\label{lem:Phi-normal}
The map $\Phi:\Ftau\to\N$ is normal.
\end{lemma}

\begin{proof}
Fix $p\in\Pf(\M)$. Normality may be checked on the finite von
Neumann algebra $p\M p$.  Let $(q_\alpha)$ be an increasing net of
projections in $p\M p$ with supremum $q$.  The abelian von Neumann
algebra generated by $p,q$ and the $q_\alpha$ is contained, by
\Cref{lem:semifinite-MASA}, in a semifinite maximal abelian algebra
$\A$.  Since $\Phi$ agrees with $\phi$ on finite projections,
\Cref{prop:abelian-restriction} gives
\[
 \Phi(q_\alpha)=\theta_{\A}(q_\alpha)
 \uparrow\theta_{\A}(q)=\Phi(q).
\]
A positive bounded linear map on a von Neumann algebra is normal if and
only if it preserves suprema of increasing nets of projections.
Therefore $\Phi|_{p\M p}$ is normal.  If
$0\le x_\alpha\uparrow x$ in $\Ftau$, then all these operators belong
to the finite corner $p\M p$ for $p=s(x)$, so
$\Phi(x_\alpha)\uparrow\Phi(x)$.
\end{proof}

By \Cref{lem:Phi-normal} and \cite[Theorem~4.5]{DJC}, $\Phi$ extends
uniquely to a normal Jordan $*$-homomorphism
\begin{equation}
 J:\M\longrightarrow\N.
 \label{eq:J}
\end{equation}
It satisfies
\begin{equation}
 J(p)=\phi(p)=s(V(p)),
 \qquad p\in\Pf(\M).
 \label{eq:J-support}
\end{equation}
Because $\phi(p)\ne0$ for every nonzero finite projection $p$, the
extension is injective; equivalently, it is isometric on $\M_{\sa}$,
as observed in \cite[Remark~4.6]{DJC}.

Repeating the finite-projection construction for $V^{-1}:Y\to X$
yields a normal Jordan $*$-homomorphism
\begin{equation}
 K:\N\longrightarrow\M
 \label{eq:K}
\end{equation}
which extends the finite-projection support map
\[
 q\longmapsto s(V^{-1}(q)),
 \qquad q\in\Pf(\N).
\]

\begin{theorem}
\label{thm:J-isomorphism}
The maps $J$ and $K$ are inverse to one another.  In particular,
$J:\M\to\N$ is a normal surjective Jordan $*$-isomorphism.
\end{theorem}

\begin{proof}
Fix $p\in\Pf(\M)$ and choose a semifinite maximal abelian algebra
$\A$ containing $p$.  Let $\theta_{\A}$ be the Boolean isomorphism in
\Cref{prop:abelian-restriction}.  By \eqref{eq:J-support},
\[
 J(p)=\theta_{\A}(p)\in e_{\A}\B_{\A}.
\]
Since $\nu|_{e_{\A}\B_{\A}}$ is semifinite, there is an increasing net
$(q_\gamma)$ of $\nu$-finite projections in $\B_{\A}$ such that
$q_\gamma\uparrow J(p)$.  On the abelian restriction, the inverse
isometry $V_{\A}^{-1}$ has Boolean map $\theta_{\A}^{-1}$.  Hence
\[
 K(q_\gamma)
 =s(V^{-1}(q_\gamma))
 =\theta_{\A}^{-1}(q_\gamma).
\]
By the normality of $K$ established above and the completeness of
$\theta_{\A}$ from \Cref{prop:band-conjugacy-general},
\[
 K(J(p))
 =\sup_\gamma K(q_\gamma)
 =\sup_\gamma\theta_{\A}^{-1}(q_\gamma)
 =p.
\]
Interchanging the roles of $\M$ and $\N$ and of $J$ and $K$ in the
preceding argument yields
\[
 J(K(q))=q,
 \qquad q\in\Pf(\N).
\]

Now let $e\in\mathcal P(\M)$.  By semifiniteness there is an increasing
net $p_\alpha\in\Pf(\M)$ with $p_\alpha\uparrow e$.  Normality yields
\[
 KJ(e)=\sup_\alpha KJ(p_\alpha)=\sup_\alpha p_\alpha=e.
\]
Similarly, $JK(f)=f$ for every projection $f\in\mathcal P(\N)$.
By the spectral theorem
\cite[Theorem~5.2.2]{KadisonRingrose1}, every self-adjoint element of
$\M$ is the operator-norm limit of finite linear combinations of its
spectral projections. Since $KJ$ and $JK$ are linear and fix all
projections, it follows that
\[
KJ=\id_{\M},
\qquad
JK=\id_{\N}.
\]
\end{proof}

\subsection{The central multiplier}

Since $J$ extends $\Phi$, formula
\eqref{eq:local-factorization-Phi} becomes
\begin{equation}
 V(x)=V(p)J(x),
 \qquad x=x^*\in\Ftau,
 \quad p\in\Pf(\M),\ p\ge s(x).
 \label{eq:local-factorization-J}
\end{equation}

For a finite projection, the coefficient in
\eqref{eq:local-factorization-J} is central on the corresponding
central summand.
\begin{proposition}
\label{prop:hp}
For every $p\in\Pf(\M)$ there is a unique
\[
 h_p\in\LS(\Z(\N)z(J(p)))_{\sa}
\]
such that
\begin{equation}
 V(p)=h_pJ(p).
 \label{eq:Vp-hp}
\end{equation}
If $p\le r$ are finite projections, then
\begin{equation}
 h_p=h_r
 \quad\text{on }z(J(p)).
 \label{eq:hp-compatible}
\end{equation}
\end{proposition}

\begin{proof}
Fix $p\in\Pf(\M)$.  If $q\le p$ is a finite projection, then
\eqref{eq:local-factorization-J}, applied with the majorizing
projection $p$, gives
\[
 V(q)=V(p)J(q).
\]
Since both $V(q)$ and $V(p)$ are self-adjoint, taking adjoints yields
\[
 V(p)J(q)=J(q)V(p).
 \label{eq:Vp-commutes-Jq}
\]

Now let $Q\le J(p)$ be any projection in $\N$. Since $K=J^{-1}$ on
projections, we may write
\[
 Q=J(q),
 \qquad
 q=K(Q)\le p.
\]
In particular, $q$ is finite. Hence \eqref{eq:Vp-commutes-Jq} shows
that $V(p)$ commutes with every projection in $J(p)\N J(p)$. Since
$s(V(p))=J(p)$ by \eqref{eq:J-support}, $V(p)$ is affiliated with
$J(p)\N J(p)$. Its spectral projections therefore commute with every
projection in this algebra and hence belong to its center. Thus
$V(p)$ is affiliated with
\[
 Z(J(p)\N J(p)).
\]

By the standard center-of-a-corner identity
\cite[Proposition~5.5.6]{KadisonRingrose1},
\[
 Z(J(p)\N J(p))
 =
 J(p)Z(\N)J(p).
\]
This center is naturally identified with
$Z(\N)z(J(p))$. Hence there exists
\[
 h_p\in\LS(\Z(\N)z(J(p)))_{\sa}
\]
such that
\[
 V(p)=h_pJ(p).
\]

To prove uniqueness, suppose that $\widetilde h_p$ has the same
property. Then
\[
 (h_p-\widetilde h_p)J(p)=0.
\]
Since $h_p-\widetilde h_p$ is central and $z(J(p))$ is the central
support of $J(p)$, it follows that
\[
 (h_p-\widetilde h_p)z(J(p))=0.
\]
Thus $h_p=\widetilde h_p$ on $z(J(p))$.

Let $p\le r$ be finite projections. Applying
\eqref{eq:local-factorization-J} to $p$ with the majorizing projection
$r$ gives
\[
 h_pJ(p)
 =
 V(p)
 =
 V(r)J(p)
 =
 h_rJ(r)J(p)
 =
 h_rJ(p).
\]
Hence
\[
 (h_p-h_r)J(p)=0.
\]
Since $h_p-h_r$ is central, the definition of $z(J(p))$ gives
\[
 (h_p-h_r)z(J(p))=0,
\]
which is precisely \eqref{eq:hp-compatible}.
\end{proof}

The compatibility relation in \eqref{eq:hp-compatible} determines a
single central locally measurable operator.
\begin{proposition}
\label{prop:global-h}
There is a unique $h\in\LS(\Z(\N))_{\sa}$ such that
\begin{equation}
 V(p)=hJ(p),
 \qquad p\in\Pf(\M).
 \label{eq:finite-projection-form}
\end{equation}
Moreover,
\begin{equation}
 V(x)=hJ(x),
 \qquad x=x^*\in\Ftau.
 \label{eq:finite-core-form}
\end{equation}
\end{proposition}

\begin{proof}
Let $p,r\in\Pf(\M)$. Since
\[
 \tau(p\vee r)\le \tau(p)+\tau(r)<\infty,
\]
the projection $u=p\vee r$ is finite. By \Cref{prop:hp},
\[
 h_p=h_u \quad\text{on }z(J(p)),
 \qquad
 h_r=h_u \quad\text{on }z(J(r)).
\]
Thus the family $(h_p)_{p\in\Pf(\M)}$ is compatible on common central
supports.

By semifiniteness
\cite[Section~8.1]{KadisonRingrose2}, there is an increasing net
$(p_\alpha)\subseteq\Pf(\M)$ such that
\[
 p_\alpha\uparrow1.
\]
Since $J$ is normal and unital,
\[
 J(p_\alpha)\uparrow1.
\]
Consequently, with
\[
 e_\alpha=z(J(p_\alpha)),
\]
we have
\[
 e_\alpha\uparrow1.
\]
By the mixing property of locally measurable operators
\cite[p.~498]{AyupovKudaybergenov},
the compatible family $(h_{p_\alpha})$ determines a unique
\[
 h\in\LS(\Z(\N))_{\sa}
\]
such that
\[
 he_\alpha=h_{p_\alpha},
 \qquad \alpha.
\]

The operator $h$ gives the required coefficient for every
$p\in\Pf(\M)$. Fix such a $p$ and put $e=z(J(p))$. For each $\alpha$,
apply \Cref{prop:hp} to the finite projection $p\vee p_\alpha$.
Restricting to $e e_\alpha$ yields
\[
 h_p=h_{p\vee p_\alpha}=h_{p_\alpha}=h.
\]
Since $e_\alpha\uparrow1$, we have $ee_\alpha\uparrow e$, and hence
\[
 he=h_p.
\]
Therefore
\[
 V(p)=h_pJ(p)=hJ(p),
\]
which proves \eqref{eq:finite-projection-form}.

To prove uniqueness, suppose that
$g\in\LS(\Z(\N))_{\sa}$ also satisfies
\[
 V(p)=gJ(p),
 \qquad p\in\Pf(\M),
\]
then
\[
 (g-h)J(p_\alpha)=0.
\]
Since $g-h$ is central,
\[
 (g-h)e_\alpha=0.
\]
Letting $e_\alpha\uparrow1$ gives $g=h$.

Let $x=x^*\in\Ftau$, and choose $p\in\Pf(\M)$ with
$p\ge s(x)$. Then
\[
 px=xp=x.
\]
By \eqref{eq:local-factorization-J} and
\eqref{eq:finite-projection-form},
\[
 V(x)
 =V(p)J(x)
 =hJ(p)J(x)
 =hJ(x).
\]
The last equality uses preservation of products of commuting elements
by Jordan $*$-homomorphisms; see \cite[Section~1]{Hamhalter23}.
Hence \eqref{eq:finite-core-form} holds.
\end{proof}

\begin{proof}[Proof of \Cref{thm:nc-main}]
Fix $x\in E(\M,\tau)_{\sa}$. By order continuity,
$\Ftau_{\sa}$ is norm dense in $E(\M,\tau)_{\sa}$; see
\cite[p.~3306]{HS}. Together with the spectral theorem
\cite[Theorem~5.2.2]{KadisonRingrose1}, this allows us to choose
finite-spectrum elements
\[
 x_n\in\Ftau_{\sa}
\]
such that
\[
 \norm{x_n-x}_E\longrightarrow0.
\]
Since $V$ is an isometry,
\[
 V(x_n)\longrightarrow V(x)
\]
in $F(\N,\nu)$, and hence locally in measure.

By the central decomposition theorem for Jordan $*$-isomorphisms
\cite[Theorem~3.3]{Stormer}, there is a central projection
$e\in\Z(\N)$ such that $eJ$ is a normal $*$-homomorphism and
$(1-e)J$ is a normal $*$-antihomomorphism. Consequently, $J$ has a
canonical extension to a Jordan $*$-isomorphism
\[
 J:\LS(\M)\longrightarrow\LS(\N),
\]
which is continuous for the local measure topologies. The extension is
again denoted by $J$.

Since $\LS(\N)$ is a topological $*$-algebra for the local measure
topology, multiplication by the fixed operator
$h\in\LS(\Z(\N))_{\sa}$ is continuous; see, for example,
\cite[Section~2]{BerChilinSukochev}. Therefore
\[
 hJ(x_n)\longrightarrow hJ(x)
\]
locally in measure. On the other hand, by \Cref{prop:global-h},
\[
 V(x_n)=hJ(x_n),
 \qquad n\ge1.
\]
The same sequence converges locally in measure to both $V(x)$ and
$hJ(x)$. Since the local measure topology is Hausdorff,
\[
 V(x)=hJ(x).
\]
As $x$ was arbitrary,
\[
 V(x)=hJ(x),
 \qquad x\in E(\M,\tau)_{\sa}.
\]

To prove that $h$ has full support, use semifiniteness to choose an
increasing net $(p_\alpha)\subseteq\Pf(\M)$ such that
$p_\alpha\uparrow1$. For every $\alpha$,
\[
 \begin{aligned}
 J(p_\alpha)
 &=s(V(p_\alpha))\\
 &=s\bigl(hJ(p_\alpha)\bigr)\\
 &=s(h)J(p_\alpha).
 \end{aligned}
\]
Since $J$ is normal and unital,
\[
 J(p_\alpha)\uparrow1.
\]
Taking suprema in this identity gives
\[
 s(h)=1.
\]

For uniqueness, suppose also that
\[
   V(x)=gL(x),
   \qquad x\in E(\M,\tau)_{\sa},
\]
where $L:\M\to\N$ is a normal surjective Jordan $*$-isomorphism and
$g\in\LS(\Z(\N))_{\sa}$ has full support.  For every
$p\in\Pf(\M)$,
\[
   L(p)=s(gL(p))=s(V(p))=J(p).
\]
Every projection of $\M$ is the supremum of an increasing net of
finite projections.  Normality therefore gives $L=J$ on all
projections, and the spectral theorem gives $L=J$ on $\M$.  Hence
\[
   (g-h)J(p)=0,
   \qquad p\in\Pf(\M).
\]
Choosing $p_\alpha\uparrow1$ as above and using the central supports
$z(J(p_\alpha))\uparrow1$, we obtain $g=h$, which is the asserted
uniqueness.

\end{proof}

\section{Mityagin's question}
\label{sec:real-Mityagin}

Mityagin asked whether a symmetric function space on a finite interval,
other than an $L_p$-space, can be linearly isomorphic or isometric to a
symmetric function space on the half-line or on the whole axis
\cite{Mityagin}. The complex isometric problem and its noncommutative
analogue were considered in \cite{FGHS}. We apply the representation
theorems proved above to the restriction argument used there.

\subsection{The commutative real case}

Two facts about restriction to sets of finite measure will be used.

\begin{lemma}
\label{lem:local-non-Lp}
Let $E(0,1)$ be a rearrangement-invariant Banach function space and let
$A\subset(0,1)$ have positive measure.  If the restriction $E(A)$,
after the usual measure-space identification, coincides with
$L_p(A)$ as a set for some $1\le p\le\infty$, then
\[
 E(0,1)=L_p(0,1)
\]
as sets.
\end{lemma}

\begin{proof}
Put $a=m(A)>0$.  By rearrangement invariance, the hypothesis is
equivalent to
\[
 E(0,a)=L_p(0,a)
\]
as sets.  Choose a finite measurable partition
\[
 (0,1)=A_1\sqcup\cdots\sqcup A_N
\]
with $m(A_j)\le a$ for every $j$.  If $f\in E(0,1)$, then
$f\chi_{A_j}\in E(A_j)$ and, after an equimeasurable transport into
$(0,a)$, belongs to $L_p$.  Since the partition is finite,
$f\in L_p(0,1)$.  Conversely, if $f\in L_p(0,1)$, then each
$f\chi_{A_j}$ belongs to $E(A_j)$, hence to $E(0,1)$, and therefore
$f=\sum_j f\chi_{A_j}\in E(0,1)$. The finite partition argument is
unchanged for $p=\infty$.
\end{proof}

\begin{lemma}
\label{lem:finite-band-modulus}
Let $(A,\mu)$ and $(B,\lambda)$ be atomless standard finite measure
spaces of positive measure.  Let $E(A)$ and $F(B)$ be real
rearrangement-invariant Banach function spaces with order-continuous
norms.  Assume that $E(A)$ does not coincide with $L_p(A)$ as a set
for any $1\le p\le\infty$.  If
\[
 W:E(A)\longrightarrow F(B)
\]
is a surjective real-linear isometry, then
\begin{equation}
 |W\chi_A|
 =\frac{\norm{\chi_A}_E}{\norm{\chi_B}_F}\,\chi_B
 \qquad\text{a.e.\ on }B.
 \label{eq:finite-band-modulus}
\end{equation}
\end{lemma}

\begin{proof}
Identify $A$ and $B$ with finite intervals, rescale both intervals to
$(0,1)$, and divide the two norms by the norms of their unit
functions.  The real finite-interval classification recalled in
\cite[Theorem~1.1.2]{FGHS} then yields a representation by a
measure-preserving transformation and a sign-valued multiplier.  In
the normalized spaces the image of the unit function therefore has
modulus one.  Undoing the norm normalization gives
\eqref{eq:finite-band-modulus}.
\end{proof}

\begin{theorem}
\label{thm:real-Mityagin-commutative}
Let $E(0,1)$ and $F(0,\infty)$ be real rearrangement-invariant Banach
function spaces with order-continuous norms.  Assume that
\[
 E(0,1)\ne L_p(0,1)
 \qquad\text{as sets for every }1\le p<\infty.
\]
The case $p=\infty$ is automatically excluded by order continuity.
Then there is no surjective real-linear isometry
\[
 T:E(0,1)\longrightarrow F(0,\infty).
\]
The same conclusion holds with $(0,\infty)$ replaced by the real
axis.
\end{theorem}

\begin{proof}
Suppose that such an isometry $T$ exists. Neither norm is Hilbertian,
so \Cref{thm:between-main} gives
\[
   Tf=w\,\Phi(f),
\]
where $w$ has full support and $\Phi$ is induced by a complete Boolean
isomorphism $\theta$. For $n\ge1$, put
\[
   B_n=(0,n),\qquad A_n=\theta^{-1}(B_n),\qquad
   G_n=A_n\setminus A_{n-1},
\]
with $A_0=\varnothing$. Then the $G_n$ are disjoint subsets of
$(0,1)$, so $m(G_n)\to0$.

Band conjugacy gives a surjective isometry
\[
   T:E(A_n)\longrightarrow F(B_n).
\]
By \Cref{lem:local-non-Lp,lem:finite-band-modulus}, there is $c_n>0$
such that
\[
   |T\chi_{A_n}|=c_n\chi_{B_n}.
\]
Since $T\chi_{A_n}=w\chi_{B_n}$ and the sets $B_n$ are nested, all
$c_n$ equal one constant $c>0$. Moreover,
$\theta(G_n)=(n-1,n)$, and hence
\[
   \|\chi_{G_n}\|_E
   =\|T\chi_{G_n}\|_F
   =c\|\chi_{(0,1)}\|_F>0
\]
for every $n$. This contradicts order continuity, since
$\|\chi_{G_n}\|_E=\varphi_E(m(G_n))\to0$.

The real line is measure-isomorphic to the half-line in the atomless
standard $\sigma$-finite category, which gives the stated real-line
case.
\end{proof}

\subsection{The noncommutative real case}

The noncommutative assertion follows from the same restriction argument
and \Cref{thm:nc-main}.

\begin{theorem}
\label{thm:real-Mityagin-nc}
Let $(\M,\tau)$ be an atomless finite von Neumann algebra with a
faithful normal tracial state, and let $(\N,\nu)$ be an atomless
semifinite von Neumann algebra with a faithful normal semifinite trace
satisfying
\[
 \nu(1)=\infty.
\]
Assume that the algebras act on separable Hilbert spaces.  Let
$E(0,1)$ and $F(0,\infty)$ be real rearrangement-invariant Banach
function spaces with order-continuous norms, and assume that
\[
 E(0,1)\ne L_p(0,1)
 \qquad\text{as sets for every }1\le p<\infty.
\]
Again, order continuity excludes coincidence with $L_\infty(0,1)$.
Then there is no surjective real-linear isometry
\[
 V:E(\M,\tau)_{\sa}
 \longrightarrow
 F(\N,\nu)_{\sa}.
\]
\end{theorem}

\begin{proof}
Assume that such an isometry $V$ exists.  The function norm of $E$ is
not proportional to the $L_2$-norm, since $E(0,1)$ does not even
coincide with $L_2(0,1)$ as a set.  By \Cref{lem:nc-nonhilbert}, the real Banach space
$E(\M,\tau)_{\sa}$ is not Hilbertian.  Hence the isometric target
cannot have norm proportional to the $L_2$-norm.
All hypotheses of \Cref{thm:nc-main} are therefore satisfied, and
\[
   V(x)=hJ(x),
\]
where $J:\M\to\N$ is a normal surjective Jordan $*$-isomorphism and
$h\in\LS(\Z(\N))_{\sa}$ has full support. Since $1\in E(\M,\tau)$,
we have $h=V(1)\in F(\N,\nu)$. Thus the central projections
\[
   e_k=\mathbf1_{(1/k,\infty)}(|h|)
\]
have finite trace and satisfy $e_k\uparrow1$.

Choose a maximal abelian von Neumann subalgebra $\B\subseteq\N$ such
that $\nu|_{\B}$ is semifinite. Since $e_k\in\Z(\N)\subseteq\B$,
the restriction $\nu|_{\B}$ is $\sigma$-finite and has infinite total
mass. Put $\A=J^{-1}(\B)$. Then $\A$ is an atomless maximal abelian
subalgebra of the finite algebra $\M$, and the restriction of $V$
defines a real-linear isometry
\[
   V|_{E(\A,\tau|_{\A})_{\sa}}:
   E(\A,\tau|_{\A})_{\sa}
   \longrightarrow
   F(\B,\nu|_{\B})_{\sa}.
\]
We verify that this restriction is onto.  If
$x\in E(\A,\tau|_{\A})_{\sa}$, then the extension of $J$ to locally
measurable operators sends $x$ to an operator affiliated with $\B$.
Since $h$ is affiliated with $\Z(\N)\subseteq\B$, the product
$V(x)=hJ(x)$ is affiliated with $\B$.  Conversely, let
$y\in F(\B,\nu|_{\B})_{\sa}$ and put $x=V^{-1}y$.  The condition
$s(h)=1$ implies that the measurable reciprocal $h^{-1}$ belongs to
$\LS(\Z(\N))$.  Hence
\[
   J(x)=h^{-1}y
\]
is affiliated with $\B$.  Applying the extension of $J^{-1}$ to
locally measurable operators shows that $x$ is affiliated with
$\A=J^{-1}(\B)$.  Since $x\in E(\M,\tau)_{\sa}$ already, it follows
that $x\in E(\A,\tau|_{\A})_{\sa}$.  Thus the displayed restriction is
surjective.

The separability assumption makes the two abelian measure spaces
standard, with total measures $1$ and $\infty$, respectively. This
contradicts \Cref{thm:real-Mityagin-commutative}.
\end{proof}

\appendix

\section{Local random-measure estimates}
\label{app:local-kernels}

This appendix develops the estimates used in
\Cref{thm:at-most-one-atom}. They produce measurable atomic kernels and,
on every source set of finite measure, a uniform $p$-variation bound.
Throughout \Cref{app:local-kernels,sec:kernel-composition}, $X$ denotes
a real rearrangement-invariant Banach function space on $(0,\infty)$
with order-continuous norm; the direct and weak-star settings are those
fixed in \Cref{sec:representing-kernels}.

We work with the Kalton--Weis random-measure representation
\cite{KaltonRep,Weis} in the form used in \cite[Section~6]{KR}. The
estimates are local in the source variable. The elementary two-band
identity below will be used repeatedly.

\begin{lemma}
\label{lem:triangular}
Let $E$ and $F$ be real Banach function spaces, and let
$T:E\to F$ be a surjective isometry. Suppose measurable sets $A$
and $D$ satisfy
\[
  P_DT=P_DTP_A,
  \qquad
  P_AT^{-1}=P_AT^{-1}P_D.
\]
Then
\[
  P_{D^c}TP_A=0.
\]
\end{lemma}

\begin{proof}
Relative to
\[
 E=E(A)\oplus E(A^c),
 \qquad
 F=F(D)\oplus F(D^c),
\]
write
\[
 T=
 \begin{pmatrix}
 P_DTP_A & P_DTP_{A^c}\\
 P_{D^c}TP_A & P_{D^c}TP_{A^c}
 \end{pmatrix}.
\]
The first hypothesis annihilates the upper-right block:
\[
 P_DT=P_DTP_A,
\]
and hence
\[
 P_DTP_{A^c}=P_DT(I-P_A)=0.
\]
The inverse relation gives, in the same way,
\[
 P_AT^{-1}=P_AT^{-1}P_D
\]
and
\[
 P_AT^{-1}P_{D^c}
 =P_AT^{-1}(I-P_D)=0.
\]
Relative to the two decompositions, therefore,
\[
 T=
 \begin{pmatrix}
 T_{11}&0\\
 T_{21}&T_{22}
 \end{pmatrix},
 \qquad
 T^{-1}=
 \begin{pmatrix}
 R_{11}&0\\
 R_{21}&R_{22}
 \end{pmatrix}.
\]
The identities $TT^{-1}=I_F$ and $T^{-1}T=I_E$ make both diagonal
blocks invertible, and the inverse matrix is
\[
 T^{-1}=
 \begin{pmatrix}
 T_{11}^{-1}&0\\
 -T_{22}^{-1}T_{21}T_{11}^{-1}&T_{22}^{-1}
 \end{pmatrix}.
\]

Let
\[
 R_E=2P_A-I,
 \qquad
 R_F=2P_D-I.
\]
Both are surjective isometries. Hence
\[
 U:=T^{-1}R_FTR_E
\]
is a surjective isometry. Direct multiplication gives
\[
 U=
 \begin{pmatrix}
 I&0\\
 -2T_{22}^{-1}T_{21}&I
 \end{pmatrix}
 =I+N,
 \qquad N^2=0.
\]
If $N\ne0$, then $U^n=I+nN$ is an isometry for every
$n\in\mathbb N$. Choosing $x$ such that $Nx\ne0$, we obtain
\[
 \norm{x}
 =\norm{U^nx}
 =\norm{x+nNx}
 \ge n\norm{Nx}-\norm{x},
\]
which is impossible as $n\to\infty$. The contradiction forces
$N=0$ and hence $T_{21}=0$; equivalently,
\[
 P_{D^c}TP_A=0.
\]
\end{proof}

\subsection{Estimates on sets of finite measure}

Fix a finite interval $I=(0,R)$ and let
\[
 I_{n,k}=((k-1)R2^{-n},kR2^{-n}],
 \qquad 1\le k\le2^n.
\]
Put
\[
 e_{n,k}=\chi_{I_{n,k}},
 \qquad
 e_{n,k}^*=\frac{2^n}{R}\chi_{I_{n,k}}.
\]
The normalization gives $\int e_{n,k}e_{n,k}^*\,dm=1$.

Let $Y$ be either $X$, with its norm topology, or $X^\times$, with the
weak-star topology $\sigma(X^\times,X)$. Let $U:Y\to Y$ be a
surjective isometry, assumed weak-star continuous in the second case.
In the weak-star case write $U=W^*$ for its preadjoint $W:X\to X$.
Since $U$ maps the closed unit ball of $X^\times$ onto itself,
\[
   \|Wx\|_X
   =\sup_{g\in B_{X^\times}}|\langle Ug,x\rangle|
   =\|x\|_X,
   \qquad x\in X.
\]
The range of $W$ is closed, while its annihilator is the kernel of
$W^*=U$; thus $W$ is onto. It follows that
$U^{-1}=(W^{-1})^*$ is weak-star continuous. In this case write
$V=W^{-1}$ for the preadjoint of $U^{-1}$.
Put
\[
   f_{n,k}=Ue_{n,k},
   \qquad
   g_{n,k}=
   \begin{cases}
      (U^{-1})^*e_{n,k}^*,&Y=X,\\
      Ve_{n,k}^*,&Y=X^\times.
   \end{cases}
\]
Then, in both cases,
\begin{equation}
   \langle U^{-1}h,e_{n,k}^*\rangle
   =
   \langle h,g_{n,k}\rangle,
   \qquad h\in Y.
   \label{eq:preadjoint}
\end{equation}
Define
\[
 B_n=\{s:g_{n,k}(s)=0\text{ for every }k\},
 \qquad B=\bigcap_{n\ge1}B_n.
\]
The definitions imply that $B_{n+1}\subset B_n$.

\begin{lemma}
\label{lem:dyadic-detection}
For every $n$ and $k$,
\[
  P_Bf_{n,k}=0.
\]
\end{lemma}

\begin{proof}
Let $Q=UP_IU^{-1}$. To verify $QP_B=0$, take $h\in Y(B)$.
By \eqref{eq:preadjoint},
\[
 \int (U^{-1}h)e_{n,k}^*\,dm
 =\int h\,g_{n,k}\,dm=0
\]
for all $n,k$. Hence every dyadic average of $P_IU^{-1}h$ vanishes,
so $P_IU^{-1}h=0$ and therefore $Qh=0$.

Set $R=2Q-I$. Since $Q=UP_IU^{-1}$, we have
$R=U(2P_I-I)U^{-1}$.  As $2P_I-I$ is an involutive surjective
isometry, so is $R$.

Since $QP_B=0$,
\[
 P_{B^c}RP_B=0.
\]
Apply \Cref{lem:triangular} to $R$, with $A=D=B^c$. We get
\[
 P_BRP_{B^c}=0.
\]
The last two identities give $P_BQP_{B^c}=0$. Together with $QP_B=0$
this says $P_BQ=0$, and $P_Bf_{n,k}=0$ because $f_{n,k}\in\Ran Q$.
\end{proof}

Fix a non-Hilbertian finite interval $I_0=(0,R_0)$ with property $(P)$.
In the dyadic notation above, take $I=I_0$ from this point onward.

The properties $(P)$ and $(P')$ are those of
\Cref{sec:representing-kernels}. The argument requires a finite-window
version of \cite[Theorem~6.1]{KR}, with a constant uniform over the
isometry. The proof below records that dependence and handles the
weak-star case simultaneously.

\begin{theorem}
\label{thm:fixed-window}
Let $0<p\le1$. There is $C_{p,I_0}<\infty$, depending only on $p$ and
the norm on $I_0$, such that every surjective isometry $U:X\to X$
satisfies
\begin{equation}
 \norm{
 \sup_{n\ge1}
 \left(\sum_{k=1}^{2^n}|Ue_{n,k}|^p\right)^{1/p}
 }_X
 \le C_{p,I_0}.
 \label{eq:fixed-window}
\end{equation}
The same conclusion holds on $X^\times$ for every weak-star continuous
surjective isometry, provided $X^\times(I_0)$ is non-Hilbertian and
has property $(P)$.
\end{theorem}

\begin{proof}
Let $Z=Y(I_0)$ and, for $n\ge1$, put
\[
   E_n=\operatorname{span}\{e_{n,k}:1\le k\le2^n\}.
\]
The dyadic conditional expectation on $I_0$, extended by zero outside
$I_0$, is
\begin{equation}
   \mathbb E_n h
   =\sum_{k=1}^{2^n}
      \left(\int_{I_0}h e_{n,k}^*\,dm\right)e_{n,k}.
   \label{eq:dyadic-expectation}
\end{equation}
The usual majorization inequality for conditional expectations makes
$\mathbb E_n$ contractive on each rearrangement-invariant space under
consideration and on its K\"othe dual. It is also self-adjoint for the
integral pairing. Under that pairing, $E_n^*$ is therefore identified
isometrically with
\[
   \operatorname{span}\{e_{n,k}^*:1\le k\le2^n\},
\]
where the span is taken in $X^\times$ in the direct case and in $X$ in
the weak-star case. In the latter case
\eqref{eq:dyadic-expectation} is weak-star continuous and its
preadjoint is the same averaging operator on $X$. Every functional on
$E_n$ consequently has a representative in the displayed span with the
same norm.

The only finite-dimensional input needed here is a coefficient estimate
uniform in the dyadic level. There is a constant $A_{p,I_0}$, depending
only on $p$ and the norm of $Z$, such that
\begin{equation}
   \left(\sum_{k=1}^{2^n}|c_k|^p\right)^{1/p}
   \le A_{p,I_0}\max_{1\le k\le2^n}|c_k|
   \label{eq:local-flinn-coefficients}
\end{equation}
whenever
\[
   F=\sum_{k=1}^{2^n}c_ke_{n,k}^*
\]
is a Flinn element of $E_n^*$.

To verify this, let $Z^\times$ be the K\"othe dual of $Z$.  The
canonical map from $Z$ into $Z^{\times\times}$ is isometric in both
cases.  In the direct case this follows from order continuity, since
$Z^*=Z^\times$; in the weak-star case $Z$ is a K\"othe dual and has the
Fatou property.  The two norms in the definition of $(P)$ have the same values in
$Z$ and in its canonical image in $Z^{\times\times}$. Hence
$Z^{\times\times}$ has property $(P)$, and $Z^\times$ has property
$(P')$.

The norm of $Z^\times$ is not proportional to the $L_2(I_0)$-norm.  If
it were, then $Z^{\times\times}$ would have a norm proportional to the
$L_2$-norm.  In the weak-star case $Z=Z^{\times\times}$ isometrically.
In the direct case finite-support simple functions are norm dense in
$Z$, and the isometric embedding $Z\hookrightarrow Z^{\times\times}$
would make the $Z$-norm proportional to the $L_2$-norm on this dense
class and hence on all of $Z$.  Both conclusions contradict the choice
of $I_0$.

Let $\widehat Z$ be the rearrangement-invariant space on $(0,1)$
obtained from $Z$ by the isometry
\[
   (\mathcal Rz)(t)=z(R_0t),
   \qquad 0<t<1.
\]
The corresponding isometry of K\"othe duals is
\[
   (\mathcal R^\times h)(t)=R_0h(R_0t).
\]
Accordingly, $\widehat Z^\times$ is not proportional to $L_2(0,1)$
and has property $(P')$.  Multiplying its norm by a positive scalar, if
necessary, gives the normalization used in \cite{KR}; this scalar
renorming does not change numerical positivity or the set of Flinn
elements.  If
\[
   J_{n,k}=((k-1)2^{-n},k2^{-n}],
\]
then
\[
   \mathcal R^\times F
   =2^n\sum_{k=1}^{2^n}c_k\chi_{J_{n,k}}.
\]
The map $\mathcal R^\times$ preserves Flinn elements.  Applying
\cite[Proposition~5.5]{KR} to the dyadic subspace of
$\widehat Z^\times$ gives
\[
   \left(\sum_{k=1}^{2^n}|2^nc_k|^p\right)^{1/p}
   \le A_{p,I_0}\max_k|2^nc_k|.
\]
Cancelling $2^n$ proves \eqref{eq:local-flinn-coefficients}.  In
particular, the same constant works for every $n$.

Let
\[
   S=
   \begin{cases}
      (U^{-1})^*,&Y=X,\\
      V,&Y=X^\times.
   \end{cases}
\]
Then $S$ is a surjective isometry on the space in which
$E_n^*$ is represented, and
\begin{equation}
   \langle Ux,Sx^*\rangle=\langle x,x^*\rangle
   \label{eq:paired-isometries}
\end{equation}
for $x\in Y$ and for $x^*$ in the corresponding dual or predual.
Recall that
\[
   f_{n,k}=Ue_{n,k},
   \qquad
   g_{n,k}=Se_{n,k}^*.
\]
Choose measurable representatives of these countably many functions.

Fix $n$ and let $(x,x^*)\in\Pi(E_n)$, where
\[
   x=\sum_{k=1}^{2^n}a_ke_{n,k},
   \qquad
   x^*=\sum_{k=1}^{2^n}a_k^*e_{n,k}^*.
\]
Under the isometric identification just established, this is also a
norming pair in the ambient K\"othe pairing.  Equation
\eqref{eq:paired-isometries} shows that $(Ux,Sx^*)$ is a norming pair.
The norming relation gives
\[
   1=\int (Ux)(Sx^*)\,dm
   \le\int |Ux|\,|Sx^*|\,dm\le1,
\]
and hence
\begin{equation}
   \left(\sum_{k=1}^{2^n}a_kf_{n,k}(s)\right)
   \left(\sum_{k=1}^{2^n}a_k^*g_{n,k}(s)\right)
   \ge0
   \label{eq:pointwise-numerical-positivity}
\end{equation}
for almost every $s$.

The set $\Pi(E_n)$ is compact and metrizable.  Taking a countable
dense subset and using continuity of the two finite sums in the
coefficients, we obtain a null set $N_n$ outside which
\eqref{eq:pointwise-numerical-positivity} holds for every
$(x,x^*)\in\Pi(E_n)$.  Put
\[
   N=\bigcup_{n\ge1}N_n.
\]
For $s\notin N$, define
\[
   F_n(s)=\sum_{k=1}^{2^n}f_{n,k}(s)e_{n,k}^*\in E_n^*,
   \qquad
   G_n(s)=\sum_{k=1}^{2^n}g_{n,k}(s)e_{n,k}\in E_n.
\]
Then the rank-one operator on $E_n^*$ given by
\[
   R_{n,s}\varphi=\varphi(G_n(s))F_n(s)
\]
is numerically positive.  Indeed, under the canonical identification
$E_n^{**}=E_n$, every element of $\Pi(E_n^*)$ has the form
$(x^*,x)$ with $(x,x^*)\in\Pi(E_n)$, and
\eqref{eq:pointwise-numerical-positivity} is precisely
\[
   x(R_{n,s}x^*)=x^*(G_n(s))F_n(s)(x)\ge0.
\]
Suppose that $G_n(s)\ne0$. If $F_n(s)=0$, the desired coefficient
estimate is trivial. Otherwise choose $x_0\in E_n$ with
$\|x_0\|=1$ and $F_n(s)(x_0)=\|F_n(s)\|$.  Applying numerical
positivity to $(F_n(s)/\|F_n(s)\|,x_0)\in\Pi(E_n^*)$ gives
$F_n(s)(G_n(s))\ge0$.  If equality held, then $R_{n,s}^2=0$.  The exponential
characterization of numerical positivity
\cite[Section~3]{KR} would then give
\[
   \|I-tR_{n,s}\|=\|\exp(-tR_{n,s})\|\le1,
   \qquad t\ge0,
\]
which is impossible for a nonzero $R_{n,s}$ as $t\to\infty$.
It follows that
\[
   \varphi\longmapsto
   \frac{\varphi(G_n(s))}{F_n(s)(G_n(s))}F_n(s)
\]
is a numerically positive rank-one projection.  Thus $F_n(s)$ is a
Flinn element of $E_n^*$.  By
\eqref{eq:local-flinn-coefficients},
\begin{equation}
   \left(\sum_{k=1}^{2^n}|f_{n,k}(s)|^p\right)^{1/p}
   \le A_{p,I_0}\max_k|f_{n,k}(s)|
   \label{eq:pointwise-flinn}
\end{equation}
whenever $s\notin N$ and $G_n(s)\ne0$.

Let
\[
   D_n=(B_n\cup N)^c.
\]
The sets $D_n$ increase with $n$.  On $D_n$, the lattice property,
\eqref{eq:pointwise-flinn}, and Krivine's inequality
\cite[Theorem~1.f.14, p.~93]{LT} give
\begin{equation}
 \begin{aligned}
 \norm{\chi_{D_n}
 \left(\sum_k|f_{n,k}|^p\right)^{1/p}}_Y
 &\le A_{p,I_0}
    \norm{\left(\sum_k|f_{n,k}|^2\right)^{1/2}}_Y\\
 &\le K_GA_{p,I_0}
    \norm{\left(\sum_k|e_{n,k}|^2\right)^{1/2}}_Y\\
 &=K_GA_{p,I_0}\norm{\chi_{I_0}}_Y.
 \end{aligned}
 \label{eq:fixed-window-level}
\end{equation}
The constant on the right is independent of $n$ and $U$.

Since
\[
   e_{n,k}=e_{n+1,2k-1}+e_{n+1,2k},
\]
linearity of $U$ and the inequality $|a+b|^p\le|a|^p+|b|^p$ show that
\[
   \sum_{k=1}^{2^n}|f_{n,k}|^p
   \le\sum_{k=1}^{2^{n+1}}|f_{n+1,k}|^p.
\]
The functions
\[
   H_n=\chi_{D_n}
       \left(\sum_{k=1}^{2^n}|f_{n,k}|^p\right)^{1/p}
\]
form an increasing sequence. By \Cref{lem:dyadic-detection}, every
$f_{n,k}$ vanishes on $B=\bigcap_nB_n$, and hence
\[
   H_n\uparrow
   \sup_{n\ge1}
   \left(\sum_{k=1}^{2^n}|f_{n,k}|^p\right)^{1/p}
\]
almost everywhere.  The space $Y$ has the Fatou property: this is part
of the Banach function-space convention in the direct case and is
automatic for the K\"othe dual $X^\times$ in the weak-star case.
Applying the Fatou property to \eqref{eq:fixed-window-level} gives
\eqref{eq:fixed-window}, with
\[
   C_{p,I_0}=K_GA_{p,I_0}\norm{\chi_{I_0}}_Y.
\]
\end{proof}

A measure-preserving transport now moves the estimate from $I_0$ to
any set of finite measure.

\begin{lemma}
\label{lem:arbitrary-window}
Let $B\subset(0,\infty)$ have finite measure and let $0<p\le1$.
There is a constant $C_{p,B}$, independent of the isometry $U$ in the
relevant direct or weak-star class, and a positive function
$H_{U,p,B}$ in the space on which $U$ acts such that
\begin{equation}
   |Uf|\le \|f\|_\infty H_{U,p,B}
   \quad\text{for every bounded }f\text{ supported in }B.
   \label{eq:controllable-window}
\end{equation}
Moreover,
\[
   \|H_{U,p,B}\|\le C_{p,B}.
\]
\end{lemma}

\begin{proof}
Let $Y=X$ in the direct case and $Y=X^\times$ in the weak-star case.
Choose a measurable partition
\[
   B=B_1\sqcup\cdots\sqcup B_q,
   \qquad m(B_r)\le R_0.
\]
For each $r$, choose a measurable set $\widetilde B_r\supset B_r$ with
$m(\widetilde B_r)=R_0$.  Since both $I_0$ and $\widetilde B_r$ have
measure $R_0$, and their complements have infinite measure, there is a
measure-preserving automorphism
\[
   \theta_r:(0,\infty)\longrightarrow(0,\infty)
\]
such that $\theta_r(I_0)=\widetilde B_r$ modulo null sets.  The operator
\[
   W_r:Y\longrightarrow Y,
   \qquad W_rg=g\circ\theta_r,
\]
is a surjective isometry; in the weak-star case it is weak-star
continuous and its preadjoint is composition with $\theta_r^{-1}$.

Fix $r$ and let $g$ be bounded and supported in $B_r$.  Put $h=W_rg$.
Then $h$ is supported in $I_0$.  Let
\[
   h_n=\sum_{k=1}^{2^n}c_{n,k}e_{n,k}
\]
be its dyadic conditional expectation on $I_0$.  We have
$|c_{n,k}|\le\|g\|_\infty$, and therefore, since $0<p\le1$,
\[
\begin{aligned}
   |UW_r^{-1}h_n|
   &\le \sum_{k=1}^{2^n}|c_{n,k}|\,|UW_r^{-1}e_{n,k}|\\
   &\le \|g\|_\infty
      \left(\sum_{k=1}^{2^n}|UW_r^{-1}e_{n,k}|^p\right)^{1/p}.
\end{aligned}
\]
The operator $UW_r^{-1}:Y\to Y$ is a surjective isometry in the same
class as $U$.  Hence \Cref{thm:fixed-window} gives
\[
   H_r:=\sup_{n\ge1}
   \left(\sum_{k=1}^{2^n}|UW_r^{-1}e_{n,k}|^p\right)^{1/p}
   \in Y_+,
   \qquad
   \|H_r\|_Y\le C_{p,I_0}.
\]
We have
\begin{equation}
   |UW_r^{-1}h_n|\le\|g\|_\infty H_r.
   \label{eq:window-dyadic-domination}
\end{equation}

In the direct case, $h_n\to h$ in $Y$ by order continuity, and a
subsequence of $UW_r^{-1}h_n$ converges almost everywhere to $Ug$.
In the weak-star case, $h_n\to h$ in $\sigma(X^\times,X)$.  Since
$UW_r^{-1}$ is weak-star continuous, Proposition~2.1 of \cite{KR}
implies that, for every $0<v\in X$,
\[
   \int v\,|UW_r^{-1}(h_n-h)|\,dm\longrightarrow0.
\]
Taking $v$ strictly positive and passing to a subsequence gives
almost-everywhere convergence.  In either case, passage to the limit
in \eqref{eq:window-dyadic-domination} yields
\[
   |Ug|\le\|g\|_\infty H_r.
\]

For a bounded function $f$ supported in $B$,
\[
   |Uf|
   \le\sum_{r=1}^q|U(f\chi_{B_r})|
   \le\|f\|_\infty\sum_{r=1}^qH_r.
\]
Thus one may take
\[
   H_{U,p,B}=\sum_{r=1}^qH_r,
   \qquad
   C_{p,B}=qC_{p,I_0}.
\]
\end{proof}

For property $(P')$, one also needs bounded continuity in measure.

\begin{lemma}
\label{lem:wstar-measure-cont}
Let $U:X^\times\to X^\times$ be a weak-star continuous surjective
isometry.  For every set $B$ of finite measure, the operator
\[
   UP_B:L_\infty(B)\longrightarrow L_0(0,\infty)
\]
is continuous from bounded convergence in measure on $B$ to local
convergence in measure on $(0,\infty)$.
\end{lemma}

\begin{proof}
Let $(f_n)\subset L_\infty(B)$ be uniformly bounded and suppose that
$f_n\to0$ in measure.  Every subsequence has a further subsequence
converging to zero almost everywhere.  Along such a subsequence,
dominated convergence gives $f_n\to0$ in $\sigma(X^\times,X)$.
Proposition~2.1 of \cite{KR}, applied to the weak-star continuous
operator $U$, then gives
\[
   \int h|Uf_n|\,dm\longrightarrow0
   \qquad(0<h\in X).
\]
Choose $h\in X$ strictly positive almost everywhere.  On a set $A$ of
finite measure, first choose $\delta>0$ so that
$m(A\cap\{h<\delta\})$ is small and then apply Markov's inequality on
$A\cap\{h\ge\delta\}$.  Hence $Uf_n\to0$ locally in measure along the
chosen subsequence.  The usual subsequence criterion then gives local convergence in
measure for the original sequence.
\end{proof}

Controllability and bounded continuity in measure give a local
random-measure representation.

\begin{proposition}
\label{prop:local-random-measure}
Let $B\subset(0,\infty)$ have finite measure.  For every isometry $U$ in
the direct or weak-star class there is a weak-star Borel family
$(\nu_s^{U,B})_s$ of finite signed Borel measures on $B$ such that
\begin{equation}
   Uf(s)=\int_B f(t)\,d\nu_s^{U,B}(t)
   \label{eq:local-kernel-formula}
\end{equation}
for every bounded measurable $f$ supported in $B$.  The family is
unique up to a null set in the parameter $s$.  Moreover, if
$N\subset B$ is null, then
\begin{equation}
   |\nu_s^{U,B}|(N)=0
   \qquad\text{for almost every }s.
   \label{eq:null-set-property}
\end{equation}
If $B\subset C$ and both sets have finite measure, then
\begin{equation}
   \nu_s^{U,C}|_B=\nu_s^{U,B}
   \qquad\text{for almost every }s.
   \label{eq:kernel-compatibility}
\end{equation}
\end{proposition}

\begin{proof}
Set $T_B=UP_B$.  By \Cref{lem:arbitrary-window}, $T_B$ is controllable.
In the direct case, bounded convergence in measure on $B$ implies
convergence in $X$: every subsequence has an almost-everywhere
convergent further subsequence dominated by a multiple of $\chi_B$,
and order continuity applies.  In the weak-star case the required
measure continuity is \Cref{lem:wstar-measure-cont}.  The random-measure
representation theorem \cite[Theorem~3.1]{KaltonRep} (see also
\cite{Weis} and \cite[Section~6]{KR}) therefore yields the weak-star
Borel family in \eqref{eq:local-kernel-formula}.  The representation theorem also gives uniqueness and the
total-variation absolute continuity in \eqref{eq:null-set-property}.
When $B\subset C$, both sides of \eqref{eq:kernel-compatibility}
represent the same operator on $L_\infty(B)$, so uniqueness identifies
them.
\end{proof}

Choose $I_n=(0,n]$.  After removing one null set, the compatibilities
\eqref{eq:kernel-compatibility} hold simultaneously for all pairs
$I_m\subset I_n$.  Write $\nu_s^U$ for the resulting locally
finite signed kernel, so that
\[
   \nu_s^U|_{I_n}=\nu_s^{U,I_n}.
\]
For an arbitrary set $B$ of finite measure, the notation
$\nu_s^U|_B$ refers to the unique local kernel supplied by
\Cref{prop:local-random-measure}; this convention is consistent on all
intersections of finite-measure sets.

For a finite signed measure $\mu$ and $0<p\le1$, define
\[
\|\mu\|_p
=
\sup\left\{
\left(\sum_{k=1}^n |\mu(A_k)|^p\right)^{1/p}
\,\middle|\,
\begin{array}{c}
n\in\NN,\quad A_1,\ldots,A_n\text{ measurable},\\
A_k\cap A_l=\varnothing\quad(k\ne l)
\end{array}
\right\}.
\]
For pairwise disjoint $B_1,\ldots,B_q$,
\begin{equation}
   \left\|\mu\bigm|_{\bigcup_{r=1}^qB_r}\right\|_p^p
   \le\sum_{r=1}^q\|\mu|_{B_r}\|_p^p.
   \label{eq:pvariation-disjoint-restrictions}
\end{equation}
For a disjoint family $(A_k)$,
\[
   |\mu(A_k)|^p
   =\left|\sum_r\mu(A_k\cap B_r)\right|^p
   \le\sum_r|\mu(A_k\cap B_r)|^p,
\]
and taking suprema gives \eqref{eq:pvariation-disjoint-restrictions}.

\begin{corollary}
\label{cor:pvariation-window}
Let $B\subset(0,\infty)$ have finite measure and let $0<p\le1$.
Then there are a constant $C_{p,B}$, independent of $U$, and a positive
function $H_{U,p,B}$ in the range space such that
\begin{equation}
   \|\nu_s^U|_B\|_p\le H_{U,p,B}(s)
   \quad\text{for almost every }s,
   \qquad
   \|H_{U,p,B}\|\le C_{p,B}.
   \label{eq:pvar-control}
\end{equation}
\end{corollary}

\begin{proof}
Use the sets $B_r\subset\widetilde B_r$ and the global rearrangement
isometries $W_r$ from the proof of \Cref{lem:arbitrary-window}.  The
pullback of $\nu_s^U|_{\widetilde B_r}$ under $\theta_r^{-1}$ is the
representing measure of $UW_r^{-1}$ on $I_0$, and its value on the
dyadic interval $I_{n,k}$ is
\[
   UW_r^{-1}e_{n,k}(s).
\]
Lemma~6.2 and the proof of Proposition~6.3 in \cite{KR} therefore give
\[
   \|\nu_s^U|_{\widetilde B_r}\|_p\le H_r(s),
   \qquad
   \|H_r\|\le C_{p,I_0},
\]
where $H_r$ is the function occurring in the proof of
\Cref{lem:arbitrary-window}.  Hence
\[
   \|\nu_s^U|_{B_r}\|_p\le H_r(s).
\]
By \eqref{eq:pvariation-disjoint-restrictions},
\[
   \|\nu_s^U|_B\|_p
   \le\left(\sum_{r=1}^qH_r(s)^p\right)^{1/p}.
\]
Set
\[
   H_{U,p,B}=\left(\sum_{r=1}^qH_r^p\right)^{1/p}.
\]
For $p<1$,
\[
   H_{U,p,B}\le q^{1/p-1}\sum_{r=1}^qH_r,
\]
and for $p=1$ the same estimate holds with constant one. This is
\eqref{eq:pvar-control}.
\end{proof}

\begin{lemma}
\label{lem:pvariation-atomic}
Let $0<p<1$ and let $\mu$ be a finite signed Borel measure on a Borel
subset of the real line.  Then $\|\mu\|_p<\infty$ if and only if
\[
   \mu=\sum_{j\ge1}c_j\delta_{t_j},
\]
where the points $t_j$ are distinct and $\sum_j|c_j|^p<\infty$.  In
this case
\begin{equation}
   \|\mu\|_p^p=\sum_j|c_j|^p.
   \label{eq:pvariation-atomic-formula}
\end{equation}
\end{lemma}

\begin{proof}
Suppose that the atomless part of $\mu$ is nonzero.  By the Jordan
decomposition, after restricting to a measurable set $D$, we obtain a
nonzero finite atomless positive measure $\lambda$ such that
$\mu|_D=\lambda$ or $\mu|_D=-\lambda$.  Put $a=\lambda(D)>0$.  For every
$n$, atomlessity gives a partition
\[
   D=D_1\sqcup\cdots\sqcup D_n,
   \qquad \lambda(D_k)=a/n.
\]
For this partition,
\[
   \|\mu\|_p^p\ge\sum_{k=1}^n|\mu(D_k)|^p=a^pn^{1-p}.
\]
Letting $n\to\infty$ shows that $\|\mu\|_p=\infty$.  Thus finite
$p$-variation forces $\mu$ to be purely atomic.

Conversely, suppose that $\mu=\sum_jc_j\delta_{t_j}$ with distinct
$t_j$.  If $A_1,\ldots,A_n$ are pairwise disjoint, then
\[
\begin{aligned}
   \sum_{k=1}^n|\mu(A_k)|^p
   &\le\sum_{k=1}^n\sum_{t_j\in A_k}|c_j|^p
    \le\sum_j|c_j|^p,
\end{aligned}
\]
because $0<p<1$.  Taking $A_k=\{t_k\}$ for the first $n$ atoms and letting
$n\to\infty$ gives the reverse inequality. Hence
\eqref{eq:pvariation-atomic-formula} holds.
\end{proof}

\begin{proposition}
\label{prop:atomic-kernel-estimates}
Fix $0<p<1$ and let
\[
   Y=
   \begin{cases}
      X,&\text{in the direct case},\\
      X^\times,&\text{in the weak-star case}.
   \end{cases}
\]
Every surjective isometry $U:Y\to Y$ in the corresponding class has a
representation
\begin{equation}
   \nu_s^U=\sum_{j=1}^\infty
   a_j^U(s)\delta_{\sigma_j^U(s)},
   \label{eq:atomic-kernel}
\end{equation}
where $a_j^U$ and $\sigma_j^U$ are measurable and, for almost every
$s$, the points corresponding to nonzero coefficients are pairwise
distinct.  If $N$ is null, then
\begin{equation}
   m\{s:a_j^U(s)\ne0,\ \sigma_j^U(s)\in N\}=0
   \qquad(j\ge1).
   \label{eq:atom-null-property}
\end{equation}
For every set $B$ of finite measure,
\begin{equation}
   \|\nu_s^U|_B\|_p^p
   =\sum_{\sigma_j^U(s)\in B}|a_j^U(s)|^p
   \quad\text{for almost every }s.
   \label{eq:atomic-pmass}
\end{equation}
Consequently, for all sets $A,B$ of finite measure,
\begin{equation}
   \int_A\sum_{\sigma_j^U(s)\in B}|a_j^U(s)|^p\,ds<\infty,
   \label{eq:local-pmass-finite}
\end{equation}
and
\begin{equation}
   \sup_U
   \int_A\sum_{\sigma_j^U(s)\in B}|a_j^U(s)|^p\,ds<\infty,
   \label{eq:uniform-ceiling}
\end{equation}
where the supremum is taken over the corresponding class of
surjective isometries.
\end{proposition}

\begin{proof}
Let $D_n=(n-1,n]$.  By \Cref{cor:pvariation-window}, the random measure
$\nu_s^U|_{D_n}$ has finite $p$-variation for almost every $s$ and is
therefore purely atomic.  The measurable enumeration theorem for
purely atomic random measures \cite[Theorem~4.1]{Weis} gives measurable
atom locations and masses on each $D_n$.  Concatenating these
enumerations over the pairwise disjoint sets $D_n$ gives
\eqref{eq:atomic-kernel}; zero coefficients may be padded by an
arbitrary fixed point.  Formula \eqref{eq:atomic-pmass} follows from
\Cref{lem:pvariation-atomic}.

If $N$ is null, \eqref{eq:null-set-property} gives
$|\nu_s^U|(N)=0$ almost everywhere.  A nonzero term with
$\sigma_j^U(s)\in N$ would contribute a nonzero atom to this
restriction, proving \eqref{eq:atom-null-property}.

Put
\[
   K_{U,p,B}(s)=\|\nu_s^U|_B\|_p.
\]
By \Cref{cor:pvariation-window},
$K_{U,p,B}\le H_{U,p,B}$, so its norm in $Y$ is bounded independently
of $U$.  In the direct case, concavity of $t\mapsto t^p$ and K\"othe
duality give
\[
\begin{aligned}
   \int_AK_{U,p,B}^p\,dm
   &\le m(A)^{1-p}\left(\int_AK_{U,p,B}\,dm\right)^p\\
   &\le m(A)^{1-p}\|\chi_A\|_{X^\times}^p
      \|K_{U,p,B}\|_X^p.
\end{aligned}
\]
In the weak-star case the corresponding estimate is
\[
   \int_AK_{U,p,B}^p\,dm
   \le m(A)^{1-p}\|\chi_A\|_X^p
      \|K_{U,p,B}\|_{X^\times}^p.
\]
The right-hand sides are bounded independently of $U$.  Combining
these estimates with \eqref{eq:atomic-pmass} proves
\eqref{eq:local-pmass-finite} and \eqref{eq:uniform-ceiling}.
\end{proof}

In the weak-star case, all operator identities may be checked on bounded
functions of finite support. This class is
$\sigma(X^\times,X)$-dense in $X^\times$, and the operators used below
are weak-star continuous, so the identities extend to the full space.

For later reference, let $\theta$ be an invertible
measure-preserving transformation and set
\[
   Wg=g\circ\theta,
\]
Then
\[
   \langle Wg,f\rangle
   =
   \langle g,f\circ\theta^{-1}\rangle,
   \qquad
   g\in X^\times,\ f\in X.
\]
Thus $W$ is weak-star continuous, and its preadjoint is
\[
   f\longmapsto f\circ\theta^{-1}.
\]
This is the preadjoint used below.

\section{Composition of representing kernels and a recurrence argument}
\label{app:kernel-composition}
\label{sec:kernel-composition}

We turn to the two remaining ingredients in
\Cref{thm:at-most-one-atom}: composition of local kernels and the
recurrence argument. The composition estimate is most conveniently
written on three labelled copies
$(\Omega_r,m_r)$, $r=0,1,2$, of the same standard atomless measure
space $(\Omega,m)$.  For each $r=0,1,2$, let $Y_r$ denote the
corresponding copy of $X$ in the norm-continuous case, and the
corresponding copy of $X^\times$, equipped with
$\sigma(X^\times,X)$, in the weak-star case.

Let
\[
 R:Y_0\to Y_1,
 \qquad
 Q:Y_1\to Y_2
\]
be surjective isometries belonging to the corresponding class of
\Cref{prop:atomic-kernel-estimates}. Write the representing kernels of
$R$ and $Q$ as
\[
 \nu_y^R
 =
 \sum_{j\ge1}\alpha_j(y)\delta_{\phi_j(y)},
 \qquad
 \nu_z^Q
 =
 \sum_{k\ge1}\beta_k(z)\delta_{\psi_k(z)}.
\]
The functions $\phi_j,\psi_k$ give the atom locations and
$\alpha_j,\beta_k$ their coefficients.  In each kernel the nonzero
atoms are listed without repetition.

Let
\[
 A_1,\dots,A_I\subset\Omega_2,
 \quad
 B_1,\dots,B_N\subset\Omega_1,
 \quad
 C_1,\dots,C_L\subset\Omega_0
\]
be finite families of pairwise disjoint sets of finite positive
measure. Put $a_i=m(A_i)$ and $b_r=m(B_r)$ and define
\begin{align}
 q_{ir}
 &=\int_{A_i}\sum_{k\ge1}|\beta_k(z)|^p
   \one_{B_r}(\psi_k(z))\,dz,
 \label{eq:qir}\\
 r_{r\ell}
 &=\int_{B_r}\sum_{j\ge1}|\alpha_j(y)|^p
   \one_{C_\ell}(\phi_j(y))\,dy.
 \label{eq:rrl}
\end{align}
Define the matrices
\[
 \Mp(Q)_{ir}=\frac{q_{ir}}{a_i},
 \qquad
 \Mp(R)_{r\ell}=\frac{r_{r\ell}}{b_r}.
\]
With this normalization,
\begin{equation}
 [\Mp(Q)\Mp(R)]_{i\ell}
 =\frac1{a_i}\sum_{r=1}^N\frac{q_{ir}r_{r\ell}}{b_r}.
 \label{eq:matrix-product}
\end{equation}
For a measure-preserving automorphism $\tau$ of $\Omega_1$, let
$V_\tau h=h\circ\tau$.

\begin{theorem}
\label{thm:kernel-composition}
Let $0<p<1$. For every $\varepsilon>0$, there is a
measure-preserving automorphism $\tau$ such that
\[
 \tau(B_r)=B_r\quad(1\le r\le N),
\]
$\tau$ is the identity outside $\bigcup_rB_r$, and
\begin{equation}
 \Mp(QV_\tau R)_{i\ell}
 \ge[\Mp(Q)\Mp(R)]_{i\ell}-\varepsilon
 \label{eq:kernel-composition}
\end{equation}
for every $i,\ell$. The same assertion holds in the weak-star
continuous isometry class.
\end{theorem}

Put $B=\bigcup_rB_r$ and $C=\bigcup_\ell C_\ell$. We first identify
the restriction to $C$ of the kernel of $QP_BV_\tau R$.

\begin{lemma}
\label{lem:composition}
Let $(\mathcal G,m_{\mathcal G})$ be a probability space carrying a
jointly measurable family of measure-preserving automorphisms of
$\Omega_1$. Assume that every automorphism leaves each $B_r$ invariant,
acts as the identity outside $B=\bigcup_rB_r$, and that
\begin{equation}
 \int_{\mathcal G}h(\tau y)\,dm_{\mathcal G}(\tau)
 =\frac1{b_r}\int_{B_r}h(u)\,du
 \quad\text{for almost every }y\in B_r
 \label{eq:orbit-identity}
\end{equation}
for every $h\in L_1(B_r)$. Put $C=\bigcup_\ell C_\ell$. Then there is a
measurable set $\mathcal G_{\rm comp}\subseteq\mathcal G$ of full measure
such that, for every $\tau\in\mathcal G_{\rm comp}$,
\begin{equation}
 \sum_{r=1}^N\sum_{\psi_k(z)\in B_r}
 \sum_{\phi_j(\tau\psi_k(z))\in C}
 |\beta_k(z)|\,|\alpha_j(\tau\psi_k(z))|<\infty
 \label{eq:absolute-composition}
\end{equation}
for almost every $z\in\Omega_2$, and
\begin{equation}
 \kappa_z^\tau
 :=\sum_{r=1}^N\sum_{\psi_k(z)\in B_r}\sum_{j\ge1}
 \beta_k(z)\alpha_j(\tau\psi_k(z))
 \delta_{\phi_j(\tau\psi_k(z))}|_C
 \label{eq:composition-kernel}
\end{equation}
defines a measurable finite signed kernel on $C$ representing
$QP_BV_\tau RP_C$. More precisely, for each
$\tau\in\mathcal G_{\rm comp}$ there is a null set
$N_\tau\subseteq\Omega_2$, independent of the test function, such that
$\kappa_z^\tau$ agrees on $\Omega_2\setminus N_\tau$ with a fixed local
representing kernel of $QP_BV_\tau RP_C$.
\end{lemma}

\begin{proof}
For $1\le r\le N$, put
\[
 L_{Q,r}(z)=\sum_{\psi_k(z)\in B_r}|\beta_k(z)|,
 \qquad
 L_{R,C}(y)=\sum_{\phi_j(y)\in C}|\alpha_j(y)|.
\]
Since $0<p<1$,
\[
 L_{Q,r}\le
 \left(\sum_{\psi_k(\cdot)\in B_r}|\beta_k|^p\right)^{1/p},
 \qquad
 L_{R,C}\le
 \left(\sum_{\phi_j(\cdot)\in C}|\alpha_j|^p\right)^{1/p}.
\]
By \Cref{prop:atomic-kernel-estimates}, the functions on the right
belong to the corresponding range spaces. Hence, for every set $D$ of
finite measure,
\[
 \int_D L_{Q,r}<\infty,
 \qquad
 \int_{B_r}L_{R,C}<\infty.
\]
Using \eqref{eq:orbit-identity} and Tonelli's theorem, we obtain
\[
\begin{aligned}
 &\int_{\mathcal G}\int_D
 \sum_{r=1}^N\sum_{\psi_k(z)\in B_r}
 |\beta_k(z)|L_{R,C}(\tau\psi_k(z))\,dz\,dm_{\mathcal G}(\tau)\\
 &\qquad=
 \sum_{r=1}^N\frac1{b_r}
 \left(\int_D L_{Q,r}(z)\,dz\right)
 \left(\int_{B_r}L_{R,C}(u)\,du\right)<\infty.
\end{aligned}
\]
Applying this identity to a countable exhaustion of $\Omega_2$ by sets
of finite measure and then using Fubini's theorem gives a set
$\mathcal G_{\rm comp}$ of full measure on which
\eqref{eq:absolute-composition} holds almost everywhere in $z$.
For every $\tau\in\mathcal G_{\rm comp}$, the series in
\eqref{eq:composition-kernel} is therefore absolutely convergent in
total variation for almost every $z$ and defines a measurable finite
signed kernel on $C$.

To make the exceptional set independent of the test function, let
\[
 S_B=I-2P_B.
\]
The operators $QV_\tau R$ and $QS_BV_\tau R$ are surjective isometries
in the same class as $Q$ and $R$. Choose their local representing
kernels on $C$ and put
\begin{equation}
 \lambda_z^\tau
 =\frac12\left(
   \nu_z^{QV_\tau R}|_C-\nu_z^{QS_BV_\tau R}|_C
 \right).
 \label{eq:comparison-kernel}
\end{equation}
Since $P_B=(I-S_B)/2$, the kernel $\lambda^\tau$ represents
$QP_BV_\tau RP_C$.

Fix $\tau\in\mathcal G_{\rm comp}$. After replacing $C$ by a Borel
representative, choose a countable algebra $\mathscr A$ generating its
Borel $\sigma$-algebra. Let $D\in\mathscr A$ and set $f=\one_D$. The
kernel formula for $R$ gives
\begin{equation}
 Rf(y)=\sum_{j\ge1}\alpha_j(y)\one_D(\phi_j(y))
 \label{eq:R-indicator-kernel}
\end{equation}
for almost every $y$, with absolute value bounded by $L_{R,C}(y)$.
The exceptional set in \eqref{eq:R-indicator-kernel} is null. Since
$\tau$ is measure preserving, its inverse image is null, and
\eqref{eq:atom-null-property} shows that, for almost every $z$, no
nonzero term $\psi_k(z)$ belongs to this inverse image.

Put $g=P_BV_\tau Rf$ and
\[
 g_n=P_B\bigl((-n)\vee((V_\tau Rf)\wedge n)\bigr).
\]
In the direct case, $g_n\to g$ in $Y_1$ by order continuity, and
hence $Qg_n\to Qg$ in norm; a subsequence converges almost everywhere.
In the weak-star case, $g_n\to g$ in $\sigma(X^\times,X)$ by dominated
convergence. The weak-star continuity of $Q$ and
\cite[Proposition~2.1]{KR}, applied with a strictly positive element of
$X$, again give an almost-everywhere convergent subsequence. Thus, in
either case,
\[
 Qg_n(z)\longrightarrow Qg(z)
 \quad\text{for almost every }z
\]
along a subsequence.
For every $n$, the local kernel formula for $Q$ gives
\[
 Qg_n(z)=\sum_{r=1}^N\sum_{\psi_k(z)\in B_r}
 \beta_k(z)g_n(\psi_k(z)).
\]
For almost every $z$ the terms on the right converge to
$\beta_k(z)(Rf)(\tau\psi_k(z))$ and are dominated by
\[
 |\beta_k(z)|L_{R,C}(\tau\psi_k(z)).
\]
Absolute convergence in \eqref{eq:absolute-composition} permits
dominated convergence term by term; together with
\eqref{eq:R-indicator-kernel}, this gives
\begin{equation}
 (QP_BV_\tau R\one_D)(z)=\kappa_z^\tau(D)
 \quad\text{for almost every }z.
 \label{eq:composition-on-algebra}
\end{equation}
The comparison kernel in \eqref{eq:comparison-kernel} satisfies
\[
 (QP_BV_\tau R\one_D)(z)=\lambda_z^\tau(D)
 \quad\text{for almost every }z.
\]
Countability of $\mathscr A$ leaves a single null set $N_\tau$ on
which all these identities may fail; outside it,
\[
 \kappa_z^\tau(D)=\lambda_z^\tau(D)
 \qquad(D\in\mathscr A)
\]
whenever $z\notin N_\tau$. Both sides are finite signed Borel measures
on $C$, so the monotone class theorem gives
\[
 \kappa_z^\tau=\lambda_z^\tau
 \qquad(z\notin N_\tau).
\]
The identification is valid in the weak-star case as well.
\end{proof}

\begin{lemma}
\label{lem:common-refinement}
Fix $K\in\NN$. For each $1\le r\le N$, there exists a sequence
$\{\cD_{r,n}\}_{n\ge1}$ of finite measurable partitions of $B_r$
such that:

\begin{enumerate}[label=\textup{(\alph*)}]
\item every member of $\cD_{r,n}$ has the same measure,
      $\cD_{r,n+1}$ refines $\cD_{r,n}$, and
      \[
      |\cD_{r,n}|\longrightarrow\infty;
      \]

\item for almost every $z$, if $1\le k\ne k'\le K$ satisfy
      \[
      \psi_k(z),\psi_{k'}(z)\in B_r
      \quad\text{and}\quad
      \psi_k(z)\ne\psi_{k'}(z),
      \]
      then there exists $n_0=n_0(z,k,k',r)$ such that, for every
      $n\ge n_0$, the points $\psi_k(z)$ and $\psi_{k'}(z)$ belong
      to different members of $\cD_{r,n}$;

\item if $\mathcal F_{r,n}$ denotes the $\sigma$-algebra generated by
      $\cD_{r,n}$, then for every $g\in L_1(B_r)$,
      \[
      \mathbb E(g\mid\mathcal F_{r,n})
      \longrightarrow g
      \qquad\text{in }L_1(B_r).
      \]
\end{enumerate}
\end{lemma}

\begin{proof}
Since $B_r$ has finite measure and the underlying measure space is
atomless and standard, we may identify $B_r$, modulo null sets, with
the interval
$(0,b_r)$, where $b_r=m(B_r)$.

Let $\theta_r:B_r\to(0,b_r)$ be a measure-preserving isomorphism.
For each $n$, define
\[
D_{r,n,j}
=
\theta_r^{-1}
\left(
\left[\frac{(j-1)b_r}{2^n},
      \frac{jb_r}{2^n}\right)
\right),
\qquad 1\le j\le2^n,
\]
and set
\[
\cD_{r,n}
=
\{D_{r,n,j}:1\le j\le2^n\}.
\]

The partition $\cD_{r,n+1}$ refines $\cD_{r,n}$, every member of
$\cD_{r,n}$ has measure $b_r2^{-n}$, and
\[
 |\cD_{r,n}|=2^n\longrightarrow\infty.
\]

Two distinct points of $(0,b_r)$ belong to different members of the
dyadic partition once $n$ is sufficiently large. Hence, outside the
null sets arising from the measure-space identification, property
\textup{(b)} follows for each of the finitely many pairs
$1\le k\ne k'\le K$.

The $\sigma$-algebras $\mathcal F_{r,n}$ increase to the measurable
$\sigma$-algebra of $B_r$ modulo null sets. L\'evy's upward theorem
gives
\[
 \mathbb E(g\mid\mathcal F_{r,n})
 \longrightarrow g
 \qquad\text{in }L_1(B_r)
\]
for every $g\in L_1(B_r)$.
\end{proof}

The averaging device is a finite product of circle groups together with
a finite permutation group.

\begin{lemma}
\label{lem:compact-randomization}
Let $P\ge2$ and let $D_1,\ldots,D_P$ be an equal-measure partition of
an atomless standard finite measure space $(B,m)$, where $m(B)=b$.
There are a compact probability space
$(\mathcal G_{\mathcal D},m_{\mathcal D})$, a jointly measurable
family of measure-preserving automorphisms $\tau$ of $B$ modulo null
sets, and conull Borel sets $D_q^0\subset D_q$ such that, with
$B_0=\bigcup_qD_q^0$:
\begin{enumerate}[label=\textup{(\roman*)}]
\item for every integrable Borel function $h$ and every $y\in B_0$,
\begin{equation}
 \int_{\mathcal G_{\mathcal D}}h(\tau y)\,dm_{\mathcal D}(\tau)
 =\frac1b\int_Bh\,dm;
 \label{eq:cell-orbit-average}
\end{equation}
\item if $y\in D_q^0$ and $y'\in D_{q'}^0$ with $q\ne q'$, then the
      distribution of $(\tau y,\tau y')$ is absolutely continuous
      with respect to $m\times m$; more precisely, for every
      nonnegative Borel function $F$ on $B\times B$,
\begin{equation}
 \int_{\mathcal G_{\mathcal D}}F(\tau y,\tau y')\,dm_{\mathcal D}(\tau)
 =\frac{P}{(P-1)b^2}
   \sum_{a\ne c}\int_{D_a\times D_c}F(u,v)\,dm(u)\,dm(v);
 \label{eq:two-point-orbit-average}
\end{equation}
\item if $0\le f,h\le M$ are constant on the partition sets, then
\begin{align}
 \int_{\mathcal G_{\mathcal D}}\int_Bf(y)h(\tau y)\,dm(y)
       \,dm_{\mathcal D}(\tau)
 &=\frac1b\left(\int_Bf\,dm\right)\left(\int_Bh\,dm\right),
 \label{eq:perm-mean}\\
 \operatorname{Var}_{\tau}
 \left(\int_Bf(y)h(\tau y)\,dm(y)\right)
 &\le\frac{b^2M^4}{P-1}.
 \label{eq:perm-var}
\end{align}
\end{enumerate}
\end{lemma}

\begin{proof}
Let $\mathbb T=\mathbb R/\mathbb Z$.  For every $q$, choose
conull Borel sets $D_q^0\subset D_q$ and $T_q^0\subset\mathbb T$ and a
measure-preserving Borel isomorphism
\[
   \iota_q:D_q^0\longrightarrow T_q^0,
\]
where $D_q$ is supplied with normalized measure; see
\cite[Theorem~9.2.2]{Bogachev}.  Extend $\iota_q$ and its inverse
arbitrarily on the null complements. Put
\[
   \mathcal G_{\mathcal D}=\mathbb T^P\times\mathfrak S_P
\]
and equip it with the product of Haar probability on $\mathbb T^P$
and normalized counting measure on the permutation group. For
$\theta=(\theta_1,\ldots,\theta_P)$, $\pi\in\mathfrak S_P$, and
$y\in D_q^0$, define
\[
   \tau_{\theta,\pi}(y)
   =\iota_{\pi(q)}^{-1}\bigl(\iota_q(y)+\theta_q\bigr),
\]
whenever the right-hand side lies in the chosen conull cores, and
define it arbitrarily on the remaining null set.  These maps give a
jointly measurable family of measure-preserving automorphisms modulo
null sets.  The arbitrary definitions do not affect any of the
integrals below.

For fixed $y\in D_q^0$, the index $\pi(q)$ is uniform and
$\theta_q$ is Haar distributed; hence
\eqref{eq:cell-orbit-average} holds. If $y\in D_q^0$ and
$y'\in D_{q'}^0$, where $q\ne q'$, then $\theta_q$ and
$\theta_{q'}$ are independent, while $(\pi(q),\pi(q'))$ is uniformly
distributed over the ordered pairs of distinct indices, which yields
\eqref{eq:two-point-orbit-average} and \textup{(ii)}.

If $f=\sum_qu_q\one_{D_q}$ and
$h=\sum_qv_q\one_{D_q}$, then
\[
   \int_Bf(y)h(\tau_{\theta,\pi}y)\,dm(y)
   =\frac bP\sum_{q=1}^Pu_qv_{\pi(q)}.
\]
Invariance of the uniform permutation measure gives the mean formula.
For the variance, sampling without replacement gives
\[
 \operatorname{Var}\left(\frac1P\sum_qu_qv_{\pi(q)}\right)
 =\frac1{P-1}
  \left(\frac1P\sum_q(u_q-\bar u)^2\right)
  \left(\frac1P\sum_q(v_q-\bar v)^2\right)
\]
and hence \eqref{eq:perm-var}.
\end{proof}

For fixed $K\in\NN$ and finite equal-measure partitions
$\mathcal D_r$ of $B_r$, write $\mathcal D=(\mathcal D_r)_{r=1}^N$ and
let $G_{K,\mathcal D}$ be the set of all $z\in\Omega_2$ such that,
whenever $1\le k\ne k'\le K$ and
\[
 \beta_k(z)\beta_{k'}(z)\ne0,
 \qquad
 \psi_k(z),\psi_{k'}(z)\in B_r,
\]
the two points $\psi_k(z)$ and $\psi_{k'}(z)$ belong to different
members of $\mathcal D_r$.

Two elementary measurability facts about the kernels will be used.

\begin{lemma}
\label{lem:coincidence}
For fixed $j,j'$, the set
\begin{equation}
 \mathcal C_{j,j'}
 =\{(u,v):\alpha_j(u)\alpha_{j'}(v)\ne0,
              \ \phi_j(u)=\phi_{j'}(v)\}
 \label{eq:coincidence-set}
\end{equation}
has product measure zero. If $z\mapsto\omega_z$ is a measurable
finite signed kernel on a standard Borel space $C$, then
\begin{equation}
 \operatorname{At}(\omega)
 =\{(z,x):|\omega_z|(\{x\})>0\}
 \label{eq:measurable-atom-set}
\end{equation}
is measurable, and every section $\operatorname{At}(\omega)_z$ is
countable.
\end{lemma}

\begin{proof}
The diagonal of a standard Borel space is Borel, so
$\mathcal C_{j,j'}$ is measurable. For fixed $v$, the set
\[
 \{u:\alpha_j(u)\ne0,\ \phi_j(u)=\phi_{j'}(v)\}
\]
is null by \eqref{eq:atom-null-property}. Fubini's theorem proves
the first assertion.

Realize $C$ as a Borel subset of a Polish space with metric $d$. The
total variations $|\omega_z|$ form a measurable positive kernel. A
monotone-class argument therefore shows that
\[
 (z,x)\longmapsto
 |\omega_z|\{t\in C:d(t,x)<1/n\}
\]
is measurable. Continuity from above gives
\[
 |\omega_z|(\{x\})
 =\lim_{n\to\infty}|\omega_z|\{t\in C:d(t,x)<1/n\},
\]
which proves measurability. A finite measure has at most countably
many atoms.
\end{proof}

\begin{lemma}
\label{lem:selected-atom-separation}
Fix $K,J\in\NN$, finite equal-measure partitions
$\mathcal D_r$ of $B_r$, and the product probability space
$(\mathcal G,m_{\mathcal G})$ formed from the probability spaces in
\Cref{lem:compact-randomization}, acting on each $B_r$ and as the
identity on $B^c$. Let $z\mapsto\omega_z$ be a measurable finite signed
kernel on $C$. Then there is a measurable set
$\mathcal P_{\rm sep}\subseteq\mathcal G$ of full measure such that, for
every $\tau\in\mathcal P_{\rm sep}$, the following assertions hold for
almost every $z\in G_{K,\mathcal D}$.

Suppose that $1\le k\le K$, $1\le j\le J$,
$\psi_k(z)\in B_r$, and
\[
 \beta_k(z)\alpha_j(\tau\psi_k(z))\ne0,
 \qquad
 x=\phi_j(\tau\psi_k(z))\in C.
\]
Then $x\notin\operatorname{At}(\omega)_z$. Moreover, if
$1\le k'\le K$, $j'\ge1$, $\psi_{k'}(z)\in B_s$, and
\[
 \beta_{k'}(z)\alpha_{j'}(\tau\psi_{k'}(z))\ne0,
 \qquad
 \phi_{j'}(\tau\psi_{k'}(z))=x,
\]
then $k'=k$ and $j'=j$. In particular, the locations of all selected
terms with $k\le K$ and $j\le J$ are pairwise distinct.
\end{lemma}

\begin{proof}
Let $B_r^0$ be the conull Borel cores in
\Cref{lem:compact-randomization}. By
\eqref{eq:atom-null-property}, after deleting a null set of $z$, every
nonzero point $\psi_k(z)\in B_r$, $k\le K$, belongs to $B_r^0$.

Consider first terms with different outer indices. Fix $k\ne k'\le K$ and $j,j'\ge1$. If
$\psi_k(z)$ and $\psi_{k'}(z)$ belong to the same $B_r$, then for
$z\in G_{K,\mathcal D}$ they lie in different cells of
$\mathcal D_r$. Formula \eqref{eq:two-point-orbit-average} and
\Cref{lem:coincidence} show that
\[
 \phi_j(\tau\psi_k(z))
 =\phi_{j'}(\tau\psi_{k'}(z))
\]
with both coefficients nonzero only on a set of $\tau$-measure zero.
If the two points belong to different sets $B_r$ and $B_s$, the same
conclusion follows from the product structure of $\mathcal G$ and the
first assertion of \Cref{lem:coincidence}. The relevant events are
measurable. Fubini's theorem and a countable union over
$k,k'\le K$ and $j,j'\ge1$ therefore give a set of full measure in
$\mathcal G$ on which no such coincidence occurs for almost every
$z\in G_{K,\mathcal D}$.

For a fixed outer index $k$, distinct nonzero terms already have
distinct locations for almost every point of $B_r$, by
\Cref{prop:atomic-kernel-estimates}. The exceptional subset of $B_r$ is
null; its inverse image under any fixed $\tau$ is null as well.
Equation~\eqref{eq:atom-null-property} for $Q$ then shows that almost no
nonzero $\psi_k(z)$ meets it. Hence no additional coincidence can come
from terms with the same outer index.

The selected locations must also be separated from the atoms of
$\omega_z$. By \Cref{lem:coincidence}, the set
$\operatorname{At}(\omega)$ is measurable and each section
$\operatorname{At}(\omega)_z$ is countable, hence null because the
underlying measure space is atomless. For fixed
$z,k,j,r$ with $\psi_k(z)\in B_r^0$, the orbit identity gives
\[
\begin{aligned}
 &m_{\mathcal G}\{\tau:\alpha_j(\tau\psi_k(z))\ne0,
       \ \phi_j(\tau\psi_k(z))\in\operatorname{At}(\omega)_z\}\\
 &\qquad=\frac1{b_r}m\{u\in B_r:\alpha_j(u)\ne0,
       \ \phi_j(u)\in\operatorname{At}(\omega)_z\}=0,
\end{aligned}
\]
where the last equality follows from
\eqref{eq:atom-null-property} for the kernel of $R$. Fubini's theorem
and a finite union over $k\le K$ and $j\le J$ give a further full-measure
subset of $\mathcal G$. Intersecting the two full-measure subsets proves
the assertion.
\end{proof}

\begin{proof}[Proof of \Cref{thm:kernel-composition}]
Put $A=\bigcup_{i=1}^IA_i$, $B=\bigcup_{r=1}^NB_r$, and
$C=\bigcup_{\ell=1}^LC_\ell$. Let
\[
 S_B=I-2P_B.
\]
Choose local kernels on $C$ for the two surjective isometries $QR$ and
$QS_BR$, and define
\begin{equation}
 \omega_z
 =\frac12\left(\nu_z^{QR}|_C+\nu_z^{QS_BR}|_C\right).
 \label{eq:outside-kernel}
\end{equation}
Since $P_{B^c}=(I+S_B)/2$, the measurable finite signed kernel $\omega$
represents the fixed operator $QP_{B^c}RP_C$. In particular, $\omega$
is independent of all parameters chosen below.

Fix $\varepsilon>0$ and put
\[
 a_*=\min_{1\le i\le I}a_i,
 \qquad
 \Delta=\frac{\varepsilon a_*}{6}.
\]
For $K,J\in\NN$, define
\begin{align*}
 q_{ir}^{(K)}
 &=\int_{A_i}\sum_{k\le K}|\beta_k(z)|^p
   \one_{B_r}(\psi_k(z))\,dz,
 &t_{ir}^{(K)}&=q_{ir}-q_{ir}^{(K)},\\
 r_{r\ell}^{(J)}
 &=\int_{B_r}\sum_{j\le J}|\alpha_j(y)|^p
   \one_{C_\ell}(\phi_j(y))\,dy,
 &s_{r\ell}^{(J)}&=r_{r\ell}-r_{r\ell}^{(J)}.
\end{align*}
Then $t_{ir}^{(K)}\downarrow0$, $s_{r\ell}^{(J)}\downarrow0$, and
\begin{equation}
 q_{ir}r_{r\ell}-q_{ir}^{(K)}r_{r\ell}^{(J)}
 =t_{ir}^{(K)}r_{r\ell}+q_{ir}^{(K)}s_{r\ell}^{(J)}.
 \label{eq:tail-identity}
\end{equation}

Choose $K$ so large that
\begin{equation}
 \sum_{r=1}^N\frac{t_{ir}^{(K)}r_{r\ell}}{b_r}
 <\frac{\Delta}{8IL}
 \qquad(1\le i\le I,\ 1\le\ell\le L).
 \label{eq:choose-K}
\end{equation}
Next choose $J$ so large that
\begin{equation}
 \sum_{r=1}^N\frac{q_{ir}^{(K)}s_{r\ell}^{(J)}}{b_r}
 <\Delta
 \qquad(1\le i\le I,\ 1\le\ell\le L).
 \label{eq:choose-J}
\end{equation}

For a measurable set $D\subseteq B_r$, put
\[
 \lambda_{ir}^{(K)}(D)
 =\int_{A_i}\sum_{k\le K}|\beta_k(z)|^p
   \one_D(\psi_k(z))\,dz.
\]
By \eqref{eq:atom-null-property},
$\lambda_{ir}^{(K)}\ll m|_{B_r}$. Let $f_{ir}^{(K)}$ be its
Radon--Nikodym derivative, so that
\[
 \int_{B_r}f_{ir}^{(K)}\,dm=q_{ir}^{(K)}.
\]
Also put
\[
 h_{r\ell}^{(J)}(y)
 =\sum_{j\le J}|\alpha_j(y)|^p\one_{C_\ell}(\phi_j(y)),
 \qquad
 \int_{B_r}h_{r\ell}^{(J)}\,dm=r_{r\ell}^{(J)}.
\]
For $M>0$, set
\[
 f_{ir}^{(K,M)}=f_{ir}^{(K)}\wedge M,
 \qquad
 h_{r\ell}^{(J,M)}=h_{r\ell}^{(J)}\wedge M,
\]
and
\[
 u_{ir}^{(K,M)}=\int_{B_r}(f_{ir}^{(K)}-M)_+\,dm,
 \qquad
 v_{r\ell}^{(J,M)}=\int_{B_r}(h_{r\ell}^{(J)}-M)_+\,dm.
\]
Choose $M$ so large that
\begin{equation}
 \sum_{r=1}^N\frac{
 u_{ir}^{(K,M)}r_{r\ell}^{(J)}
 +q_{ir}^{(K)}v_{r\ell}^{(J,M)}}{b_r}
 <\Delta
 \qquad(1\le i\le I,\ 1\le\ell\le L).
 \label{eq:choose-M}
\end{equation}
Truncation only decreases the two integrals, so
\begin{equation}
 \begin{aligned}
 \frac1{b_r}\left(\int f_{ir}^{(K,M)}\right)
 \left(\int h_{r\ell}^{(J,M)}\right)
 \ge{}&\frac{q_{ir}^{(K)}r_{r\ell}^{(J)}}{b_r}\\
 &-\frac{u_{ir}^{(K,M)}r_{r\ell}^{(J)}
       +q_{ir}^{(K)}v_{r\ell}^{(J,M)}}{b_r}.
 \end{aligned}
 \label{eq:amplitude-loss}
\end{equation}
Set
\[
 \eta=\frac{\Delta}{3N}.
\]

For each $r$, take the refining sequence of partitions in
\Cref{lem:common-refinement}. At a common level $n$, write
$\mathcal D^{(n)}=(\mathcal D_{r,n})_{r=1}^N$ and define
\[
 d_{ir}^{(K,n)}
 =\int_{A_i\setminus G_{K,\mathcal D^{(n)}}}
   \sum_{k\le K}|\beta_k(z)|^p
   \one_{B_r}(\psi_k(z))\,dz.
\]
For almost every $z$, the nonzero locations $\psi_k(z)$ with $k\le K$
are pairwise distinct by \Cref{prop:atomic-kernel-estimates}. Part
\textup{(b)} of \Cref{lem:common-refinement} then gives
$\one_{\Omega_2\setminus G_{K,\mathcal D^{(n)}}}\downarrow0$ almost
everywhere, and dominated convergence yields
\[
 d_{ir}^{(K,n)}\longrightarrow0.
\]
Part \textup{(c)} of the same lemma gives
\[
 \mathbb E(f_{ir}^{(K,M)}\mid\mathcal F_{r,n})
 \longrightarrow f_{ir}^{(K,M)},
 \qquad
 \mathbb E(h_{r\ell}^{(J,M)}\mid\mathcal F_{r,n})
 \longrightarrow h_{r\ell}^{(J,M)}
\]
in $L_1(B_r)$, while $|\mathcal D_{r,n}|\to\infty$. We may therefore
choose one common level such that, with
$\mathcal D_r=\mathcal D_{r,n}$,
$\mathcal F_r=\mathcal F_{r,n}$, and
$P_r=|\mathcal D_r|$,
\begin{equation}
 M\|f_{ir}^{(K,M)}-
     \mathbb E(f_{ir}^{(K,M)}\mid\mathcal F_r)\|_1<\frac\eta4,
 \qquad
 M\|h_{r\ell}^{(J,M)}-
     \mathbb E(h_{r\ell}^{(J,M)}\mid\mathcal F_r)\|_1<\frac\eta4
 \label{eq:conditional-errors}
\end{equation}
for all $i,r,\ell$,
\begin{equation}
 \frac{b_r^2M^4}{P_r-1}<\frac{\eta^2}{16INL}
 \qquad(1\le r\le N),
 \label{eq:variance-choice}
\end{equation}
and
\begin{equation}
 M\sum_{r=1}^Nd_{ir}^{(K,n)}<\Delta
 \qquad(1\le i\le I).
 \label{eq:choose-partition}
\end{equation}
Fix this level, write $\mathcal D=(\mathcal D_r)_{r=1}^N$,
$G=G_{K,\mathcal D}$, and
$d_{ir}^{(K,\mathcal D)}=d_{ir}^{(K,n)}$.

Restricting the measure defining $\lambda_{ir}^{(K)}$ to $A_i\cap G$
gives a density $f_{ir}^{(K,G)}\le f_{ir}^{(K)}$ satisfying
\begin{equation}
 \int_{B_r}\bigl(f_{ir}^{(K)}-f_{ir}^{(K,G)}\bigr)\,dm
 =d_{ir}^{(K,\mathcal D)}.
 \label{eq:good-loss}
\end{equation}
For each $r$, let
$(\mathcal G_r,m_{\mathcal G_r})$ and $B_r^0$ be supplied by
\Cref{lem:compact-randomization} for the fixed partition
$\mathcal D_r$. The product
\[
 (\mathcal G,m_{\mathcal G})
 =\prod_{r=1}^N(\mathcal G_r,m_{\mathcal G_r})
\]
acts on $\Omega_1$ by the corresponding automorphism on each $B_r$ and
by the identity on $B^c$. In particular,
\begin{equation}
 \int_{\mathcal G}g(\tau y)\,dm_{\mathcal G}(\tau)
 =\frac1{b_r}\int_{B_r}g\,dm
 \quad\text{for almost every }y\in B_r.
 \label{eq:orbit-identity-product}
\end{equation}

Put
\[
 \widetilde f_{ir}
 =\mathbb E(f_{ir}^{(K,M)}\mid\mathcal F_r),
 \qquad
 \widetilde h_{r\ell}
 =\mathbb E(h_{r\ell}^{(J,M)}\mid\mathcal F_r),
\]
and
\[
 Z_{ir\ell}(\tau)
 =\int_{B_r}\widetilde f_{ir}(y)
                  \widetilde h_{r\ell}(\tau y)\,dy.
\]
By \Cref{lem:compact-randomization},
\[
 \mathbb E Z_{ir\ell}
 =\frac1{b_r}\left(\int f_{ir}^{(K,M)}\right)
                    \left(\int h_{r\ell}^{(J,M)}\right),
\]
and \eqref{eq:variance-choice} gives
\[
 \operatorname{Var}Z_{ir\ell}<\frac{\eta^2}{16INL}.
\]
Chebyshev's inequality and a union bound give a measurable set
$\mathcal P_{\rm good}\subseteq\mathcal G$ with
\[
 m_{\mathcal G}(\mathcal P_{\rm good})>\frac78
\]
such that
\[
 Z_{ir\ell}(\tau)\ge\mathbb EZ_{ir\ell}-\eta
\]
for all $i,r,\ell$ and every $\tau\in\mathcal P_{\rm good}$. Moreover,
\eqref{eq:conditional-errors} implies
\[
 \left|
 \int f_{ir}^{(K,M)}(y)h_{r\ell}^{(J,M)}(\tau y)\,dy
 -Z_{ir\ell}(\tau)
 \right|<\frac\eta2.
\]
For every $\tau\in\mathcal P_{\rm good}$ this gives
\begin{equation}
 \begin{aligned}
 \int_{B_r}f_{ir}^{(K,M)}(y)
       h_{r\ell}^{(J,M)}(\tau y)\,dy
 \ge{}&\frac1{b_r}\left(\int f_{ir}^{(K,M)}\right)
                    \left(\int h_{r\ell}^{(J,M)}\right)-3\eta.
 \end{aligned}
 \label{eq:permutation-lower-bound}
\end{equation}

Define
\[
 S_{ir\ell}(\tau)
 =\int_{B_r}f_{ir}^{(K,G)}(y)h_{r\ell}^{(J)}(\tau y)\,dy.
\]
By the definition of $f_{ir}^{(K,G)}$,
\begin{equation}
 \begin{aligned}
 \sum_{r=1}^NS_{ir\ell}(\tau)
 ={}&\int_{A_i\cap G}
 \sum_{r=1}^N\sum_{\substack{k\le K\\\psi_k(z)\in B_r}}
 |\beta_k(z)|^p\\
 &\quad\cdot
 \sum_{j\le J}|\alpha_j(\tau\psi_k(z))|^p
 \one_{C_\ell}(\phi_j(\tau\psi_k(z)))\,dz.
 \end{aligned}
 \label{eq:selected-pmass}
\end{equation}
Furthermore,
\[
 S_{ir\ell}(\tau)
 \ge\int_{B_r}(f_{ir}^{(K,G)}\wedge M)(y)
                 h_{r\ell}^{(J,M)}(\tau y)\,dy.
\]
Since
\[
 \int_{B_r}\bigl((f_{ir}^{(K,M)}
                   -(f_{ir}^{(K,G)}\wedge M))_+\bigr)\,dm
 \le d_{ir}^{(K,\mathcal D)},
\]
relations \eqref{eq:amplitude-loss} and
\eqref{eq:permutation-lower-bound} imply that, for
$\tau\in\mathcal P_{\rm good}$,
\begin{equation}
 \begin{aligned}
 S_{ir\ell}(\tau)\ge{}&
 \frac{q_{ir}^{(K)}r_{r\ell}^{(J)}}{b_r}
 -\frac{u_{ir}^{(K,M)}r_{r\ell}^{(J)}
       +q_{ir}^{(K)}v_{r\ell}^{(J,M)}}{b_r}\\
 &-M d_{ir}^{(K,\mathcal D)}-3\eta.
 \end{aligned}
 \label{eq:S-lower}
\end{equation}

Put
\[
 H_{r\ell}(y)
 =\sum_{j\ge1}|\alpha_j(y)|^p
   \one_{C_\ell}(\phi_j(y))
\]
and
\[
 Y_{i\ell}^{(K)}(\tau)
 =\int_{A_i\cap G}
 \sum_{r=1}^N\sum_{\substack{k>K\\\psi_k(z)\in B_r}}
 |\beta_k(z)|^pH_{r\ell}(\tau\psi_k(z))\,dz.
\]
The orbit identity and Tonelli's theorem give
\begin{equation}
 \int_{\mathcal G}Y_{i\ell}^{(K)}(\tau)\,dm_{\mathcal G}(\tau)
 \le\sum_{r=1}^N\frac{t_{ir}^{(K)}r_{r\ell}}{b_r}.
 \label{eq:outer-tail-average}
\end{equation}
By \eqref{eq:choose-K}, Markov's inequality, and a union bound, the set
\[
 \mathcal P_{\rm tail}
 =\{\tau:Y_{i\ell}^{(K)}(\tau)\le\Delta
       \text{ for all }i,\ell\}
\]
satisfies
\[
 m_{\mathcal G}(\mathcal P_{\rm tail})>\frac78.
\]

Apply \Cref{lem:composition} to the fixed product space
$(\mathcal G,m_{\mathcal G})$. It gives a full-measure set
$\mathcal G_{\rm comp}$ such that $\kappa^\tau$ represents
$QP_BV_\tau RP_C$ for every $\tau\in\mathcal G_{\rm comp}$. Since every
$\tau$ acts as the identity on $B^c$, the kernel
$\kappa^\tau+\omega$ represents $QV_\tau RP_C$. By uniqueness of local
representing kernels, after changing it on a $z$-null set depending on
$\tau$,
\begin{equation}
 \nu_z^{QV_\tau R}|_C=\kappa_z^\tau+\omega_z
 \qquad(\tau\in\mathcal G_{\rm comp}).
 \label{eq:full-composition-kernel}
\end{equation}
The fixed kernel $\omega$, the integers $K,J$, and the partitions
$\mathcal D_r$ now satisfy the hypotheses of
\Cref{lem:selected-atom-separation}. Let
$\mathcal P_{\rm sep}\subseteq\mathcal G$ be the full-measure set given
there.

Fix $\tau\in\mathcal G_{\rm comp}\cap\mathcal P_{\rm sep}$. At every
selected location in \eqref{eq:selected-pmass}, the kernel
\eqref{eq:full-composition-kernel} contains the selected term, no other
term with outer index at most $K$, and no atom of $\omega_z$. Only terms with $k>K$ can therefore contribute to cancellation at
that location. For
$0<p\le1$,
\[
 |c+d|^p\ge|c|^p-|d|^p,
 \qquad
 \left|\sum_n d_n\right|^p\le\sum_n|d_n|^p.
\]
The selected locations are pairwise distinct, so each outer-tail term
is charged at most once. Using \eqref{eq:selected-pmass} and integrating
over $A_i$ gives
\begin{equation}
 a_i\Mp(QV_\tau R)_{i\ell}
 \ge\sum_{r=1}^NS_{ir\ell}(\tau)-Y_{i\ell}^{(K)}(\tau).
 \label{eq:kernel-composition-preerror}
\end{equation}
Combining \eqref{eq:tail-identity}, \eqref{eq:S-lower}, and
\eqref{eq:kernel-composition-preerror} yields, for
$\tau\in\mathcal P_{\rm good}\cap\mathcal G_{\rm comp}
\cap\mathcal P_{\rm sep}$,
\begin{equation}
 a_i\Mp(QV_\tau R)_{i\ell}
 \ge\sum_{r=1}^N\frac{q_{ir}r_{r\ell}}{b_r}
 -\mathcal E_{i\ell}(\tau),
 \label{eq:kernel-composition-error}
\end{equation}
where
\begin{align}
 \mathcal E_{i\ell}(\tau)={}&
 \sum_r\frac{t_{ir}^{(K)}r_{r\ell}}{b_r}
 +\sum_r\frac{q_{ir}^{(K)}s_{r\ell}^{(J)}}{b_r}
 \nonumber\\
 &+\sum_r\frac{u_{ir}^{(K,M)}r_{r\ell}^{(J)}
       +q_{ir}^{(K)}v_{r\ell}^{(J,M)}}{b_r}
 \nonumber\\
 &+M\sum_rd_{ir}^{(K,\mathcal D)}+3N\eta
 +Y_{i\ell}^{(K)}(\tau).
 \label{eq:total-error}
\end{align}

Since $\mathcal G_{\rm comp}$ and $\mathcal P_{\rm sep}$ have full
measure while $\mathcal P_{\rm good}$ and $\mathcal P_{\rm tail}$ each
have measure greater than $7/8$, their intersection is nonempty. Choose
$\tau$ in this intersection. By
\eqref{eq:choose-K}, \eqref{eq:choose-J}, \eqref{eq:choose-M},
\eqref{eq:choose-partition}, the definition of $\eta$, and the definition
of $\mathcal P_{\rm tail}$,
\[
 \mathcal E_{i\ell}(\tau)<6\Delta
 =\varepsilon a_*
 \le\varepsilon a_i
\]
for all $i,\ell$. Dividing \eqref{eq:kernel-composition-error} by $a_i$
and using \eqref{eq:matrix-product} proves
\eqref{eq:kernel-composition} simultaneously for all $i,\ell$.

In the weak-star case, $V_\tau$, $S_B$, and the band projections are
weak-star continuous, while \Cref{lem:composition} already uses the
weak-star local kernels. Every step above is therefore valid in that
setting as written.
\end{proof}

\subsection{A recurrence argument}
\label{sec:return}

Let $Y=X$ in the direct case and $Y=X^\times$ in the weak-star case.
Fix a surjective isometry $T:Y\to Y$ in the corresponding class and
put $S=T^{-1}$.  Regard the domain and range as labelled copies
$\Omega_E$ and $\Omega_F$, and write
\begin{align}
 \nu_s^T&=\sum_{j\ge1}a_j(s)\delta_{\sigma_j(s)},
 &s&\in\Omega_F,
 \label{eq:T-kernel}\\
 \nu_t^S&=\sum_{k\ge1}b_k(t)\delta_{\rho_k(t)},
 &t&\in\Omega_E.
 \label{eq:S-kernel}
\end{align}
Put
\[
 \Omega_T^{(2)}
 =\left\{s\in\Omega_F:
   \left|\{j\ge1:a_j(s)\ne0\}\right|\ge2\right\}.
\]
For $q\ge1$ and indices
\[
 j_1,k_1,j_2,k_2,\ldots,j_q,k_q,
\]
define, on the set where all the indicated coefficients are nonzero,
\[
 x_0(s)=s,\qquad
 x_{2r-1}(s)=\sigma_{j_r}(x_{2r-2}(s)),\qquad
 x_{2r}(s)=\rho_{k_r}(x_{2r-1}(s)).
\]

The role of the finite-measure recurrence in
\cite[Theorem~6.4]{KR} is played here by a local return lemma.
\begin{lemma}
\label{lem:return}
Let $B\subset\Omega_F$ have finite positive measure, let $j\ge1$, and
let $H\subset B$ have positive measure. Suppose that $a_j\ne0$ almost
everywhere on $H$. Then there are $q\ge1$, indices
\[
 j_1=j,k_1,j_2,k_2,\ldots,j_q,k_q,
\]
and a measurable set $H_0\subset H$ of positive measure such that, for
almost every $s\in H_0$,
\[
 a_{j_r}(x_{2r-2}(s))\,b_{k_r}(x_{2r-1}(s))\ne0
 \quad(1\le r\le q),
 \qquad x_{2q}(s)\in B.
\]
\end{lemma}

\begin{proof}
Assume the contrary. Put $\vartheta(r)=r/(1+r)$. For a finite positive
measure $\lambda\ll m_F$, define
\[
 (\cK_T\lambda)(A)
 =\int_{\Omega_F}\sum_{r\ge1}2^{-r}\vartheta(|a_r(s)|)
   \one_A(\sigma_r(s))\,d\lambda(s).
\]
For a finite positive measure $\mu\ll m_E$, define
\[
 (\cK_S\mu)(D)
 =\int_{\Omega_E}\sum_{r\ge1}2^{-r}\vartheta(|b_r(t)|)
   \one_D(\rho_r(t))\,d\mu(t).
\]
Both maps contract the cone of finite positive measures, and
\eqref{eq:null-set-property} shows that absolute continuity is
preserved.

Define
\[
 \mu_0(A)=\int_H\vartheta(|a_j(s)|)\one_A(\sigma_j(s))\,ds,
\]
and recursively
\[
 \lambda_n=\cK_S\mu_n,
 \qquad
 \mu_{n+1}=\cK_T\lambda_n.
\]
Set
\[
 \eta_E=\sum_{n\ge0}2^{-n-1}\mu_n,
 \qquad
 \eta_F=\sum_{n\ge0}2^{-n-1}\lambda_n.
\]
Let
\[
A=\car\eta_E,\qquad D=\car\eta_F,
\]
where $\car\mu$ denotes the carrier of $\mu$, namely the set on which
its Radon--Nikodym derivative with respect to the underlying measure is
positive, modulo null sets; see
\cite[Theorem~3.8, p.~90]{Folland}. Then
\[
 \eta_E\sim m_E|_A,
 \qquad
 \eta_F\sim m_F|_D.
\]
A positive density is strictly positive almost everywhere on its
carrier; the measure and the restricted Lebesgue measure therefore
have the same null sets.

The identities
\[
 \cK_S\eta_E=\eta_F,
 \qquad
 \cK_T\eta_F=2\eta_E-\mu_0
\]
and the positivity of all summands imply that, outside null sets, every
nonzero term of the kernel of $S$ with initial point in $A$ has its image
in $D$, while every nonzero term of the kernel of $T$ starting in $D$
has its image in $A$. This gives
\begin{equation}
 P_AS=P_ASP_D,
 \qquad
 P_DT=P_DTP_A.
 \label{eq:triangular-carriers}
\end{equation}
In the weak-star class these identities first hold on bounded
functions with support of finite measure and then extend by weak-star
continuity.

If $D\cap B$ had positive measure, then $\eta_F(B)>0$ and hence
$\lambda_n(B)>0$ for some $n$. Expanding $\lambda_n(B)$ gives a
countable sum of nonnegative integrals indexed by finite sequences of
indices occurring in the kernels of $T$ and $S$. One of these nonnegative integrals must be positive, producing the
indices and the set $H_0$ from the lemma. This contradicts the
assumption, so $B\cap D$ is null. After changing representatives on
a null set, we may assume that $B\subset D^c$.

Since $\mu_0(A^c)=0$ and
$\vartheta(|a_j(s)|)>0$ for almost every $s\in H$, the definition of
$\mu_0$ implies
\[
 \sigma_j(s)\in A
 \qquad\text{for almost every }s\in H.
\]
Because $H\subset B\subset D^c$, for almost every $s\in H$ the term
\[
 a_j(s)\delta_{\sigma_j(s)}
\]
occurs in the representing kernel of $P_{D^c}TP_A$. By
\Cref{prop:atomic-kernel-estimates}, the locations corresponding to
nonzero coefficients are pairwise distinct for almost every $s$.
No other term can cancel it. Since $a_j(s)\ne0$ almost everywhere on
$H$ and $m(H)>0$, uniqueness of the representing kernel yields
\[
 P_{D^c}TP_A\ne0.
\]
But \eqref{eq:triangular-carriers} and \Cref{lem:triangular} give
\[
 P_{D^c}TP_A=0,
\]
a contradiction.
\end{proof}

\begin{lemma}
\label{lem:carrier-propagation}
Let $G$ be measurable, let $x_0,\ldots,x_L$ be measurable maps obtained
from a fixed finite sequence of the maps $\sigma_j$ and $\rho_k$, and
write
\[
 x_{h+1}=\varphi_h\circ x_h \qquad(0\le h<L).
\]
Let $c_h$ be the coefficient associated with $\varphi_h$ and suppose
that
\[
 c_h(x_h(s))\ne0
 \quad\text{for almost every }s\in G,\qquad 0\le h<L.
\]
Put
\[
 \mu_h=(x_h)_*(m|_G),
 \qquad W_h=\car\mu_h.
\]
Then $\mu_h\sim m|_{W_h}$ and
\[
 c_h\ne0\quad\text{almost everywhere on }W_h,
 \qquad
 \varphi_h(W_h)\subset W_{h+1}
 \quad\text{modulo null sets}.
\]
\end{lemma}

\begin{proof}
Repeated use of \eqref{eq:null-set-property} gives $\mu_h\ll m$.
Writing $d\mu_h=q_h\,dm$ and
$W_h=\{q_h>0\}$ gives $\mu_h\sim m|_{W_h}$. By hypothesis,
$c_h(x_h(s))\ne0$ for almost every $s\in G$, hence
$\mu_h(\{c_h=0\})=0$ and therefore $c_h\ne0$ almost everywhere on
$W_h$. Also,
\[
 \mu_h\{x:\varphi_h(x)\notin W_{h+1}\}
 =\mu_{h+1}(W_{h+1}^c)=0,
\]
and equivalence with $m|_{W_h}$ proves the target inclusion.
\end{proof}

\begin{theorem}
\label{thm:finite-matrix-construction}
Suppose that $m(\Omega_T^{(2)})>0$. Then there exist pairwise
disjoint sets of finite measure
\[
 D_1,\dots,D_m\subset\Omega_F
\]
and pairwise disjoint sets of finite measure
\[
 C_1,\dots,C_n\subset\Omega_E,
\]
together with finite collections of selected terms from the kernels
of $T$ and $S$, and nonnegative matrices
\[
 B_p\in M_{m,n}(\R_+),
 \qquad
 C_p\in M_{n,m}(\R_+),
 \qquad 0\le p\le p_0,
\]
such that:
\begin{enumerate}[label=\textup{(\roman*)}]
\item $B_p$ and $C_p$ are entrywise dominated by the matrices obtained
      from all terms in the kernels of $T$ and $S$ on the sets $D_i$
      and $C_j$ by the construction in \eqref{eq:qir}--\eqref{eq:rrl};
\item there exist $w\in(0,\infty)^m$ and $\delta>0$ such that
      \[
        B_0C_0w=(1+\delta)w;
      \]
\item for all sufficiently small $p>0$,
      \[
        B_pC_pw\ge
        \left(1+\frac{\delta}{2}\right)w.
      \]
\end{enumerate}
\end{theorem}

\begin{proof}
Since $m(\Omega_T^{(2)})>0$ and
\[
 \left\{
 s:
 \left|\{j\ge1:a_j(s)\ne0\}\right|\ge2
 \right\}
 =
 \bigcup_{j_1<j_2}
 \{s:a_{j_1}(s)a_{j_2}(s)\ne0\},
\]
there exist fixed indices $j_1\ne j_2$ for which the set on the
right has positive measure. By semifiniteness, we may choose a
set of finite measure $B\subset\Omega_F$ of positive measure such that
\begin{equation}
 a_{j_1}(s)a_{j_2}(s)\ne0
 \qquad\text{for almost every }s\in B.
 \label{eq:two-nonzero-terms}
\end{equation}
By \Cref{prop:atomic-kernel-estimates}, the corresponding points
$\sigma_{j_1}(s)$ and $\sigma_{j_2}(s)$ are distinct almost
everywhere on $B$.

For $i=1,2$, consider all finite alternating sequences of maps arising
from the kernels of $T$ and $S$ which begin with the $j_i$-th term of
the kernel of $T$ and whose first return to the $F$-side set $B$
occurs at the final step.  These sequences form a countable family.
Their domains cover $B$ up to a null set: otherwise their complement
would have positive measure, and \Cref{lem:return}, applied to that
complement, would produce one further first-return sequence, a
contradiction.

Fix exhaustions by sets of finite measure of $\Omega_E$ and $\Omega_F$.
Subdivide the domain of each sequence according to exhaustion sets
containing its finitely many intermediate points, and make the
resulting domains disjoint within each family. For $i=1,2$ this produces
pairwise disjoint measurable sets $G_{i,r}$,
$r\ge1$, each associated with one fixed finite sequence whose first return
occurs at its final step,
such that
\[
 B=\bigsqcup_{r\ge1}G_{i,r}
 \qquad\text{up to a null set}.
\]
Choose $0<\varepsilon<1/2$. Since $m(B)<\infty$, finite subfamilies
may be chosen so that
\begin{equation}
 \sum_{r=1}^{N_i}m(G_{i,r})
 >(1-\varepsilon)m(B),
 \qquad i=1,2.
 \label{eq:finite-cover}
\end{equation}
The sets selected for $i=1$ and $i=2$ need not be disjoint from one
another, since they correspond to the distinct initial indices
$j_1$ and $j_2$.

Index the finitely many selected sequences by $\alpha$. Let
$L_\alpha$ denote the length of the $\alpha$-th sequence,
$G_\alpha$ its initial domain, and let
\[
 x_{\alpha,h}:G_\alpha\longrightarrow\Omega_E
 \quad\text{or}\quad
 x_{\alpha,h}:G_\alpha\longrightarrow\Omega_F
\]
be the measurable map giving the point reached after $h$ steps. Write
$\varphi_{\alpha,h}$ for the map used at the $h$th step and
$c_{\alpha,h}$ for its coefficient. For $1\le h<L_\alpha$, define
\[
 \mu_{\alpha,h}
 =(x_{\alpha,h})_*(m|_{G_\alpha}),
 \qquad
 W_{\alpha,h}
 =\car\mu_{\alpha,h}.
\]
Each set $x_{\alpha,h}(G_\alpha)$ lies in one of the chosen exhaustion
sets, so every $W_{\alpha,h}$ has finite measure.
Lemma~\ref{lem:carrier-propagation} supplies the propagation from one
carrier to the next. Since the selected sequences return to
$B$ for the first time only at their final step, every intermediate
$F$-side carrier is disjoint from $B$ up to a null set.

Consider the finite measurable partition of the union of the selected
$E$-side carriers generated by these carriers, and denote its
non-null members by
\[
 C_1,\dots,C_n.
\]
Similarly, partition the union of the intermediate $F$-side carriers
and denote its non-null members by
\[
 D_2,\dots,D_m.
\]
Set
\[
 D_1=B.
\]
The sets $C_1,\dots,C_n$ are therefore pairwise disjoint, as are
$D_1,\dots,D_m$, and each has finite measure.

For a partition set $U\ne B$ occurring in one of the selected
carriers, define
\[
 d(U)
 =
 \min\{L_\alpha-h:
       U\subset W_{\alpha,h}\text{ up to a null set}\}.
\]
Choose an occurrence for which this minimum is attained, and use the
next map in the corresponding sequence. If $V$ ranges over the
partition sets on the opposite side, put
\[
 G_{U,V}
 =
 \{x\in U:\varphi_{\alpha,h}(x)\in V\}.
\]
By \Cref{lem:carrier-propagation}, these sets are pairwise disjoint
and cover $U$ up to a null set. Consequently,
\[
 \sum_V\frac{m(G_{U,V})}{m(U)}=1.
\]

Suppose that $m(G_{U,V})>0$ and $V\ne B$. Then $V$ meets the next
carrier in positive measure. Because $V$ belongs to the partition generated by the selected
carriers, it lies in that carrier up to a null set. Hence
\begin{equation}
 d(V)\le d(U)-1.
 \label{eq:distance-decrease}
\end{equation}
The strict decrease of $d$ brings every selected partition set other
than $B$ back to $B$ in finitely many steps.

Define the transitions leaving $B$ as follows. For every selected set
$G_{i,r}$, partition $G_{i,r}$ according to which $C_q$ contains
$\sigma_{j_i}(s)$. The total normalized weight of these transitions
is
\begin{equation}
 \Gamma
 =
 \frac1{m(B)}
 \left(
 \sum_{r=1}^{N_1}m(G_{1,r})
 +
 \sum_{r=1}^{N_2}m(G_{2,r})
 \right)
 >2(1-\varepsilon)>1.
 \label{eq:row-sum-at-B}
\end{equation}

Retain only those partition sets which are reachable from $B$ by the
selected transitions. If a retained set has a transition of positive
measure to another set, then the latter is also reachable from $B$;
hence no positive transition is lost. By
\eqref{eq:distance-decrease}, every retained set can in turn reach
$B$. Therefore the resulting finite bipartite directed graph is
strongly connected.

For every directed edge $e:U\to V$ with domain $G_e\subset U$, define
\[
 \gamma_e(0)=\frac{m(G_e)}{m(U)}.
\]
Let $B_0$ be the matrix obtained by summing these weights over the
selected $T$-transitions, and let $C_0$ be defined analogously from
the selected $S$-transitions. The corresponding bipartite adjacency
matrix is
\[
 \cA_0
 =
 \begin{pmatrix}
 0&B_0\\
 C_0&0
 \end{pmatrix}.
\]
Every row other than the row corresponding to $B$ has sum one,
whereas the row corresponding to $B$ has sum $\Gamma>1$. Therefore
\[
 \cA_0\one\ge\one,
\]
with strict inequality in the coordinate corresponding to $B$.

Strong connectivity makes $\cA_0$ irreducible. Let $\xi>0$ be a left Perron vector of $\cA_0$. Then
\[
 \rho(\cA_0)\,\xi^\top\one
 =
 \xi^\top\cA_0\one
 >
 \xi^\top\one,
\]
and hence
\[
 \rho(\cA_0)>1.
\]
Moreover,
\[
 \cA_0^2
 =
 \begin{pmatrix}
 B_0C_0&0\\
 0&C_0B_0
 \end{pmatrix}.
\]
The two-step directed graph on the $F$-side vertices is strongly
connected, so $B_0C_0$ is irreducible and
\[
 \rho(B_0C_0)
 =
 \rho(\cA_0)^2
 >1.
\]
Let $w>0$ be a Perron vector of $B_0C_0$ and set
\[
 \delta=\rho(B_0C_0)-1.
\]
Then
\[
 B_0C_0w=(1+\delta)w.
\]

Continuity at $p=0$ completes the construction. For an edge
$e:U\to V$, let $c_e$ denote the coefficient of the selected kernel
term defining that edge, and set
\[
 \gamma_e(p)
 =
 \frac1{m(U)}
 \int_{G_e}|c_e(x)|^p\,dx.
\]
Choose $p_0\in(0,1)$ as in
\Cref{prop:atomic-kernel-estimates}. Since only
finitely many source sets of finite measure and selected kernel terms
are involved, \Cref{prop:atomic-kernel-estimates} gives
\[
 |c_e|^{p_0}\in L_1(G_e)
\]
for every selected edge $e$. For $0<p\le p_0$,
\[
 |c_e|^p\le1+|c_e|^{p_0},
 \qquad
 |c_e|^p\longrightarrow1
 \quad\text{almost everywhere on }G_e.
\]
Dominated convergence gives
\[
 \gamma_e(p)\longrightarrow\gamma_e(0)
 \qquad (p\downarrow0).
\]
There are only finitely many edges, so
\[
 B_p\longrightarrow B_0,
 \qquad
 C_p\longrightarrow C_0
\]
entrywise. Hence
\[
 B_pC_pw
 \longrightarrow
 B_0C_0w
 =
 (1+\delta)w.
\]
Because $w>0$, for all sufficiently small $p>0$ we therefore have
\[
 B_pC_pw
 \ge
 \left(1+\frac{\delta}{2}\right)w.
\]

At a source set different from $B$ only one kernel term is retained,
and its domain is partitioned according to the target sets. At $B$
the domains belonging to each fixed initial index $j_i$ are pairwise
disjoint, and $j_1\ne j_2$. Hence no contribution is counted twice,
so $B_p$ and $C_p$ are entrywise dominated by the matrices formed from
all kernel terms, giving \textup{(i)}.
\end{proof}

\begin{proof}[Proof of \Cref{thm:at-most-one-atom}]
Suppose, to the contrary, that
\[
 m\left\{
 s:
 \left|\{j\ge1:a_j^T(s)\ne0\}\right|\ge2
 \right\}>0.
\]
By \Cref{thm:finite-matrix-construction}, we may choose $p>0$,
$w\in(0,\infty)^m$, and $\delta_0>0$ such that
\begin{equation}
 B_pC_pw\ge(1+\delta_0)w.
 \label{eq:matrix-expansion}
\end{equation}
Let $\widehat B_p$ and $\widehat C_p$ denote the matrices obtained
from all terms in the kernels of $T$ and $S=T^{-1}$, respectively, on
the sets of finite measure furnished by
\Cref{thm:finite-matrix-construction}. By part
\textup{(i)} of that theorem,
\[
 \widehat B_p\ge B_p,
 \qquad
 \widehat C_p\ge C_p
\]
entrywise. Set
\[
 D_p=B_pC_p.
\]

Choose
\[
 1<\lambda<1+\delta_0,
\qquad
 w_*=\min_{1\le i\le m}w_i,
\qquad
 W=\sum_{i=1}^m w_i,
\]
and then choose
\[
 0<\varepsilon_0
 <
 \frac{(1+\delta_0-\lambda)w_*}{W}.
\]
Apply \Cref{thm:kernel-composition} with $Q=T$ and $R=S$. There exists a
measure-preserving transformation $\tau_0$ such that
\[
 U:=TV_{\tau_0}S
\]
is a surjective isometry and the matrix $M$ associated with $U$ and
the sets $D_1,\dots,D_m$ satisfies
\[
 M
 \ge
 \widehat B_p\widehat C_p-\varepsilon_0\mathbf J
 \ge
 D_p-\varepsilon_0\mathbf J,
\]
where $\mathbf J$ is the $m\times m$ all-ones matrix. In view of
\eqref{eq:matrix-expansion},
\[
 \begin{aligned}
 Mw
 &\ge D_pw-\varepsilon_0\mathbf Jw\\
 &\ge (1+\delta_0)w-\varepsilon_0W\one\\
 &\ge \lambda w.
 \end{aligned}
\]
We have obtained
\begin{equation}
 Mw\ge\lambda w.
 \label{eq:subeigenvector}
\end{equation}

Fix $1<\mu<\lambda$ and choose
\[
 0<\varepsilon_1
 <
 \frac{(\lambda-\mu)w_*}{W}.
\]
Construct recursively surjective isometries $U_n$ whose local
transition matrices $M_n$ satisfy
\begin{equation}
 M_nw\ge\mu^n w.
 \label{eq:exponential-growth}
\end{equation}
For $n=1$, take $U_1=U$ and $M_1=M$. Since $\lambda>\mu$,
\eqref{eq:subeigenvector} gives
\[
 M_1w\ge\mu w.
\]

Assume that $U_n$ and $M_n$ have been constructed. Apply
\Cref{thm:kernel-composition} to $Q=U$ and $R=U_n$ on the same sets of
finite measure.
For the resulting measure-preserving transformation $\tau_n$, set
\[
 U_{n+1}=UV_{\tau_n}U_n.
\]
The corresponding matrix satisfies
\[
 M_{n+1}\ge MM_n-\varepsilon_1\mathbf J.
\]
Since all matrices involved are nonnegative, the induction hypothesis
and \eqref{eq:subeigenvector} yield
\[
 \begin{aligned}
 M_{n+1}w
 &\ge MM_nw-\varepsilon_1W\one\\
 &\ge \mu^nMw-\varepsilon_1W\one\\
 &\ge \lambda\mu^n w-\varepsilon_1W\one.
 \end{aligned}
\]
By the choice of $\varepsilon_1$,
\[
 \varepsilon_1W\one
 \le(\lambda-\mu)w
 \le(\lambda-\mu)\mu^n w,
\]
and hence
\[
 M_{n+1}w
 \ge
 \lambda\mu^n w-(\lambda-\mu)\mu^n w
 =
 \mu^{n+1}w.
\]
This completes the induction and proves \eqref{eq:exponential-growth}
for every $n\ge1$.

The uniform estimate \eqref{eq:uniform-ceiling}, however, provides a
constant
$K_p<\infty$, independent of $n$, such that
\[
 (M_n)_{ij}\le K_p
 \qquad
 \text{for all }n,i,j.
\]
For each fixed $i$ it follows that
\[
 \mu^n w_i
 \le
 (M_nw)_i
 =
 \sum_j(M_n)_{ij}w_j
 \le
 K_p\sum_jw_j
 =
 K_pW.
\]
Since $\mu>1$ and $w_i>0$, the left-hand side tends to infinity as
$n\to\infty$, whereas the right-hand side is independent of $n$.
This contradiction proves
\[
 m\left\{
 s:
 \left|\{j\ge1:a_j^T(s)\ne0\}\right|\ge2
 \right\}=0.
\]
\end{proof}

\end{document}